%% file: reconstruction.tex
\documentclass{amsart}
\usepackage{verbatim}
\usepackage[textsize=scriptsize]{todonotes}
\usepackage{longtable}
\usepackage{bbm}
\usepackage{amscd}
\usepackage{stmaryrd}

\usepackage{etoc}
\etocsettocstyle{\noindent\textbf{}\par}{}

\usepackage[margin=1in]{geometry}
\input{Preamble.tex}

\let\scr=\mathcal
\def\cd{\mathrm{cd}}
\def\vcd{\mathrm{vcd}}
\def\pvcd{\mathrm{pvcd}}
\def\pcd{\mathrm{pcd}}
\def\comp{\wedge}
\newcommand{\MGL}{\mathrm{MGL}}
\newcommand{\KGL}{\mathrm{KGL}}
\newcommand{\BU}{\mathrm{BU}}
\newcommand{\KU}{\mathrm{KU}}
\newcommand{\ku}{\mathrm{ku}}
\newcommand{\bu}{\mathrm{bu}}
\newcommand{\BGL}{\mathrm{BGL}}
\newcommand{\bgl}{\mathrm{bgl}}
\newcommand{\veff}{{\text{veff}}}
\newcommand{\SH}{\mathcal{SH}}
\newcommand{\Spc}{\mathcal{S}\mathrm{pc}}
\newcommand{\et}{{\acute{e}t}}
\newcommand{\ret}{{r\acute{e}t}}
\newcommand{\wequi}{\simeq}
\newcommand{\1}{\mathbbm{1}}
\def\map{\mathrm{map}}
\DeclareRobustCommand{\ul}{\underline}
\newcommand{\heart}{\heartsuit}
\newcommand{\Gm}{{\mathbb{G}_m}}
\newcommand{\Gmp}[1]{{\mathbb{G}_m^{\wedge #1}}}
\newcommand{\cell}{\text{cell}}
\def\PSh{\mathcal{P}}
\newcommand{\Shv}{\mathcal{S}\mathrm{hv}}
\def\ph{\mathord-}
\def\op{\mathrm{op}}
\newcommand{\Ab}{\mathrm{Ab}}
\def\Cat{\mathcal{C}\mathrm{at}{}}

\def\CAlg{\mathrm{CAlg}}
\def\adj{\rightleftarrows}
\def\Sm{{\mathcal{S}\mathrm{m}}}
\def\Sch{{\mathcal{S}\mathrm{ch}}}

\def\pFin{{\Fin}^{\mathrm{pro}}}
\def\FEt{\mathrm{FEt}{}}
\newcommand{\proet}{{pro{\acute{e}t}}}
\newcommand{\Mod}{\mathrm{Mod}}
\newcommand{\RR}{\mathbb{R}}
\def\P{\mathbb P}
\def\CMon{\mathrm{CMon}}
\def\Mon{\mathrm{Mon}}
\def\Pro{\mathrm{Pro}}
\def\Ind{\mathrm{Ind}}
\def\Nis{\mathrm{Nis}}
\def\Zar{\mathrm{Zar}}
\DeclareMathOperator{\CB}{\mathrm{CB}}
\DeclareMathOperator{\id}{\mathrm{id}}
\def\Et{\mathcal{E}\mathrm{t}}
\def\mot{\mathrm{mot}}
\def\CQL{\mathrm{CQL}}
\def\bCQL{\mathrm{bCQL}}
\newcommand{\lra}[1]{\langle #1 \rangle}
\newcommand{\fpsr}[1]{\llbracket #1 \rrbracket}
\newcommand{\FFree}{\mathrm{FFree}}
\newcommand{\Span}{\mathrm{Span}}
\def\h{\mathrm h}
\newcommand{\limone}{\mathrm{lim}^1}
\newcommand{\Prf}{\mathrm{Prf}}

\newcommand{\ShvSp}{{\mathcal{S}\mathrm{p}}}
\newcommand{\GalSp}{\mathtt{SH}}
\newcommand{\GenSp}{{\mathcal{S}\mathrm{p}}^{gen}}

\newcommand{\NB}[1]{\todo[color=gray!40]{#1}}
\newcommand{\tom}[1]{\todo[color=green!40]{#1}}
\newcommand{\TODO}[1]{\todo[color=red]{#1}}

\newtoggle{final}
\toggletrue{final}

\iftoggle{final} {
\renewcommand{\todo}[1]{}
\renewcommand{\NB}[1]{}
\renewcommand{\tom}[1]{}
\renewcommand{\TODO}[1]{}
}

\title{Motivic stable stems and Galois approximations of cellular motivic categories}
\date{\today}

\author{Tom Bachmann}
\address{Mathematischen Institut, JGU Mainz, Germany}
\email{tom.bachmann@zoho.com}

\author{Robert Burklund}
\address{Department of Mathematical Sciences, University of Copenhagen, Denmark}
\email{rb@math.ku.dk}

\author{Zhouli Xu}
\address{UCLA Department of Mathematics, Los Angeles, CA 90095-1555, USA}
\email{xuzhouli@ucla.edu}

\begin{document}
\begin{abstract}
We reconstruct (appropriately completed) categories of cellular motivic spectra over fields of small cohomological dimension in terms of only their absolute Galois groups.
As our main application, we determine the motivic stable stems (away from the characteristic) of almost all fields.
\end{abstract}

\maketitle

\setcounter{tocdepth}{1}
\tableofcontents
\setcounter{tocdepth}{2}

\section{Introduction}
\localtableofcontents

\subsection{Invitation: the motivic stable stems}
Let $k$ be a field.
Recall (e.g. from \cite[\S4.1]{norms}) the motivic stable category $\SH(k)$.
It contains a bigraded family of spheres $\Sigma^{p,q} \1$ and consequently motivic spectra $E \in \SH(k)$ naturally have bigraded homotopy groups $\pi_{p,q}(E)$.
Accepting the significance of $\SH(k)$, an obvious task is to determine $\pi_{**}(\1)$, the \emph{motivic stable stems}.

A first step in this direction was completed by F. Morel \cite{morel2004motivic-pi0} by determining the lowest stem: \[ \pi_{p,q}(\1)(k) = \begin{cases} 0 & p<q \\ K_{-p}^{MW}(k) & p=q \end{cases}. \]
Here $K_*^{MW}(k)$ denotes the \emph{Milnor--Witt $K$-groups}, defined using generators and relations in terms of the arithmetic of $k$ \cite[\S3]{A1-alg-top}.

Guided by experience from classical homotopy theory, we split up this problem using notions of completion and localization at various elements in the motivic stable stems.
We can use the motivic Hopf element $\eta \in \pi_{1,1}(\1)$ to split attention across $\pi_{**}(\1_\eta^\comp)$ and $\pi_{**}(\1[\eta^{-1}])$.
Since the latter is dealt with in \cite{bachmann-eta}, we shall concentrate on the former.
Next we can employ the classical primes to split the problem across $\pi_{**}(\1_{\eta,\ell}^\comp)$ for all primes $\ell$, and the rationalization $\pi_{**}(\1_\eta^\comp) \otimes \QQ$.
Since the latter is essentially the same as $K_*(k) \otimes \QQ$ (rational algebraic $K$-theory of $k$) \cite[Theorems 16.1.4, 16.2.13, 16.1.3 and Corollary 14.2.14]{triangulated-mixed-motives}, we shall again concentrate on the former.

In some sense one knows a complete answer provided that $k=\bar k$ is algebraically closed (and of characteristic $\ne \ell$).
Indeed in this case there is an isomorphism \[ \pi_{**}(\1_{\ell,\eta}^\comp)(\bar k) \wequi (\pi_{**}^{syn})_\ell^\comp, \] where $\pi_{**}^{syn}$ denotes the \emph{synthetic stable stems} (i.e. the bigraded homotopy groups of the unit in synthetic spectra \cite{Pstragowski}).
In light of this, the following result determines the motivic stable stems for a very large class of fields.

\begin{thm}[see Theorem \ref{thm:tensor-product-formula-easy}] \label{thm:intro-intro}
Let $k$ be a field of characteristic $\ne \ell$ containing all $\ell^n$-th roots of unity for all $n$.
There is an isomorphism \[ \pi_{**}(\1_{\ell,\eta}^\comp)(k) \wequi (\pi_{**}^{syn} \otimes K_*^{MW}(k))_\ell^\comp. \]
\end{thm}
\begin{rmk}
Let us elaborate on the tensor product appearing in the above theorem.
It is customary to split the bigraded motivic stems into a singly graded sequence of singly graded objects of the form $\pi_n(\1)_* := \pi_{n,n-*}(\1)$.
The tensor product is in the category of singly graded modules over $\pi_0^{syn}$.
If $\ell \ne 2$ one has $\pi_0^{syn} = \Z_\ell$ concentrated in degree $0$.
If $\ell = 2$ then $\pi_0^{syn} = \Z_2[\eta]/2\eta$, where $\eta$ has degree $-1$.
Note that this is the same as $K_*^{MW}(\CC)_\ell^\comp$, as needed.
Note also that $K_*^{MW}(k)_\ell^\comp$ is a module over $\pi_0^{syn}$ in an obvious way, for any $k$ and $\ell$ (provided, if $\ell=2$, $k$ contains a square root of -$1$), so the tensor product makes sense.
\end{rmk}

In particular the Theorem asserts that there is a ring map $\pi_{**}^{syn} \to \pi_{**}(\1_{\ell,\eta}^\comp)(k)$.
As we will see presently, this is actually its most difficult part.

\subsection{Sketch proof of a special case} \label{subsec:intro-sketch}
Let $k$ be a field, containing an algebraically closed subfield $\bar k_0 \subset k$.
We shall determine $\pi_{**}(\1_{\ell,\eta}^\comp)(k)$ in terms of $K_*^{MW}(k)$ and $\pi_{**}(\1_{\ell,\eta}^\comp)(\bar k_0)$, by studying the motivic Adams--Novikov spectral sequence.
Throughout we assume that $k$ has characteristic different from $\ell$.

The slice spectral sequence for $\ell$-adic algebraic cobordism takes the form \[ \pi_{**}(H\Z_\ell^\comp) \otimes L_* \Rightarrow \pi_{**}(\MGL_\ell^\comp), \] where $L_* = \pi_{2*} \MU$ is the Lazard ring \cite[(8.6)]{hoyois2015algebraic}.
By the resolution of the Bloch--Kato conjecture \cite{haesemeyer2019norm} we have $\pi_{**}(H\Z_\ell^\comp) \wequi K_*^M(k)_\ell^\comp[\tau]$.
Here $\tau$ is a Bott element, which in fact already exists in $\pi_{0,-1}(\1_\ell^\comp)(\bar k_0)$ \cite[Proposition 4.7 and Remark 2.5]{bachmann-bott}.
Since $K_*^M$ lifts to $\pi_{**}\MGL$, one sees that the spectral sequence must collapse and we obtain \[ \pi_{**}(\MGL_\ell^\comp)(k) \wequi K_*^M(k)_\ell^\comp[\tau] \otimes L_* \]

We are now ready to look at the $\ell$-adic motivic Adams--Novikov spectral sequence, which takes the form \[E_1(k) := \pi_{**}(\MGL_\ell^\comp)(k) \otimes_{L_*} \pi_{2*} \MU^{\otimes *} \Rightarrow \pi_{**}(\1_{\ell,\eta}^\comp)(k). \]
Observe now that $E_1(k)$ is a module both over $E_1(\bar k_0)$ and over $K_*^M(k)_\ell^\comp$.
It follows from the Bloch--Kato conjecture that $K_*^M(k)_\ell^\comp$ is a free, hence flat, $\ell$-complete module (see Theorem \ref{thm:tensor-product-formula-easy}(1)).
For simplicity, also assume that $\ell \ne 2$.
Since all of $K_*^M(k)_\ell^\comp$ lifts to the sphere, it consists of permanent cycles, and hence (using that $K_*^M(k)_\ell^\comp$ is flat) we obtain a morphism of spectral sequences \[ (E_*(\bar k_0) \otimes K_*^M(k))_\ell^\comp \to E_*(k). \]
This is an isomorphism on the $E_1$-page, hence the $E_\infty$-page.
We deduce that \[ \pi_{**}(\1_{\ell,\eta}^\comp)(k) \wequi (\pi_{**}(\1_{\ell,\eta}^\comp)(\bar k_0) \otimes K_*^M(k))_\ell^\comp. \]
This is Theorem \ref{thm:intro-intro}, at least under out additional assumptions---and assuming that we know $\pi_{**}(\1_{\ell,\eta}^\comp)(\bar k_0) \wequi \pi_{**}^{syn}$.
For this, see e.g. \cite[\S6]{GIKR}.

\subsection{Galois approximation}
It will serve the recall in some more detail the isomorphism $\pi_{**}(\1_{\ell,\eta}^\comp)(\bar k_0) \wequi \pi_{**}^{syn}$.
We know that $\1_{\ell,\eta}^\comp$ is obtained as the ($\ell$-adic) totalization of cobar construction on $\1 \to \MGL$ (see e.g. \cite[Theorem 2.3]{bachmann-topmod}): \[ \1_{\ell,\eta}^\comp \wequi \Tot \CB^\bullet(\MGL)_\ell^\comp. \]
(Here $\CB^n(\MGL) := \MGL^{\otimes n+1}$.)
To determine the motivic stable stems, it will thus suffice to describe the cosimplicial spectra \[ \Gamma(\bar k_0, *)^\bullet := \imap(\Sigma^{2*, *} \1, \CB^\bullet(\MGL)_\ell^\comp)(\bar k_0). \]
A key observation of \cite{GIKR} is that these spectra coincide, up to some strategic truncation, just with the usual bar construction of $\MU$.
The totalization of this truncated bar construction is essentially by definition the synthetic sphere spectrum.
Motivically, the origin of these truncations is the Bloch--Kato conjecture: the fact that motivic cohomology is a truncation of étale cohomology.
This also immediately suggests a generalization: for fields $k$ which are not separably closed, instead of attempting to construct $\Gamma(k,*)^\bullet$ using just topology, we should try to do it using one further piece of information: the absolute Galois group $Gal(k)$ of $k$, together with its cyclotomic character.\footnote{The reason for requiring the cyclotomic character is that motivic cohomology is truncated étale cohomology with coefficients in roots of unity, and in order to access these cohomology groups in terms of profinite group cohomology, one needs the cyclotomic character.}

Write $\scr G_{/\Z_\ell^\times}$ for the category of profinite groups together with a morphism to $\Z_\ell^\times$.
If $k$ is a field of characteristic $\ne \ell$, then $Gal(k)$ together with its ($\ell$-adic) cyclotomic character defines an object $G(k) \in \scr G_{/\Z_\ell^\times}$.
By a \emph{Galois approximation} we mean a functor \[ F(\ph, *)^\bullet: \scr G_{/\Z_\ell^\times}^\op \times \Z \times \Delta \to \ShvSp \] together with a natural transformation \[ \alpha_{(\ph)}: \Gamma(\ph, *)^\bullet \to F(G(\ph), *)^\bullet. \]
\textbf{The main object of this article} is to construct such a Galois approximation with the property that $\alpha_k$ is an equivalence for appropriate fields $k$, e.g. those of small cohomological dimension (say $\le 2$).

Before delving into how this is done, let us elaborate on the use of such a construction for computing the motivic stable stems.
Write $\QQ(\mu_{\ell^\infty})$ for the field obtained by adjoining all the $\ell$-power roots of unity, and let $k$ be a field (of characteristic zero, say) containing all $\ell$-power roots of unity, that is, $\QQ(\mu_{\ell^\infty}) \subset k$.
Write $e \in \scr G_{/\Z_\ell^\times}$ for the trivial group, mapping trivially to $\Z_\ell^\times$.
For $G \in \scr G_{/\Z_\ell^\times}$, there exists a (necessarily unique) morphism $G \to e$ if and only if the cyclotomic character of $G$ is trivial.
In particular there is a unique morphism $p: G(\QQ(\mu_{\ell^\infty})) \to e$.
Now we contemplate the diagram \[ F(e, *)^\bullet \xrightarrow{p^*} F(G(\QQ(\mu_{\ell^\infty})), *)^\bullet \xleftarrow{\alpha} \Gamma(\QQ(\mu_{\ell^\infty}), *)^\bullet \to \Gamma(k, *)^\bullet. \]
Since $\QQ(\mu_{\ell^\infty})$ has cohomological dimension $1$\NB{right?}, the left-pointing arrow $\alpha$ above is an equivalence.
Since $e = G(\bar k)$ (for $\bar k$ an algebraically closed field) and $\alpha_{\bar k}$ is an equivalence, we see that $\Tot F(e,*)^\bullet$ is the synthetic sphere.
Thus the above construction provides the sought-after morphism $\pi_{**}^{syn} \to \pi_{**}(\1_{\eta,\ell}^\comp)(k)$.

\subsection{Quillen--Lichtenbaum properties and étale algebraic cobordism}
Having illustrated the utility of Galois approximations, we now explain our strategy for constructing them.
An illustrative special case of the problem is to attempt to reconstruct the spectra \[ \map(\Sigma^{2*, *} \1, \MGL_\ell^\comp). \]
We do this in two steps:
\begin{enumerate}
\item First, we study the morphism $\MGL_\ell^\comp \to L_\et(\MGL)_\ell^\comp$, and in particular its effect on $\map(\Sigma^{2*, *} \1, \ph)$.
  We seek to prove that this morphism is essentially just a truncation, i.e. that a relation approximately of the form \[ \map(\Sigma^{2*, *} \1, \MGL_\ell^\comp) \wequi \tau_{\ge 2*} \map(\Sigma^{2*, *} \1, L_\et(\MGL)_\ell^\comp) \] holds.
  Such a result is called a \emph{Chow--Quillen--Lichtenbaum property}.
\item Second, we need to study $\map(\Sigma^{2*, *} \1, L_\et(\MGL)_\ell^\comp)$.
  This seems approachable because of the rigidity results of \cite{bachmann-SHet}: \[ L_\et(\MGL)_\ell^\comp \in \SH_\et(k)_\ell^\comp \wequi \ShvSp(k_\et)_\ell^\comp. \]
  In other words, $L_\et(\MGL)_\ell^\comp$ is essentially an ordinary spectrum with a continuous action of $Gal(k)$.
  In fact one knows that the underlying spectrum is obtained by setting $k=\bar k$, that is, is just given by $\MU_\ell^\comp$.
  \textbf{One of the main technical aspects of this article} is to show that $L_\et(\MGL)_\ell^\comp$ is given by $\MU_\ell^\comp$ together with a certain continuous action of $\Z_\ell^\times$, with the action of $Gal(k)$ via the cyclotomic character.\footnote{This statement has to be adapted slightly for $\ell=2$.}
\end{enumerate}

\subsection{Reverse-linear overview}
In keeping with our motivation from applications, let us briefly review the contents of each section, beginning at the end.
Each section has its own introduction where more details can be found, as well as its own table of contents.

In \S\ref{sec:motivic-stable-stems} we determine, for a large class of fields $k$, the $(\ell,\eta$)-completed motivic stable stems.
The answer is essentially a functor of the Milnor--Witt $K$-theory of $k$ alone.
The case where $k$ contains all $\ell$-power roots of unity is much easier and is dealt with in Theorem \ref{thm:tensor-product-formula-easy}.
A more general but also much more complicated result is Theorem \ref{thm:motivic-stems-hard}.

In order to prove these results, we need to construct a large number of elements in the motivic stable stems.
It turns out to be enough to do this for appropriate fields $k$ of small cohomological dimension.
We do this by, in this case, reconstructing the entire motivic Adams--Novikov spectral sequence just from the Galois group of $k$.
We call this a \emph{Galois approximation}; this is the subject of \S\ref{sec:galois-approx}.

All the previous sections contain ingredients for our Galois approximation.
In \S\ref{sec:CQL} we define and study the Chow--Quillen--Lichtenbaum properties for $\MGL$ and other spectra, in other words, we study the difference between algebraic cobordism and its étale version.
In \S\ref{sec:adams-summands} we recall the construction of Adams summands of localizations $\MGL$.
To use the Quillen--Lichtenbaum result for $\MGL$, we need to study its étale localization.
This is done in \S\ref{sec:etale-MGL}.
For this we use the description of $\MGL$ as a Thom spectrum for the rank zero summand of $K$-theory, thus we must first study étale $K$-theory as a sheaf of spaces.
This is the content of \S\ref{sec:etale-K}.

\subsection{Notation and conventions}
Parts of this work are very technical and require us to treat many objects which are intricately related and yet different.
In an attempt not to overburden notation, we often denote similar objects by the same letter, even when their are technically distinct.
We introduce distinguishing notation only when we feel clarity would be affected otherwise.
We explain some of these choices in what follows.

\subsubsection*{Pro-sites}
We denote by $\nu^*$ various functors arising from mapping a site to a related site of pro-objects (typically solidifications), as in $\nu^*: \Shv(S_\et) \to \Shv(S_\proet)$.

\subsubsection*{Cyclotomic characters}
For a scheme $S$, we denote by $S_\et$ its étale topos.
In various situations this comes with a cyclotomic character $\pi: S_\et \to B_\et G$.
We denote then by $\pi^*$ various functors arising from these morphisms, i.e., from the fact that in appropriate situations, finite $G$-sets define finite étale $S$-schemes.
We use the same notation also for the pro-sites, as in $\pi^*: \Shv_\proet(BG) \to \Shv_\proet(S)$.
See \S\ref{subsec:large-pro-etale} and \S\ref{subsec:pro-G-site} for more details.

\subsubsection*{Categories of motivic sheaves}
Much of this work is concerned with the categories of motivic spectra $\SH(S)$ and their avatars.
These come in the following types:
\begin{itemize}
\item Categories obtained by stabilizing with respect to $\P^1$.
  These we denote by $\SH(S)$, $\SH_\et(S)$, $\SH_\proet(S)$, and so on.
\item Categories which are (essentially) sheaves of spectra, so obtained by stabilizing an $\infty$-topos with respect to $S^1$.
  These we denote by $\ShvSp(S_\et)$, and so on.
\item Categories obtained as modules over appropriate sphere-graded ring spectra, i.e. Galois approximations.
  These we denote by $\GalSp(\Pi_{\le 1} S_\et)$, and so on.
  See \S\ref{subsec:cellular-reconstruction} for more details.
\item Occasionally we have to use categories of genuine $G$-spectra, for a finite group $G$.
  We denote these by $\GenSp(BG)$.
\end{itemize}

\subsubsection*{Miscellaneous}
All $\infty$-topoi are hypercompleted unless explicitly stated otherwise.
All schemes are qcqs.

\subsubsection*{Table of some commonly used notation}
\begin{longtable}{lll}
$\Shv(B_\et G)$, $\ShvSp(B_\et G)$ & (spectral) sheaves on discrete $G$-sets & \S\ref{subsub:SHet-BG} \\
$\Shv(B_\proet G)$, $\ShvSp(B_\proet G)$ & & \S\ref{subsec:pro-G-site} \\
$\Spc(S)$, $\SH^{S^1}(S)$, $\SH(S)$ & motivic categories & \\
$\Spc_\et(S)$, $\SH_\et(S)$ & étale versions & \\
$\SH_\proet(S)$ & & \S\ref{subsec:large-pro-etale} \\
$\Shv_\proet(\Sch_S)$, $\SH_\proet(\Sch_S)$ & & \S\ref{subsec:large-pro-etale} \\
$\GalSp(X), \GalSp(X)^{AT}$ & Galois approximations & \S\ref{subsec:cellular-reconstruction} \\
$G_\ell$ & & \S\ref{subsec:cyclo-char} \\
$\bu_\ell, \BU_\ell, \widetilde\BU_\ell$ & & Definition \ref{dfn:bul} \\
$J_{\proet,\ell}^\comp, j_{\proet,\ell}^\comp, j_{\et,\ell}^\comp$ & (pro-)étale $j$-homomorphisms & \S\ref{subsec:etale-mot-j} \\
$j_{G_\ell,\ell}, \MU_\ell$ & & \S\ref{subsec:equivariant-j} \\
$J_{G_\ell,\ell}$ & & \eqref{eq:JG} \\
$\tau_{\ge n}^\ell$ & pro-$\ell$ cover & \S\ref{subsubsec:pro-cover} \\
$L_\ell$ & categorical $\ell$-completion & \S\ref{subsub:general-l-comp} \\
$\scr C_\ell^\comp$ & subcategory of $\ell$-complete objects & \\
$\pi_*^s$ & classical stable stems & \\
$C(X, Y)$ & continuous maps of topological spaces & \\
$\CB(A) \in \Fun(\Delta, \scr C)$ & cobar construction on monoid $A \in \scr C$ & \\
$\Spc$ & $\infty$-category of spaces \\
$\ShvSp$ & $\infty$-category of spectra \\
$\ShvSp^\Ss$ & sphere-graded spectra & Example \ref{ex:sphere-graded-spectra} \\
$\GenSp(BG)$ & genuine $G$-spectra \\
$\PSh_\ShvSp(\scr C)$ & spectral presheaves on $\scr C$ \\
$\Pi_{\le 1}(S_\et)$ & étale fundamental groupoid & Definition \ref{def:Prf} \\
$\pcd_\ell(S)$ & punctual $\ell$-étale dimension & $\sup_{s\in S} \cd_\ell(s)$ \\
$\pvcd_\ell(S)$ & punctual virtual $\ell$-étale dimension & $\sup_{s\in S} \vcd_\ell(s)$ \\
\end{longtable}

\subsection{Acknowledgements}
Throughout the many years that we have worked on this project, we have benefited from discussions with many persons.
Our special thanks go to Mike Hopkins, Marc Hoyois, Akhil Mathew and Fabien Morel.\NB{more?}

\section{Preliminaries}
\localtableofcontents

\bigskip
In this section we collect a variety of preliminary results.
It is probably best skipped on first reading.

\subsection{Graded objects} \label{subsec:graded}
For a symmetric monoidal $\infty$-category $\scr C$ and an $\scr E_\infty$-monoid $A$, denote by $\scr C^A = \Fun(A, \scr C)$ the symmetric monoidal (via Day convolution) category of $A$-graded objects in $\scr C$.
Then for a symmetric monoidal $\infty$-category $\scr D$, a symmetric monoidal functor $\scr C^{A} \to \scr D$ is the same as a symmetric monoidal functor $\scr C \to \scr D$ together with a morphism of $\scr E_\infty$-monoids $A \to \scr D^\wequi$ (see e.g. \cite[Lemma 6.23]{bachmann-linearity}).
\begin{exm} \label{ex:sphere-graded-spectra}
Let $A = \Ss$ be the sphere spectrum, i.e. the free grouplike $\scr E_\infty$-monoid on a point.
Then a symmetric monoidal functor $\scr C^{\Ss} \to \scr D$ is the same as a symmetric monoidal functor $\scr C \to \scr D$ together with a choice of invertible object in $\scr D$.
\end{exm}

\begin{lem} \label{lem:graded-t}
Let $\scr C$ be presentably symmetric monoidal and stable with a $t$-structure.
Suppose given any map $f: A \to \Z$.
Then \[ \{(X_s) \in \scr C^A \mid X_a \in \scr C_{\ge f(a)} \} \] defines the non-negative part of a $t$-structure on $\scr C^A$, with truncation \[ \tau^A_{\ge 0}(X)_a \wequi \tau_{\ge f(a)}(X_a). \]
If $f(a) + f(b) \ge f(a+b)$ for all $a,b \in A$, $f(0) = 0$ and the $t$-structure on $\scr C$ is compatible with a symmetric monoidal structure, then so is the one on $\scr C^A$.
\end{lem}
\begin{proof}
Set $A_i = f^{-1}(\{i\})$.
Ignoring the symmetric monoidal structure, we have \begin{equation} \label{eq:CA-decomp} \scr C^A \wequi \prod_{i \in \Z} \Fun(A_i, \scr C). \end{equation}
Thus for the first part we may assume that $f$ is constant, in which case the claim is clear; we also see from this description that the negative part can be described similarly to the positive part.

Now we establish the second part of the lemma.
Note that $\scr C^A$ is simultaneously tensored over $\Spc^A$ and $\scr C$ (receiving symmetric monoidal functors from both).
For $a \in A$ and $c \in \scr C$, write $a \otimes c \in \scr C^A$ for the resulting object, where we view $a \in \Spc^A$ via the Yoneda embedding.
Then we have $\Map(c \otimes a, X) \wequi \Map(c, X_a)$.
From this we deduce that objects of the form $\Sigma^{f(a)} c \otimes a$ for $c \in \scr C_{\ge 0}$ and $a \in A$ generate $(\scr C^A)_{\ge 0}$ under colimits and extensions.
(Indeed we have seen that an object is in the negative part if and only if it is right orthogonal to our proposed generators.)
However we have $(c \otimes a) \otimes (c' \otimes a') \wequi (c+c') \otimes (a\otimes a')$ (where the middle $\otimes$ refer to the symmetric monoidal structure on $\scr C^A$).
Consequently the generators of our $t$-structure are closed under binary tensor products.
Since $f(0) \le 0$\footnote{Note that $f(0)+f(0) \ge f(0)$ means $f(0) \ge 0$, so asking in addition that $f(0) \le 0$ just means that $f(0)=0$.} we see that also $\1 \in (\scr C^A)_{\ge 0}$.
It follows that $(\scr C^A)_{\ge 0}$ is closed under finite tensor products, as needed.
\end{proof}

\subsection{Pro-objects} \label{subsec:pro}
Given an $\infty$-category $\scr C$, the category of \emph{pro-objects in $\scr C$} is \cite[\S3.1, Example 3.1.2]{lurie2011derived} \[ \Pro(\scr C) \wequi \Ind(\scr C^\op)^\op. \]
Note that for pro-objects $\{X_i\}, \{Y_j\}$ we have \begin{equation} \label{eq:pro-maps} \Map_{\Pro(\scr C)}(\{X_i\},\{Y_j\}) \wequi \lim_j \colim_i \Map_{\scr C}(X_i, Y_j): \end{equation} essentially by definition $\{X_i\} \wequi \lim_i X_i \in \Pro(\scr C)$ and similarly for $\{Y_i\}$, and each $Y_i$ is cocompact in $\Pro(\scr C)$ by construction.
We denote by $c: \scr C \to \Pro(\scr C)$ the ``constant pro-object'' functor (dual Yoneda embedding).
If $\scr D$ is an $\infty$-category with small cofiltered limits, then functors from $\Pro(\scr C)$ to $\scr D$ preserving cofiltered limits are the same as functors from $\scr C$ to $\scr D$ \cite[Proposition 5.3.5.10]{HTT}.
In particular, if $\scr C$ itself has small cofiltered limits, then the identity $\scr C \to \scr C$ extends to a functor\NB{``materialization''} \[ m: \Pro(\scr C) \to \scr C, \{X_i\} \mapsto \lim_i X_i, \] yielding an adjunction \[ c: \scr C \adj \Pro(\scr C): m. \]

\subsubsection*{Monoidal structure}
If $\scr C$ carries a symmetric monoidal structure, then so does $\Pro(\scr C)$ \cite[Remark 2.4.2.7, Proposition 4.8.1.10]{HA}.
It is given by $\{X_i\} \otimes \{Y_j\} \wequi \{X_i \otimes Y_j\}$.
Thus the functor $c$ is symmetric monoidal, and its right adjoint $m$ is lax symmetric monoidal.

\subsubsection*{Pro-$\ell$-completion}
If $\scr C$ is stable then so is $\Pro(\scr C)$ \cite[Lemma 2.5]{MR4012551}.
There is thus a notion of internal $\ell$-completion; for constant objects this recovers the pro-$\ell$-completion \[ c(X)_\ell^\comp \wequi \{X/\ell^n\}_n. \]
If $\scr C$ is also symmetric monoidal, then since the tensor product in $\Pro(\scr C)$ preserves cofiltered limits by construction, internal $\ell$-completion is a smashing localization.

\subsubsection*{$t$-structure}
\begin{lem} \label{lem:pro-t}
If $\scr C$ has a $t$-structure, then $\Pro(\scr C_{\ge 0}), \Pro(\scr C_{\le 0})$ define a $t$-structure on $\Pro(\scr C)$, with truncations computed levelwise.
If $\scr C$ in addition has a symmetric monoidal structure compatible with the $t$-structure, then the $t$-structure and symmetric monoidal structure on $\Pro(\scr C)$ are also compatible.
\end{lem}
\begin{proof}
It is clear that $\Pro(\scr C_{\ge 0}), \Pro(\scr C_{\le 0})$ are closed under appropriate shifts, and that the functorial truncations in $\scr C$ define the required truncation cofiber sequences for $\Pro(\scr C)$.
The required orthogonality is immediate from \eqref{eq:pro-maps}.
The compatibility with the monoidal structure is clear by construction.
\end{proof}

\subsubsection*{Pro-$\ell$ covers} \label{subsubsec:pro-cover}
Consider the functor \[ \tau_{\ge 0}^\ell: \scr C \xrightarrow{c} \Pro(\scr C) \xrightarrow{(\ph)_\ell^\comp} \Pro(\scr C) \xrightarrow{\tau_{\ge 0}} \Pro(\scr C) \xrightarrow{m} \scr C, \] explicitly given by \[ \tau^{\ell}_{\ge 0}(E) \wequi \lim_i \tau_{\ge 0}(E/\ell^i). \]
Since all functors in the composite expression are lax symmetric monoidal, so is $\tau^{\ell}_{\ge 0}$.
Replacing the functor $\tau_{\ge 0}$ in the above definition by the lax symmetric monoidal transformation $\tau_{\ge 0} \to \id$ we obtain a lax symmetric monoidal natural transformation \[ \tau_{\ge 0}^{\ell} \to (\ph)_\ell^\comp. \]

\subsection{Unstable $\ell$-completion} \label{subsec:unstable-ell-comp}
\subsubsection{Spaces}
Recall that a map $f: X \to Y \in \Spc$ is called an $\F_\ell$-equivalence if $\Sigma^\infty_+ f \otimes H\F_\ell \in \ShvSp$ is an equivalence; equivalently $(\Sigma^\infty_+ f)_\ell^\comp \in \ShvSp_\ell^\comp$ is an equivalence.
These are the weak equivalences in a Bousfield localization of $\Spc$ which we call \emph{unstable $\ell$-completion} and denote by $L_\ell$.\footnote{There are other notions of unstable $\ell$-completion, which we shall make no use of.}
More generally, if $\scr L$ is a set of primes, then $f$ is called an unstable $\scr L$-equivalence (or $\F_\scr{L}$-equivalence) if $\Sigma^\infty_+ f \otimes H\F_\ell$ is an equivalence for all $\ell \in \scr L$; we call the associated localization functor $L_{\scr L}$ the \emph{unstable $\scr L$-completion}.
Of course $\scr L_\ell = \scr L_{\{\ell\}}$.
Note that if $\scr L = \emptyset$ then all maps are $\F_\scr{L}$-equivalences, which is a bit pathological.
We shall hence always assume that $\scr L$ is \emph{non-empty}.

The following is well-known.
\begin{lem} \label{lemm:p-complete-coproducts}
The functor $L_\scr{L}: \Spc \to \Spc$ preserves coproducts.
\end{lem}
\begin{proof}
It is clear that $\F_\scr{L}$-equivalences are stable under coproducts; hence we need to show that $\scr{L}$-complete spaces are stable under coproducts.
Let $f: X \to Y$ be an $\F_\scr{L}$-equivalence.
Pick $\ell_0 \in \scr L$.
Note that $H_0(X, \F_{\ell_0}) \wequi \F_{\ell_0}\{\pi_0 X\}$ and so $H_0(f, \F_{\ell_0})$ being an isomorphism implies that $\pi_0(f)$ is an isomorphism.
We can thus write $X \wequi \coprod_{i \in I} X_i$ and $Y \wequi \coprod_{i \in I} Y_i$ such that each $X_i, Y_i$ is connected and $f(X_i) \subset Y_i$.
Then \[ H_*(f, \F_\ell) \wequi \bigoplus_{i \in I} H_*(f_i, \F_\ell) \] and so each $f_i$ is an $\F_\scr{L}$-equivalence.\NB{Isomorphism being stable under retraction.}
In other words, $\F_\scr{L}$-equivalences of connected spaces generate all $\F_\scr{L}$-equivalences under colimits (in fact coproducts).

Now let $\{Z_i \in \Spc\}_{i \in I}$ be $\scr{L}$-complete spaces.
We need to show that for any $\F_\scr{L}$-equivalence $f: X \to Y$ with $X, Y$ connected we have \[ f^*: \Map(Y, \coprod_i Z_i) \xrightarrow{\wequi} \Map(X, \coprod_i Z_i). \]
This is clear since maps out of connected spaces preserve coproducts and each $Z_i$ is $\ell$-complete.
\end{proof}

The extra generality of a set of primes is often superficial.
\begin{lem} \label{lem:L-splitting}
Let $X \in \Spc$ be nilpotent.
Then \[ L_{\scr L} X \wequi \prod_{\ell \in \scr L} L_\ell X. \]
\end{lem}
\begin{proof}
Note that $\ell$-complete objects are $\scr L$-complete.
It follows that the right hand side is $\scr L$-complete.
It also follows that $L_\ell L_{\scr L} X \wequi L_\ell X$, for any $\ell \in \scr L$.
Hence there is a natural map $L_{\scr L} X \to \prod_{\ell \in \scr L} L_\ell L_{\scr L} X \wequi \prod_{\ell \in \scr L} L_\ell X$ which we need to prove is an equivalence.
Equivalently we want that \[ L_\ell L_{\scr L} X \wequi L_\ell \prod_{\ell' \in \scr L} L_{\ell'} X, \] for any $\ell \in \scr L$.
Let $X' = \prod_{\ell' \in \scr L \setminus \ell} L_{\ell'} X$.
We thus need to prove that $H_*(X, \F_\ell) \wequi H_*(L_\ell X \times X', \F_\ell)$.
By the Künneth formula for this we want $H_*(X', \F_\ell) \wequi \F_\ell$, or equivalently $L_\ell X' \wequi *$.
Using the explicit formulas for $\ell$-completion \cite[Theorem 2.6]{barthel2019comparison}, for this it suffices to prove that multiplication by $\ell$ is an automorphism of $\pi_* X'$, or equivalently of $L_{\ell'} X$.
This is clear from the formulas.\NB{slicker replacement for last two sentences?}
\end{proof}

\subsubsection{General case} \label{subsub:general-l-comp}
\begin{dfn}
Given an $\infty$-category $\scr C$ and $X \in \scr C$, we say that $X$ is $\ell$-complete (respectively $\scr L$-complete) if for each $Y \in \scr C$ the space $\Map(Y, X)$ is $\ell$-complete (respectively $\scr L$-complete).
\end{dfn}
It is clear that this can be checked on a class of generators $Y$ of $\scr C$, and that $\ell$-complete (respectively $\scr L$-complete) objects are closed under limits.
If $\scr C$ is presentable these are the local objects of a Bousfield localization which we denote by $L_\ell$ (respectively $L_\scr{L}$) and call $\ell$-completion (respectively $\scr L$-completion).\NB{this clashes with the previous definition of $L_\ell$ only if $\scr C = \Ab$, in which case the current one is stupid anyways...}
For emphasis, we may sometimes call this notion \emph{categorical $\ell$-completion} (respectively \emph{categorical $\scr L$-completion}).
The maps inverted by this Bousfield localization are called \emph{categorical $\F_\ell$-equivalences}.

\begin{exm} \label{ex:completion-discrete}
By Lemma \ref{lemm:p-complete-coproducts}, discrete objects are always $\scr L$-complete.
It follows that $\tau_{\le 0} L_\scr{L} X \wequi \tau_{\le 0} X$, for any $X \in \scr C$.
\end{exm}

\begin{exm} \label{ex:SH-p-completion}
A spectrum $E \in \ShvSp$ is $\ell$-complete in the usual sense (i.e. $E \wequi \lim_n E/\ell^n$) if and only if all of its homotopy groups are derived $\ell$-complete, if and only if $\Omega^{\infty + n} E$ is a $\ell$-complete space for every $n \in \Z$ (use Lemma \ref{lemm:p-complete-coproducts} and e.g. \cite[Theorem 2.6]{barthel2019comparison}).
It follows that $E$ is categorically $\ell$-complete if and only if $E$ is $\ell$-complete in the usual sense.
\end{exm}

\begin{exm} \label{ex:p-completion-stable}
If $\scr C$ is a stable $\infty$-category and $E \in \scr C$, then $E \wequi \lim_n E/\ell^n$ if and only if for every $Y \in \scr C$ we have $\map(Y, E) \wequi \lim_n \map(Y, E)/\ell^n$.
Thus by Example \ref{ex:SH-p-completion} this holds if and only if $E$ is categorically $\ell$-complete.
\end{exm}

Categorical $\scr L$-completion behaves well under adjunctions.
\begin{prop} \label{prop:completion-permanence}
Let $L: \scr C \adj \scr D: R$ be an adjunction.
Then $L$ preserves categorical $\F_\scr{L}$-equivalences and $R$ preserves categorically $\scr{L}$-complete objects.
\end{prop}
\begin{proof}
If $f: X \to Y \in \scr C$ is an $\F_\scr{L}$-equivalence and $Z \in \scr D$ is $\scr{L}$-complete, then $\Map(Lf, Z) \wequi \Map(f, RZ)$ is an equivalence as soon is $RZ$ is $\scr{L}$-complete; i.e. it suffices to prove the second claim.
Thus we need to show that $\Map(X, RZ)$ is $\scr{L}$-complete; since this is the same as $\Map(LX, Z)$ this is true by assumption.
\end{proof}

\begin{exm} \label{ex:biadj-pres-ell-comp}
A functor with a left and a right adjoint preserves categorical $\scr{L}$-completions.
\end{exm}
\begin{exm} \label{ex:l-comp-presh}
Let $\scr C$ be a small $\infty$-category.
The functor $L_\scr{L}: \PSh(\scr C) \to \PSh(\scr C)$ is given by applying $\scr{L}$-completion sectionwise.
Indeed evaluation at any object $c \in \scr C$ admits left and right adjoints, so commutes with $L_\scr{L}$ by Example \ref{ex:biadj-pres-ell-comp}.
\end{exm}

\begin{rmk} \label{rmk:p-compl-products}
The Künneth theorem implies that $\scr{L}$-completion of spaces commutes with finite products.
It follows that categorical $\scr{L}$-completeness is a \emph{mode} \cite[Proposition 5.2.10]{carmeli2020ambidexterity}.
\end{rmk}

\subsubsection{$\Sigma$-presheaves}
Let $\scr C$ be a small category with finite coproducts.
\begin{prop} \label{prop:p-comp-PSigma}
The functor $L_\scr{L}: \PSh_\Sigma(\scr C) \to \PSh_\Sigma(\scr C)$ is obtained by applying $L_\scr{L}: \Spc \to \Spc$ sectionwise.
\end{prop}
\begin{proof}
The functor $L_\Sigma: \PSh(\scr C) \to \PSh_\Sigma(\scr C)$ preserves $\F_\scr{L}$-equivalences by Proposition \ref{prop:completion-permanence}.
Using Example \ref{ex:l-comp-presh} it hence suffices to show that if $F \in \PSh_\Sigma(\scr C) \subset \PSh(\scr C)$ and $F'$ is obtained by applying $L_\scr{L}$ sectionwise, then $F' \in \PSh_\Sigma(\scr C)$.
This follows from Remark \ref{rmk:p-compl-products}.
\end{proof}

\begin{dfn}
Suppose given $X \to Y \in \Spc$ with $Y$ $\scr{L}$-complete.
We say that \emph{the completion is fiberwise} if for every point $y \in Y$ the following diagram is a pullback square
\begin{equation*}
\begin{CD}
L_\scr{L}(X_y) @>>> L_\scr{L}(X) \\
@VVV              @VVV     \\
\{y\}       @>>> Y.
\end{CD}
\end{equation*}
\end{dfn}

\begin{prop} \label{prop:p-comp-slice} \NB{this result feels suboptimal (e.g. it also holds for usual presheaves)}
Let $\scr X \in \PSh_\Sigma(\scr C)$ and $F = (\scr F \to \scr X) \in \PSh_\Sigma(\scr C)_{/\scr X}$.
Write $L_\scr{L} F = (L_\scr{L}^{\scr X} \scr F \to \scr X)$ for the $\F_\scr{L}$-completion in $\PSh_\Sigma(\scr C)_{/\scr X}$.
\begin{enumerate}
\item Given $\alpha: c \to \scr X$ (with $c \in \scr C$) we have \[ (L_\scr{L}^{\scr X} \scr F)(c) \times_{\scr X(c)} \{\alpha\} \wequi L_\scr{L}(\scr F(c) \times_{\scr X(c)} \{\alpha\}). \]
\item Suppose $\scr X$ is $\F_\scr{L}$-local.
  Put $L_\scr{L}^0 F = (L_\scr{L} \scr F \to \scr X)$.
  The canonical map $F \to L_\scr{L}^0 F$ is an $\F_\scr{L}$-localization if and only if for every $\alpha: c \to \scr X$ (with $c \in \scr C$) the $\scr{L}$-completion of $\scr F(c) \to \scr X(c)$ is fiberwise.
\end{enumerate}
\end{prop}
\begin{proof}
(1) Recall that $\PSh_\Sigma(\scr C)_{/\scr X} \wequi \PSh_\Sigma(\scr C_{/\scr X})$.\NB{ref?}
Under this equivalence, evaluation at $(\alpha: c \to \scr X) \in \scr C_{/\scr X}$ is given by $F \mapsto \scr F(c) \times_{\scr X(c)} \{\alpha\}$.
The result thus follows from Proposition \ref{prop:p-comp-PSigma}.

(2) Note that $L_\scr{L}^0 \scr F$ is $\F_\scr{L}$-local.
Since $\PSh_\Sigma(\scr C)_{/\scr X}$ is generated under colimits by objects of the form $T = (c \to \scr X)$, $L_\scr{L} F \to L_\scr{L}^0 F$ is an equivalence if and only if \[ \Map_{\PSh_\Sigma(\scr C)_{/\scr X}}(T, L_\scr{L} F) \to \Map_{\PSh_\Sigma(\scr C)_{/\scr X}}(T, L_\scr{L}^0 F) \] is an equivalence for every such $T$.
This is equivalent to our stated condition by (1).
\end{proof}

The following can be used to produce some fiberwise $\ell$-completions.
\begin{lem} \label{lemm:fiberwise-comp-trick}
Let $E \to G \in \ShvSp_{\ge 1}$.
The $\ell$-completion of $\Omega^\infty E \to \Omega^\infty G$ is fiberwise if and only if $\pi_1(G)/\pi_1(E)$ is derived $\ell$-complete.
\end{lem}
\begin{proof}
The fiber sequence $F \to E \to G$ induces a fiber sequence $F_\ell^\comp \to E_\ell^\comp \to G_\ell^\comp$.
Since $\Omega^\infty$ preserves $\ell$-completions of connected spectra, we obtain a fiber sequence $\Omega^\infty(F_\ell^\comp) \to L_\ell \Omega^\infty(E) \to L_\ell \Omega^\infty(G)$.
The completion is thus fiberwise if and only if $L_\ell \Omega^\infty(F) \wequi \Omega^\infty(F_\ell^\comp)$.
This will happen if and only if $\pi_0(F)$ contributes only trivially to $F_\ell^\comp$, i.e. is already derived $\ell$-complete.
This was to be shown.
\end{proof}

\subsubsection{Topoi locally of homotopy dimension $\le 0$}
\begin{prop} \label{prop:htpy-dim-0-PSigma} \hfill
\begin{enumerate}
\item Let $\scr C$ be a site and $W \in \scr C$ a quasi-compact object of homotopy dimension $\le 0$.
  For $\scr F \in \PSh(\scr C)$, denote by $aF$ the associated hypersheaf.
  Then \[ (a\scr F)(W) \wequi \colim_{W \wequi W_1 \amalg \dots \amalg W_n} \scr F(W_1) \times \dots \times \scr F(W_n), \] where the (filtered) colimit is over all coproduct decompositions of $W$.
\item Let $\scr X$ an $\infty$-topos and $\scr X_0 \subset \scr X$ a small subcategory closed under finite coproducts, consisting of quasi-compact objects of homotopy dimension $\le 0$, and generating $\scr X$ under colimits.
  Then $\scr X \wequi \PSh_\Sigma(\scr X_0)$.
\end{enumerate}
\end{prop}
\begin{proof}
(1)
Let $\{U_\alpha \to W\}$ be a covering family of $W$.
Since $W$ has homotopy dimension $0$, the epimorphism $\amalg_\alpha U_\alpha \to W$ has a section; since coproducts are disjoint and $W$ is quasi-compact, the covering thus admits a refinement which is a disjoint union decomposition $W = W_1 \amalg W_2 \amalg \dots \amalg W_n$.
Note that each $W_i$ is quasi-compact of homotopy dimension $\le 0$.
It follows that \[ \scr F^\dagger(W) \wequi \colim_{W \wequi W_1 \amalg \dots \amalg W_n} F(W_1) \times \dots \times F(W_n). \]
We thus need to show that $\scr F^\dagger \to a\scr F$ induces an equivalence on $W$.
First note that $(\scr F^\dagger)^\dagger(W) \wequi \scr F^\dagger(W)$, since finite products commute with filtered colimits of spaces \cite[Proposition 5.5.8.6, Example 5.5.8.3]{HTT} and each of the $W_i$ has homotopy dimension $\le 0$ itself.
Since for $0$-truncated presheaves, applying $\dagger$ twice yields the sheafification, we deduce that \[ \pi_0(\scr F^\dagger(W)) \wequi \pi_0(\scr F)^\dagger(W) \wequi \ul{\pi}_0(\scr F)(W) \wequi \ul{\pi}_0(a\scr F)(W), \] where $\ul{\pi}_0$ denotes the associated homotopy sheaf.
Applying this to $a\scr F \wequi (a\scr F)^\dagger \wequi a a \scr F$ we deduce that $\pi_0(\scr F^\dagger)(W) \wequi \pi_0(a\scr F)(W)$.
Now let $b \in \scr F^\dagger(W)$, $n \ge 1$.
We have \[ \pi_n(\scr F^\dagger(W), b) \wequi \pi_n((\scr F^\dagger)^\dagger(W), b) \wequi \pi_n(\scr F^\dagger, b)^\dagger(W) \wequi \ul{\pi}_n(\scr F^\dagger,b)(W) \wequi \ul{\pi}_n(a \scr F, b)(W). \]
Again we may apply the same reasoning to $a\scr F$ to deduce that $\pi_n(\scr F^\dagger(W), b) \wequi \pi_n(a\scr F(W), b)$.
The result follows.

(2)
By \cite[Proposition 5.5.8.25(1)]{HTT}, it suffices to prove that each $W \in \scr X_0$ is compact projective, i.e. $\Map(W, \ph)$ preserves sifted colimits \cite[Remark 5.5.8.20]{HTT}.
Since $\scr X$ is locally of homotopy dimension $\le 0$ it is Postnikov complete \cite[Proposition 7.2.1.10]{HTT} and so in particular hypercomplete.
There is thus a Grothendieck site $\scr C$ (which we may assume to contain $\scr X_0$) such that $\scr X \wequi \Shv(\scr C)^\comp$ \cite[Remark 6.5.2.20]{HTT}.
Since sifted colimits commute with finite products \cite[Proposition 5.5.8.6]{HTT}, the required result follows from (1).
\end{proof}

We shall call $W \in \scr X$ w-contractible if it is quasi-compact and of homotopy dimension $\le 0$.

\begin{cor} \label{cor:htpy-dim-0-conn}
Let $\scr X$ as above, $X \in \scr X$, $E \in \ShvSp(\scr X)$.
\begin{enumerate}
\item We have $X \in \scr X_{\ge n}$ (respectively $X \in \scr X_{\le n}$) if and only if $X(W) \in \Spc_{\ge n}$ (respectively $X(W) \in \Spc_{\le n}$) for all $W$ w-contractible.
\item We have $\ul{\pi}_0(X)(W) = \pi_0(X(W))$.
\item We have $E \in \ShvSp(\scr X)_{\ge n}$ (respectively $E \in \ShvSp(\scr X)_{\le n}$) if and only if $E(W) \in \ShvSp_{\ge n}$ (respectively $E(W) \in \ShvSp_{\le n}$) for all $W$ w-contractible.
\item We have $\ul{\pi}_n(E)(W) \wequi \pi_n(E(W))$.
\end{enumerate}
\end{cor}
\begin{proof}
(1) and (2) are immediate from the equivalence $\scr X \wequi \PSh_\Sigma(\scr X_0)$ of Proposition \ref{prop:htpy-dim-0-PSigma}(2).
(3) and (4) are immediate from (1) and (2).
\end{proof}

\begin{cor} \label{cor:htpy-dim-0-compl}
Let $\scr X$ as above.
\begin{enumerate}
\item For $X \in \scr X$, $W$ w-contractible we have $(L_\scr{L} X)(W) \wequi L_\scr{L}(X(W))$.
\item For $E \in \ShvSp(\scr X)_{\ge 1}$ we have $E_\ell^\comp \in \ShvSp(\scr X)_{\ge 1}$ and $\Omega^\infty(E_\ell^\comp) \wequi L_\ell \Omega^\infty(E)$.
\end{enumerate}
\end{cor}
\begin{proof}
(1) Immediate from Proposition \ref{prop:htpy-dim-0-PSigma}(2) and Proposition \ref{prop:p-comp-PSigma}.

(2) Since $\Omega^\infty(E_\ell^\comp)$ is $\F_\ell$-complete, there is a canonical map $L_\ell \Omega^\infty(E) \to \Omega^\infty(E_\ell^\comp)$.
By Proposition \ref{prop:htpy-dim-0-PSigma}(2), we need only prove this induces an equivalence on sections over w-contractible $W$; similarly by Corollary \ref{cor:htpy-dim-0-conn}(3) it suffices to establish the connectivity on sections over $W$.
By (1) the unstable $\ell$-completion is performed sectionwise on $W$, and this is always true for the stable $\ell$-completion.
By Corollary \ref{cor:htpy-dim-0-conn}(3) the connectivity assumption passes to sections.
We thus reduce to $\scr X = \Spc$, in which case the result is well-known and follows e.g. from \cite[Theorem 2.6]{barthel2019comparison}.
\end{proof}

\begin{rmk}
Without the local homotopy dimension $\le 0$ assumption, it is not even clear if $\Omega^\infty(E_\ell^\comp)$ is connected (though $L_\ell \Omega^\infty E$ always is, by Example \ref{ex:completion-discrete}).
\end{rmk}

\subsubsection{Completion, coconnectivity and sums}
\begin{lem} \label{lemm:complete-sums} \NB{ref?}
Let $J$ be a set, and for each $j \in J$ let $E_j \in \ShvSp_\ell^\comp$.
Then \[ \pi_i \left( \left(\oplus_j E_j\right)_\ell^\comp \right) \wequi L_\ell \left( \oplus_j \pi_i E_j \right). \]
(In other words, the $\ell$-completion of the Eilenberg--Mac Lane spectrum $\oplus_j \pi_i E_j$ is again an Eilenberg--Mac Lane spectrum, and coincides with the left hand side.)
In particular, if each $E_j \in \ShvSp_{\le 0}$ then \[ \left(\oplus_j E_j \right)_\ell^\comp \in \ShvSp_{\le 0}. \]
\end{lem}
\begin{proof}
Recall that $E \in \ShvSp_\ell^\comp$ if and only if $\pi_i E \in \ShvSp_\ell^\comp$ for all $i$ (see e.g. \cite[Theorem 2.6(1)]{barthel2019comparison}).
Combining with an argument using Postnikov towers (formation of which commutes with sums) we reduce to $E_j \in \ShvSp^\heart$ for all $j$, and it remains to prove that in this case \[ \pi_1\left( \left(\oplus_j E_j\right)_\ell^\comp \right) \wequi 0. \]
We may as well establish the ``in particular'' part in general.
Since $S := \oplus_j E_j \in \ShvSp_{\le 0}$ we have $S_\ell^\comp \in \ShvSp_{\le 1}$ and hence need only show that $\pi_1$ vanishes.\NB{This is given by the infinitely divisible elements in $\pi_0(S) \wequi \oplus_j \pi_0(E_j)$. Every entry must be infinitely divisible, whence zero.}
Writing $P = \prod_j E_j$ we have\footnote{The first equivalence follows just from the definition $S_\ell^\comp \wequi \lim_n S/\ell^n$ using that $\pi_i S = 0$ for $i >0$.} \[ \pi_1(S_\ell^\comp) \wequi \Hom(\Z/p^\infty, \pi_0 S) \hookrightarrow \Hom(\Z/p^\infty, \pi_0 P) \wequi \pi_1(P_\ell^\comp) \wequi 0, \] since $P \in \ShvSp_{\le 0} \cap \ShvSp_\ell^\comp$.
This was to be shown.
\end{proof}

\subsection{Unstable $\ell$-localization}
We slightly generalize well-known results, see e.g. \cite[\S3]{wickelgren2019simplicial} \cite[\S2.3]{asok2019localization}.
Let $\scr L$ be a set of primes.
We denote by $L_{\scr L^{-1}} \Spc \subset \Spc$ the localization of spaces at the unstable $\scr L$-local equivalences, i.e. at the standard $n$-fold covering maps $\rho_n: S^1 \to S^1$ for $n \in \scr L$, as well as the maps $S^n_+ \wedge \rho_n$.

\begin{exm} \label{exm:L-local-simple-space}
If $X \in \Spc$ is simple (i.e. a disjoint union of connected spaces with abelian fundamental group acting trivially on the higher homotopy groups), then $X$ is $\scr L$-local if and only if all homotopy groups at all base points are $\scr L$-local (i.e. multiplication by $\ell \in \scr L$ is an automorphism).
In fact if $X$ is simple but not necessarily $\scr L$-local, then the homotopy groups of $L_{\scr L^{-1}} X$ are the $\scr L$-localizations of the homotopy groups of $X$ \cite{MR1123451}.
\end{exm}

\begin{lem} \label{lem:L-local-colim}
The category of $\scr L$-local spaces is stable under coproducts and filtered colimits.
\end{lem}
\begin{proof}
Clear since unstable $\scr L$-equivalences are by definition generated by maps between compact connected spaces.
\end{proof}

Given an $\infty$-category $\scr C$ and $X \in \scr C$, we say that $X$ is $\scr L$-local if for each $Y \in \scr C$ the space $\Map(Y, X)$ is $\scr L$-local.
If $\scr C$ is presentable these are the local objects of a Bousfield localization which we denote by $L_{\scr L^{-1}}$ and call $\scr L$-localization.
If $\scr L = \{p \mid p \ne \ell\}$, we also call this $\ell$-localization and denote it by $L_{(\ell)}$.

\begin{exm}
Let $\scr C$ be stable and $X, Y \in \scr C$.
Let $\ell: X \to X$ be the sum of $\ell$ copies of the identity map and suppose that $Y$ is $\ell$-local.
Then $\ell^*: \Map(X, Y) \to \Map(X, Y)$ is obtained as $\Omega \rho_\ell^*$ on $\Map(\Sigma^{-1} X, Y)$ and hence an equivalence.
Thus $\ell: X \to X$ is an $\ell$-local equivalence.
Conversely, if $Y$ is local for all maps $\ell: X \to X$, then $Y$ is $\ell$-local by Example \ref{exm:L-local-simple-space}.
It follows that categorical $\scr L$-localization recovers the usual notion.
\end{exm}

\begin{prop} \label{prop:localization-permanence}
Let $L: \scr C \adj \scr D: R$ be an adjunction.
Then $L$ preserves $\scr L$-local equivalences and $R$ preserves $\scr{L}$-local objects.
\end{prop}
\begin{proof}
Exactly the same as Proposition \ref{prop:completion-permanence}.
\end{proof}

\begin{cor} \label{cor:loc-topos}
Let $\scr X$ be an $\infty$-topos and $p: \scr X \to \Spc$ a point.
Then $p$ preserves $\scr L$-local objects and $\scr L$-local equivalences.
If $\scr X$ has enough points, then $\scr L$-local objects and $\scr L$-local equivalences can be checked on stalks.
\end{cor}
\begin{proof}
$p$ is a left adjoint by assumption, hence preserves $\scr L$-local equivalences by Proposition \ref{prop:localization-permanence}.
Since any point is of the form $X \mapsto \colim_U \Map(U, X)$ for some pro-object $U_\bullet$ of $\scr X$ \cite[Remark A.9.1.4]{SAG}, preservation of $\scr L$-local objects follows from Lemma \ref{lem:L-local-colim}.
If there are enough points, equivalences can be checked on stalks, and hence so can $\scr L$-local equivalences and $\scr L$-local objects by what we just did.
\end{proof}

\subsection{The étale site of a profinite group}\label{subsub:SHet-BG}
Let $G$ be a profinite group.
We have the category $\Fin_G$ of finite (discrete) $G$-sets, which we can make into a site by giving it the effective epimorphism topology, i.e. the topology of (joint) surjections.
Somewhat abusively, we call this the \emph{étale topology} (as opposed to the topology of disjoint unions, which we call the \emph{Nisnevich topology}).
We denote by $B_\et G$ the resulting $\infty$-topos.
We write $\ShvSp(B_\et G)$ for the associated category of hypercomplete spectral sheaves.
We have the evident symmetric monoidal embedding \[ \Fin_G \to \Spc^{BG} =: \Fun(BG, \Spc). \]
Passing to presheaves, this factors the stalk functor, and hence preserves étale hyperequivalences.
Consequently we obtain an induced functor \begin{equation}\label{eq:BG-top-disc-comp} \Shv_\et(BG) \to \Spc^{BG}, \end{equation} and similarly on stabilizations.

\begin{rmk} \label{rmk:BG-disc}
Note that $\pi_*(G) = 0$ for $*>0$, and hence $G$ is homotopically discrete.
It follows that when forming the category $\Spc^{BG}$, the topology on $G$ could just as well have been discrete.
Unless $G$ is finite, one thus does not expect the above functor to be an equivalence.
\end{rmk}

\begin{rmk} \label{rmk:SHet-BG-finite}
On the other hand, if $G$ \emph{is} finite, then $\Shv_\et(BG) \wequi \Spc^{BG}$.
See e.g. \cite[Example 4.4]{clausen2019hyperdescent}.
\end{rmk}

Sometimes the category $\ShvSp(B_\et G)$ can be described in terms of $\ShvSp^{BH}$ for an appropriate group $H$, Remark \ref{rmk:BG-disc} notwithstanding.
The following is a minor generalization of \cite[Lemma 3.3]{bhatt2020remarks}.
\begin{lem} \label{lem:describe-SHetp-BG}
Let $G$ be a profinite group, $P$ a discrete group and $q: P \to G$ a homomorphism with dense image.
Suppose that for some cofinal family of open subgroups $K_i \subset G$, $\cd_\ell(q^{-1} K_i)$ is bounded and $\colim_i H^n(q^{-1} K_i, M) = 0$ for any $\Z/\ell$-vector space $M$ with a continuous $G$-action and $n>0$.
(This condition holds if for some cofinal family, each $B q^{-1}K_i$ is a finite complex, their dimensions are bounded, and $\colim_i H^n(q^{-1}K_i, \Z/\ell) = 0$ for $n>0$.\NB{better conditions?})

Then the induced map \[ q^*: \ShvSp(B_\et G)_\ell^\comp \to (\ShvSp^{BG})_\ell^\comp \to (\ShvSp^{BP})_\ell^\comp \] is fully faithful onto the subcategory of those $P$-spectra where the $P$-action on the mod $\ell$ homotopy groups extends to a continuous $G$-action (i.e. every element is fixed by $ker(q)$ and by $q^{-1}(K_i)$ for $i$ sufficiently large).
\end{lem}
\begin{proof}
We essentially repeat the proof of \cite[Lemma 3.3]{bhatt2020remarks}.
If $E \in \ShvSp(B_\et G)$ then the homotopy groups of the stalk are given as $\colim_i A_i$, where $A_i$ carries an action by $G/K_i$ (assuming here, without loss of generality, that $K_i$ is normal); hence any spectrum in the essential image has the claimed property.
Since $q^*$ is conservative, it thus remains to show that if $E \in (\ShvSp^{BP})_\ell^\comp$ is as claimed, then $q^*q_* E/\ell \to E/\ell$ is an equivalence.
Since $q$ is topologically surjective and the space $G/K_i$ is discrete, $P/q^{-1}(K_i) \wequi G/K_i$ and the requirement unwinds as \[ E/\ell \wequi \colim_i (E/\ell)^{h q^{-1}(K_i)}. \]
The bounded $\ell$-cohomological dimension implies that in order to prove this, we may assume that $E$ is bounded above (see e.g. \cite[Lemma 2.9]{bachmann-SHet}).
The homotopy groups of the right hand side in degrees $\ge n$ only depend on those of the left hand side in degrees $\ge n$; hence we may assume that $E$ is bounded and by induction that $E$ is concentrated in degree $0$.
Considering the cofiber sequence $(E/\ell)_{\ge 1} \to E/\ell \to (E/\ell)_{\le 0}$, it thus suffices to prove that if $M$ is an $\F_\ell$-vector space with an action by $P$ extending continuously to $G$, then \[ M \wequi \colim_i H^*(q^{-1}(K_i), M). \]
This is clear in degree $0$, and the vanishing in positive degrees holds by assumption.

To prove the parenthetical statement, let $\{M_j \subset M\}$ be the set of $P$-invariant subspaces fixed by $q^{-1}(K_i)$ for $i$ sufficiently large.
This set is filtered and $M = \bigcup_j M_j$ (the latter because if $x \in M$ then $\Z/\ell[P]x$ is a finite set, and hence fixed by $q^{-1}(K_i)$ for $i$ sufficiently large).
The assumption on $B q^{-1}(K_i)$ implies that $H^*(q^{-1}(K_i), \ph)$ preserves filtered colimits \cite[Theorem 4]{brown1975homological}\NB{more direct reference??}.
Hence we may assume that $M$ is fixed by $q^{-1}(K_i)$ for all $i$, or in other words $M = \bigoplus_\Lambda \Z/\ell$.
Then $H^*(q^{-1}(K_i),M) \wequi \bigoplus_\Lambda H^*(q^{-1}(K_i),\Z/\ell)$ by preservation of filtered colimits again, yielding the desired reduction.
\end{proof}

\begin{exm} \label{ex:SH-BG}
Lemma \ref{lem:describe-SHetp-BG} applies in the following cases.
\begin{enumerate}
\item $G=P$ finite.
\item $G=\Z_\ell, P=\Z$.
\item $G=\Z_\ell^\times$, $P=\mu \times \Z$, $\mu \subset \Z_\ell^\times$ the roots of unity and $\Z \subset \Z_\ell^\times \wequi \mu \times \Z_\ell$ generating the second factor.
\item $G=\Z_\ell^\times$, $P=\Z$, $\ell$ odd, $q$ the diagonal.
  (This is topologically surjective because $\Z_\ell^\times \wequi \mu_{\ell-1} \times \Z_\ell$.)
\end{enumerate}
In case (1) we take $K_i =\{e\}$ and the condition is trivial.
In cases (2) and (3) one gets $q^{-1}(K_i) = \ell^i \Z$, whereas in case (4) we have $q^{-1}(K_i) = (\ell-1)\ell^i \Z$.
The parenthetical condition holds because $Bq^{-1}(K_i) = S^1$ and $H^*(q^{-1}(K_i), \Z/\ell) = \Z/\ell[\epsilon]/\epsilon^2$ for all $i$, and the transition maps annihilate $\epsilon$.
\end{exm}

\begin{lem} \label{lem:rigid-generation}
Let $G$ be a profinite group.
The category $\ShvSp(B_\et G)$ is generated by strongly dualizable objects, namely the suspension spectra of finite discrete $G$-sets.
\end{lem}
\begin{proof}
It is clear by construction that the category is generated by the suspension spectra of finite discrete $G$-sets; it remains to prove that these are strongly dualizable.
Any finite discrete $G$-set is inflated from a finite quotient of $G$, so we may assume that $G$ is finite.
We write $\GenSp(BG)$ for the category of genuine equivariant $G$-spectra.
Since there is a functor $\GenSp(BG) \to \ShvSp(B_\et G)$, the claim reduces to the analogous claim for $\GenSp(BG)$, which is well-known.\NB{ref?}
\end{proof}

\subsection{Cyclotomic characters} \label{subsec:cyclo-char}
The Galois extension $\QQ(\mu_{\ell^\infty})/\QQ$ has group $\Z_\ell^\times$ and is unramified outside $(\ell)$ \cite[\S10]{neukirch2013algebraic}.
It follows that there is a canonical functor $\Fin_{\Z_\ell^\times} \to \FEt_{\Z[1/\ell]}$, inducing the \emph{cyclotomic character} \[ \pi: (\Spec \Z[1/\ell])_\et \to B\Z_\ell^\times. \]
If $S$ is any scheme with $1/\ell \in S$, then we obtain also the composite $\pi: S_\et \to (\Spec \Z[1/\ell])_\et \to B\Z_\ell^\times$.

We will want to use some variants of this construction.
\begin{thm}[Dwyer--Friedlander] \label{thm:etale-fund-gps}
Let $G_2^\mu = \Z/2 *_2 \Z_2$ and $G_\ell^\mu$ the free pro-$\ell$ group on $(\ell+1)/2$ generators.
Assume that $\ell=2$ or $\ell$ is odd and regular.
There exists a map $\pi_1^\et(\Spec(\Z[1/\ell,\mu_\ell])) \to G_\ell^\mu$ inducing \[ H^*_\et(\Spec(\Z[1/\ell,\mu_\ell]), \Z/\ell) \wequi H^*_\et(BG_\ell^\mu, \Z/\ell). \]
\end{thm}
\begin{proof}
For a scheme $X$, denote by $s(X) \in \Pro(\Spc^\pi)$ its étale homotopy type (see e.g. \cite[\S5]{hoyois2018higher}.)
Then $\pi_1(s(X)) \wequi \pi_1^\et(X)$.
Write $L: \Pro(\Spc^\pi) \to \Pro(\Spc^\ell)$ for the pro-$\ell$ completion.
A map $\alpha: A \to B \in \Pro(\Spc^\pi)$ induces an equivalence $L\alpha$ if and only if $H^*(\alpha, \Z/\ell)$ is an isomorphism \cite[Theorem 7.4.11]{barnea2017pro}.
If $G$ is a pro-$\ell$ group, objects of the form $K(G,1)$ lie in $\Pro(\Spc^\ell)$.

Now let $X=\Spec(\Z[1/\ell,\mu_\ell])$, $H_2 = \Z/2 * \Z$ and $H_\ell$ the free group on $(\ell+1)/2$ generators.
In \cite{dwyer1983conjectural}, the authors construct zigzags of equivalences in $\Pro(\Spc^\pi)$ between $s(X)$ and $K(H_\ell, 1)$.
It follows that $LsX \wequi LK(H_\ell,1) \wequi K(G_\ell^\mu,1)$.
The result follows.
\end{proof}

\begin{dfn} \label{def:modified-cyc-char}
We obtain as before a canonical functor (for $\ell=2$ or odd regular) $\Fin_{G_\ell^\mu} \to \FEt_{\Z[1/\ell,\mu_\ell]}$, inducing the \emph{modified cyclotomic character} \[ \pi_\mu: (\Spec \Z[1/\ell,\mu_\ell])_\et \to BG_\ell^\mu. \]
In order to unify notation, put $G_2 = G_2^\mu$ and $G_\ell = \Z_\ell^\times$ for any $\ell$ odd, and $\tilde\pi = \pi_\mu$ if $\ell=2$ (note that $\Z[1/2] = \Z[1/2,\mu_2]$), $\tilde\pi = \pi$ for $\ell$ odd.
\end{dfn}

\begin{rmk} \label{rmk:pi-factorization} \NB{justification?}
We have a factorization \[ \pi: (\Spec \Z[1/2])_\et \xrightarrow{\tilde\pi} BG_2 \to B(\Z/2 \times \Z_2) \wequi B\Z_2^\times, \] where the second map is induced by projection to the abelianization.
Similarly if $\ell$ is odd regular, then we have a factorization \[ \pi: (\Spec \Z[1/\ell,\mu_\ell])_\et \xrightarrow{\pi_\mu} BG_\ell^\mu \to B\Z_\ell \to B(\Z/(\ell-1) \times \Z_\ell) \wequi B\Z_\ell^\times, \] where the middle map comes from the addition map $G_\ell^\mu \to \Z_\ell$.
\end{rmk}

\subsection{Galois equivariant realization}
Let $k$ be a field of characteristic $\ne \ell$ and finite virtual $\ell$-étale cohomological dimension.
Put $G=Gal(k^s/k)$, for some fixed separable closure $k^s$.
Étale realization (see e.g. \cite[Theorem 6.6]{bachmann-SHet}) together with \eqref{eq:BG-top-disc-comp} yields a symmetric monoidal functor \[ \SH(k) \to \SH_\et(k)_\ell^\comp \wequi \ShvSp(k_\et)_\ell^\comp \wequi \ShvSp(B_\et G)_\ell^\comp \to (\ShvSp^{BG})_\ell^\comp \] which we call Galois equivariant realization.

Note that if $k=\CC$ then we also have complex realisation $r_\CC: \SH(\CC) \to \ShvSp$.
\begin{lem} \label{lemm:galois-equiv-enhancement}
The diagram
\begin{equation*}
\begin{CD}
\SH(\CC) @>{r_\CC}>> \ShvSp \\
@VVV               @VVV \\
\SH_\et(\CC)_\ell^\comp @= \ShvSp_\ell^\comp
\end{CD}
\end{equation*}
commutes.
\end{lem}
In particular if $k \subset \CC$ and $E \in \SH(k)$, we may view the Galois equivariant realization of $E$ as a Galois action on $r_\CC(E)_\ell^\comp$.
\begin{proof}
We may replace $\ShvSp$ by $\ShvSp_\ell^\comp$ in the above diagram.
Since complex realization factors through étale localization, we may replace $\SH(\CC)$ by $\SH_\et(\CC)_\ell^\comp$.
Now all the categories are equivalent to $\ShvSp_\ell^\comp$ and all the functors are symmetric monoidal, so the diagram commutes since $\ShvSp_\ell^\comp$ is the initial stable presentably symmetric monoidal $\ell$-complete category.
\end{proof}

\subsection{Slice completeness}
For any scheme $S$, we have the slice cover and cocover functors $f_n, f^n: \SH(S) \to \SH(S)$; see \cite{voevodsky-slice-filtration} for the definition of $f_n$ and put $f^n = \cof(f_{n+1} \to \id)$.

\begin{lem} \label{lem:MGL-sc}
Let $\scr S$ be a set of primes, $S$ essentially smooth over a Dedekind scheme on which all primes \emph{not} in $\scr S$ are invertible.
We have $f_n \MGL_S[\scr S^{-1}] \in \Sigma^{2n,n} \SH(S)^\veff$, and similarly for $\KGL[\scr S^{-1}]$, or any other spectrum in the subcategory generated under colimits and extensions by $\Sigma^{2*,*}\MGL_S[\scr S^{-1}]$, e.g. obtained from $\MGL_S[\scr S^{-1}]$ by inverting or killing elements in degrees $(2*,*)$.
In particular all of these spectra is slice complete.
\end{lem}
\begin{proof}
The claim implies slice completeness by Lemma \ref{lemm:sc-nondeg} below.
We have $\KGL[\scr S^{-1}] \wequi \MGL/(x_2, x_3, \dots)[x_1^{-1}]$ by \cite[Theorem 5.2]{spitzweck2010relations}, which is of the claimed form.
If $E \in \SH(S)$ satisfies $f_n E \in \Sigma^{2n,n} \SH(S)^\veff$ (for all $n$), then the same is true for $\Sigma^{2p,p} E$ for any $p \in \Z$, and hence for any colimit or extension of such spectra.
In particular the property persists after inverting or killing elements in degree $(2*,*)$.
We have thus reduced to proving the claim for $\MGL_S[\scr S^{-1}]$.

There is a certain sequence of categories \[ I \supset I_1 \supset I_2 \supset \dots \] and a diagram $D: I \to \SH(\Z)$ \cite[\S4]{spitzweck2010relations}.
In fact $I$ is the poset of monomials in variables $x_1, x_2, \dots$ ordered by divisibility, $x_i$ is given degree $i$, $I_n$ consists of the subcategory of monomials of degree $\ge n$, and $D$ sends a monomial of degree $n$ to $\Sigma^{2n,n}\MGL$.
The proof of \cite[Theorem 4.7]{spitzweck2010relations} shows that if the canonical map \[ \MGL_S/(x_1, x_2, \dots)[\scr S^{-1}] \to H\Z[\scr S^{-1}] \] is an equivalence onto the zero slice, then \begin{equation} \label{eq:MGL-fi-formula} f_i \MGL_S[\scr S^{-1}] \wequi \colim_{I_i} D_S[\scr S^{-1}]. \end{equation}
Since $D_S(I_i) \subset \Sigma^{2i,i} \SH(S)^\veff$, the Hopkins--Morel isomorphism over $S$ (after inverting $\scr S$) and $s_0(\MGL) \wequi H\Z$ (after inverting $\scr S$) together imply the desired result.
The former was established in \cite[Theorem 11.3]{spitzweck2012commutative} and the latter in \cite[Theorem B.4]{norms}.
\end{proof}

\begin{lem} \label{lemm:sc-nondeg}
Let $S$ be a noetherian scheme of finite dimension and $E \in \SH(S)$.
Suppose that there is a sequence $n_i$ such that $\lim_{i \to \infty} n_i = \infty$ and $f_i E \in \SH(S)_{\ge n_i}$ (for all $i$).
Then $E \wequi \lim_i f^i E$ and also $[X, E] \wequi \lim_i [X, f^i E]$ for any compact $X \in \SH(S)$ (e.g. $X = \Sigma^{p,q}\1$).
\end{lem}
\begin{proof}
Let $X$ be compact.
We claim that $[X, f_i E] = 0$ for $i$ sufficiently large.
This implies that $[X, \lim_i f_i E] = 0$ and $[X, \lim_i f^i E] \wequi \lim_i [X, f^i E]$ (using the Milnor exact sequence).
Since $\SH(S)$ is compactly generated, the former vanishing shows that $\lim_i f_i E = 0$, and hence $E \wequi \lim_i f^i E$ as desired.

To prove the claim, since $\SH(S)$ is compactly generated by objects of the form $\Sigma^{i,j} \Sigma^\infty_+ U$ for $X \in \Sm_S$, and the category of objects $X$ satisfying the claim is thick, we may (replacing $E$ by a suspension if necessary) assume that $X = \Sigma^\infty_+ U$.
If $U$ has dimension $e$ and $F \in \SH(S)_{>e}$, then $[\Sigma^\infty_+U, F] = 0$ \cite[Proposition 3.7]{schmidt2018stable}.
The result follows.
\end{proof}

\begin{rmk} \label{rmk:MGL-slices}
In the situation of Lemma \ref{lem:MGL-sc}, one has $s_* \MGL[\scr S^{-1}] \wequi H\Z \otimes L_*[\scr S^{-1}]$ and $s_* \KGL[\scr S^{-1}] \wequi H\Z[\scr S^{-1}, v^{\pm 1}]$ with $|v|=(2,1)$.
This follows from \cite[Theorems 4.7 and 5.2]{spitzweck2010relations}, which apply as in the proof of Lemma \ref{lem:MGL-sc}.
\end{rmk}

\subsection{The $b$-topology} \label{subsec:b-top}
Recall the $b$-topology and localization from \cite{real-and-etale-cohomology,elmanto2019scheiderer}.
To be precise, whenever we speak of $b$-sheaves or $b$-sheafification $L_b$, we mean hypersheaves and hypersheafification.

\begin{lem} \label{lemm:b-pullback-square}
Let $E \in \SH(S)$.
Then we have a pullback square
\begin{equation*}
\begin{CD}
L_b E @>>> L_\et E \\
@VVV        @VVV \\
E[\rho^{-1}] @>>> L_\et(E)[\rho^{-1}].
\end{CD}
\end{equation*}
\end{lem}
\begin{proof}
This is a formal (see e.g. \cite[Paragraph 1.2]{quigley2019parametrized}) consequence of the fact that $\SH_b(S)$ is a recollement of $\SH_{\ret}(S)$ and $\SH_\et(S)$ \cite[Example A.14]{elmanto2019scheiderer}.
\end{proof}

\begin{cor} \label{cor:Lb-mod-rho}
We have $L_b(E)[\rho^{-1}] \wequi E[\rho^{-1}]$ and $L_b(E)/\rho \wequi L_\et(E)/\rho$.
\end{cor}
\begin{proof}
Immediate from the pullback square of Lemma \ref{lemm:b-pullback-square}
\end{proof}

\begin{lem} \label{lemm:b-top-coh-dim}
Let $S$ be a qcqs scheme of finite dimension and finite $\pvcd$, with $1/2 \in S$.
The category $L_b \SH(S)_\ell^\comp$ is compactly generated by (desuspended) suspension spectra mod $\ell$.
\end{lem}
\begin{proof}
If $X$ is qcqs scheme with $1/2 \in X$, then $cd_b(X) \le cd_\et(X[\sqrt{-1}]) + 1$ \cite[Corollary 7.4.1 and 7.18]{real-and-etale-cohomology}.
Everything follows formally from this.
In more detail, if $X \in \Sm_S$, then one deduces that $X$ has finite $b$-cohomological dimension.
This implies that the suspension spectrum of $X$ mod $\ell$ is compact in the topos of $b$-sheaves, and then it remains so in $L_b \SH(S)_\ell^\comp$ (see e.g. \cite[Lemma 2.7, proof of Lemma 2.16]{bachmann-SHet}).
\end{proof}

\begin{cor} \label{cor:Lb-colim}
The functor $L_b: \SH(S)_\ell^\comp \to \SH(S)_\ell^\comp$ preserves colimits.
\end{cor}
\begin{proof}
Immediate from Lemma \ref{lemm:b-top-coh-dim}.
\end{proof}

\subsection{Motivic colimits} \label{sec:motivic-colimits}
Let $\scr C$ be a small category with a final object $*$ and $\scr D \in \Fun(\scr C^\op, \Cat)$.
Suppose that $\scr D(*)$ is cocomplete and each of the functors $f^*: \scr D(*) \to \scr D(c)$ for $f: c \to * \in \scr C$ admits a left adjoint $f_\#$.
Finally suppose given $A \in \PSh(\scr C)_{/\scr D^\wequi}$.
We define a functor \[ M = M_A: \PSh(\scr C)_{/A} \wequi \PSh(\scr C_{/A}) \to \scr D(*) \] as the cocontinous extension of the functor $\scr C_{/A} \to \scr D(*)$ sending $c \to A$ to the $f_\#(X)$, where $f: c \to *$ is the unique map and $X \in \scr D(c)$ is the corresponding object.
See \cite[\S2]{bachmann-colimits} for more details (in a slightly different context).

We recall some of the salient properties of this construction.
\begin{prop} \label{prop:mot-colim-basics} \hfill
\begin{enumerate}
\item If $A$ is a sheaf (respectively hypersheaf) for some topology $\tau$ on $\scr C$, then $M_A$ factors through $\Shv_\tau(\scr C)_{/A}$ (respectively $\Shv_\tau^\comp(\scr C)_{/A}$).
\item If $\alpha: \scr D \to \scr E \in \Fun(\scr C^\op, \Cat)$, where $\scr E$ satisfies all of the same assumptions and moreover all of the exchange transformations $f_\#\alpha \to \alpha f_\#$ are isomorphisms for all $f: c \to * \in \scr C$, then there is a canonical equivalence \[ \alpha M_{\alpha_\# A} \wequi M_A \alpha_\#. \]
  Here by $\alpha_\#: \PSh(\scr C)_{/\scr D^\wequi} \to \PSh(\scr C)_{/\scr E^\wequi}$ we denote the functor of composition with $\alpha$.
\item If $p: \scr C_0 \to \scr C$ is a functor of small categories such that $\scr C_0$ has a terminal object $*$ preserved by $p$, then there is a canonical equivalence \[ M_{p^*A} \wequi M_A p. \]
  Here by $p^*: \PSh(\scr C) \to \PSh(\scr C_0)$ we denote the restriction functor and by $p: \PSh(\scr C_0)_{/p^* A} \to \PSh(\scr C)_{/A}$ the functor induced by the left adjoint of $p^*$.
\end{enumerate}
\end{prop}
\begin{proof}
(1) Straightforward adaptation of \cite[proof of Proposition 16.9(2)]{norms}.

(2,3) Since both sides are cocontinuous functors, it suffices to check the claims on representables (i.e. objects of the form $c \to A$, respectively $c_0 \to p^*A$), where they are clear.
\end{proof}

\subsection{Projection formula}
We shall use the following well-known fact.
\begin{lem} \label{lemm:projection-formula}
Let \[ L: \scr C \adj \scr D: R \] be an adjunction between symmetric monoidal $\infty$-categories, with $L$ symmetric monoidal.
For $c \in \scr C, d \in \scr D$ there is a canonical binatural transformation \[ c \otimes R(d) \to R(L(c) \otimes d). \]
This natural transformation is an equivalence if one of the following holds:
\begin{enumerate}
\item $c$ is strongly dualizable.
\item $\scr C, \scr D$ are stable, $\scr C$ is compact-rigidly generated and $L$ sends a compact generating family of $\scr C$ to compact objects of $\scr D$.
\end{enumerate}
\end{lem}
\begin{proof}
The natural transformation is the composite \[ c \otimes Rd \xrightarrow{u} RL(c \otimes Rd) \wequi R(L(c) \otimes LRd) \xrightarrow{c} R(L(c) \otimes d), \] where $u$ denotes a unit of adjunction, $c$ denotes a counit of adjunction, and the middle equivalence comes from the symmetric monoidal structure on $L$.

(1) Straightforward manipulation using strong dualizability shows that for any $X \in \scr C$ the canonical map $[X, c \otimes Rd] \to [X, R(L(c) \otimes d)]$ is a bijection.
The result follows.

(2) The assumptions imply that $R$ preserves colimits.
For any $d \in \scr D$, the class of objects $\scr C_0 \subset \scr C$ on which the transformation is an equivalence is stable under finite limits and colimits (by stability), and under all colimits by the previous sentence.
Since $\scr C_0$ contains the strongly dualizable objects (by (1)), and such objects generate $\scr C$ under colimits (by assumption), we conclude.
\end{proof}

\section{Étale $K$-theory} \label{sec:etale-K}
\localtableofcontents

\bigskip
Let $S$ be a scheme with $1/\ell \in S$.
We can view $K(1)$-local algebraic $K$-theory as an étale sheaf on $S$: $L_{K(1)} K(\ph) \in \ShvSp(S_\et)$.
As observed in \cite[Theorem 3.9]{bhatt2020remarks}, this sheaf essentially only depends on the cyclotomic character of $S_\et$.
That is, write $\pi: S_\et \to B\Z_\ell^\times$ for the cyclotomic character, and write $\KU_\ell \in \ShvSp(B_\et \Z_\ell^\times)$ for the classical $\ell$-adically completed complex $K$-theory spectrum $\KU_\ell^\comp$ with the continuous action of $\Z_\ell^\times$ by Adams operations.
Then one has \[ L_{K(1)} K(\ph) \wequi \pi^* \KU_\ell \in \ShvSp(S_\et)_\ell^\comp. \]

In this section we will establish an unstable version of this result.
Write $L_\et K^\circ \in \Shv(S_\et)$ for the étale sheafification of the presheaf of spaces given by the rank zero part of algebraic $K$-theory.
Write $\BU_\ell \in \Shv(B_\et \Z_\ell^\times)$ for the connected cover of the infinite loop sheaf of $\KU_\ell$.
We would like to establish an equivalence of the form \[ \pi^* \BU_\ell \wequi L_\et K^\circ \] up to $\ell$-adic completion.
The difficulty which arises is that we need to use a notion of $\ell$-adic equivalence in an unstable setting, that is, for sheaves of \emph{spaces}.
We recalled a construction of unstable $\ell$-completion in \S\ref{subsec:unstable-ell-comp}.
As we showed there, this operation behaves particularly well if the relevant topos has homotopy dimension $0$.
Unfortunately this is not the case for the topoi mentioned so far.
Our way out is to replace the étale topos by the pro-étale topos of \cite{bhatt2013pro}.
This has homotopy dimension $0$, and in reasonable situations contains the usual étale topos as a full subcategory.\footnote{Recall our convention that all topoi are hypercomplete(d).}

\subsubsection*{Organization}
This section is organized as follows.
First in \S\ref{subsec:large-pro-etale} we prove some facts about various types of pro-étale sheaves on $S$.
We prove somewhat more than is needed for our immediate application, in preparation for later sections.
Then in \S\ref{subsec:pro-G-site} we define the analog of the pro-étale site for $\Shv(B_\et \Z_\ell^\times)$ and establish some of its basic properties.
Again in anticipation of future needs, our results are a bit more expansive than immediately necessary.
Finally in \S\ref{subsec:Ket-main} we prove our main reconstruction result, along the lines mentioned above.
Due to all our preliminary work, the proof is essentially straightforward.
Assuming we have a comparison map (which we can obtain e.g. from \cite[Theorem 3.9]{bhatt2020remarks}), we need only check that we have an equivalence on stalks.
Thus we need to compare $\ell$-adic $K$-theory of strictly henselian local rings with $\BU_\ell^\comp$.
This is possible by celebrated results of Gabber and Suslin.

\subsection{The large pro-étale site} \label{subsec:large-pro-etale}
Recall the pro-étale topology, and in particular the notion of weakly étale maps, from \cite{bhatt2013pro}.
\begin{dfn}
We denote by $\Sch_S$ the category of $S$-schemes, suitably bounded in cardinality.
By a pro-étale covering we mean an fpqc coverings consisting of weakly étale maps.
\end{dfn}

The main reason for considering this site is as follows.
Recall the notion of a weakly contractible scheme from \cite[Definition 1.4]{bhatt2013pro}.
\begin{lem}
Weakly contractible schemes define quasi-compact objects of homotopy dimension $\le 0$ of $\Shv_\proet(\Sch_S)$, and generate the category under colimits.
\end{lem}
\begin{proof}
These properties are inherited from the usual pro-étale sites. (See in particular \cite[Therem 1.5]{bhatt2013pro} for the generation; quasi-compactness is essentially automatic and homotopy dimension $\le 0$ is essentially the definition of weakly contractible.).
\end{proof}

We define \[ \Spc_\proet(S) = L_{\A^1} \Shv_\proet(\Sch_S), \SH_\proet^{S^1}(S) = \Spc_\proet(S)_*[(S^1)^{-1}] \] and \[ \SH_\proet(S) = \Spc_\proet(S)_*[(\P^1)^{-1}]. \]
Beware that $\Spc_\proet(S)$ is \emph{not} a localization of $\Spc(S)$ (or even $\ul{\Spc}(S)$), and so on.

\begin{lem} \label{lem:SH-proet-descent}
The functors $\Spc_\proet(\ph)$, $\Spc_\proet(\ph)_*$, $\SH^{S^1}_\proet(\ph)$, $\SH_\proet(\ph)$, $\SH_\proet^{S^1}(S)_\ell^\comp$ and $\SH_\proet(\ph)_\ell^\comp$ are pro-étale hypersheaves.
\end{lem}
\begin{proof}
By topos descent \cite[Theorem 6.1.3.9(3)]{HTT}, $\Shv_\proet(\Sch_{\ph}) \wequi \Shv_\proet(\Sch_S)_{/\ph}$ is a pro-étale hypersheaf.
This property persists when passing to pointed objects because passing to pointed objects commutes with limits.
To show that $\Spc_\proet(\ph)$, $\Spc_\proet(\ph)_*$ are sheaves it remains to show that $\scr X \in \Shv_\proet(\Sch_{\ph})$ is $\A^1$-invariant if and only if it is so on pullbacks to a pro-étale covering of the base, which is clear.
Since inverting a symmetric object is a limit, it preserves the sheaf condition.
It remains to deal with the $\ell$-complete categories, whence we need to show that being $\ell$-complete can be checked on a pro-étale covering.
This holds since being $\ell$-complete is equivalent to being right orthogonal to $\ell$-periodic objects.
\end{proof}

\begin{lem} \label{lem:SHS-SH-proet}
Let $1/\ell \in S$.
We have \[ \SH^{S^1}_\proet(S)_\ell^\comp \wequi \SH_\proet(S)_\ell^\comp. \]
\end{lem}
\begin{proof}
It suffices to prove that $\Gm$ is invertible in $\SH^{S^1}_\proet(S)_\ell^\comp$.
This follows from the same assertion for $\SH^{S^1}_\et(S)_\ell^\comp$, i.e. \cite[Theorem 3.1]{bachmann-SHet2}.
\end{proof}

\begin{dfn}
We call $S$ \emph{strongly locally $\ell$-étale finite} if it is étale locally of uniformly bounded $\ell$-étale cohomological dimension \cite[Definition 2.11]{bachmann-SHet}.
\end{dfn}

\begin{dfn}
Left Kan extension supplies us with functors \[ \epsilon^*: \ShvSp(S_\proet) \to \ShvSp_\proet(\Sch_S) \quad\text{and}\quad \nu^*: \ShvSp(S_\et) \to \ShvSp_\proet(\Sch_S). \]
We denote by similar notation other variants of these functors.
(That is $\epsilon^*$ extends from the small to the big pro-étale site, whereas $\nu^*$ extends from the (small or big) étale site to the big pro-étale site.)
\end{dfn}

\begin{lem} \label{lemm:nu-proet-pSm-ff}
The functor $\pi^*: \Shv(S_\proet) \to \Shv_\proet(\Sch_S)$ is fully faithful, and similarly stabilization.
\end{lem}
\begin{proof}
It is enough to prove the unstable statement, and since the analog for presheaves holds, it suffices to show that the presheaf level functor $\pi_*^p$ sends hypercoverings to hyper-pro-étale equivalences.
Thus let $X_\bullet \to X \in \Sch_S$ be a pro-étale hypercovering.
To prove that $\pi_*^p(|X_\bullet|) \wequi |\pi_*^p(X_\bullet)| \to \pi_*^p(X)$ is a hyper-pro-étale equivalence, by universality of colimits it suffices to prove that for $Y \to S$ weakly étale and $Y \to X$ any map, the base change $|X_\bullet \times_X Y| \to Y$ is a (hyper)pro-étale equivalence.
This follows from the fact that this base change is a pro-étale hypercovering.
\end{proof}

\begin{prop} \label{prop:strongly-finite-ff}
Let $1/\ell \in S$.
\begin{enumerate}
\setcounter{enumi}{-1}
\item If $E \in \ShvSp(S_\proet)_\ell^\comp$ has classical mod $\ell$ homotopy sheaves, then $\epsilon^*(E) \in \ShvSp_\proet(\Sch_S)_\ell^\comp$ is $\A^1$-invariant.
\item $\epsilon^*: \ShvSp(S_\proet)_\ell^\comp \to \SH_\proet(S)_\ell^\comp$ is fully faithful on objects with classical mod $\ell$ homotopy sheaves\NB{is this assumption necessary?}.
\end{enumerate}
Now assume that in addition $S$ strongly locally $\ell$-étale finite.
\begin{enumerate}
\setcounter{enumi}{1}
\item $\nu^*: \ShvSp(S_\et)_\ell^\comp \to \SH_\proet(\Sch_S)_\ell^\comp$ is fully faithful.
\item $\nu^*: \SH_\et(S)_\ell^\comp \to \SH_\proet(S)_\ell^\comp$ is fully faithful (and similarly for $S^1$-spectra).
\end{enumerate}
\end{prop}
\begin{proof}
(0) Since $\epsilon^*$ commutes with truncation, Postnikov towers converge in $\ShvSp_\proet(\Sch_S)$, and $\A^1$-local objects are stable under limits, it suffices to prove this for bounded above spectra.
Using \cite[Lemma 2.6]{bachmann-SHet} this reduces to spectra concentrated in one degree.
Thus we need to prove that for $X \in \Sch_S$ and $F \in \Ab(S_\et)$ a $\Z/\ell$-module we have $\map(X, \nu^*(F)) \wequi \map(X \times \A^1, \nu^*(F))$.
Since $\nu^*$ is fully faithful on bounded above spectra over any base \cite[Lemma 5.1.2]{bhatt2013pro}, we have reduced to proving that étale cohomology with mod $\ell$ coefficients is homotopy invariant, which is well-known \cite[Corollaire XV.2.2]{sga4}.

(1) By Lemma \ref{lem:SHS-SH-proet} it suffices to prove the statement with target $S^1$-spectra.
Via Lemma \ref{lemm:nu-proet-pSm-ff}, we reduce to (0).

(2) The functor factors as $\ShvSp(S_\et)_\ell^\comp \to \ShvSp(S_\proet)_\ell^\comp \to \SH_\proet(\Sch_S)_\ell^\comp$; the first functor is fully faithful by \cite[Lemma 3.4]{bachmann-SHet}, and the second one is fully faithful by Lemma \ref{lemm:nu-proet-pSm-ff}.

(3) Since the source identifies with $\ShvSp(S_\et)_\ell^\comp$ \cite[Theorem 3.1]{bachmann-SHet2}, this functor is equivalent to the composite $\ShvSp(S_\et)_\ell^\comp \to \ShvSp(S_\proet)_\ell^\comp \to \SH_\proet(S)_\ell^\comp$.
We just saw that the first functor is fully faithful, and the second is fully faithful by (1).
The result follows.
\end{proof}

\subsection{The pro-$G$ site} \label{subsec:pro-G-site}
Let $G$ be a profinite group.
We obtain the category (see \S\ref{subsub:SHet-BG}) \[ \Shv_\et(BG) := \Shv(\Fin_G). \]

Denote by $\pFin_G$ the subcategory of $\Pro(\Fin_G)$ consisting of objects with suitably bounded cardinality.
We give this the topology of effective epimorphisms.
This way we obtain a category \[ \Shv_\proet(BG) := \Shv(\pFin_G). \]

\begin{rmk} \label{rmk:solidification-G}
By \cite[Proposition 3.3.9]{barwick2019pyknotic}, our category $\Shv_\proet(BG)$ coincides with the \emph{solidification} of the $\infty$-topos $\Shv_\et(BG)$.
\end{rmk}

\begin{lem}
The category $\Shv_\proet(BG)$ is generated by coherent objects of homotopy dimension $0$, provided that the implicit cardinality bound is sufficiently large.
\end{lem}
\begin{proof}
Since effective epimorphisms in $\Pro(\Fin_G)$ are cofiltered limits of effective epimorphisms in $\Fin_G$ \cite[Proposition 6.1.15(3)]{lurie2018ultracategories}, the effective epimorphism topology on $\pFin_G$ coincides with what is called the \emph{transfinite topology} in \cite[\S5]{kerz2015transfinite}.
Hence the result is a special case of \cite[Theorem 4.2]{kerz2015transfinite}.
\end{proof}

\begin{exm} \label{ex:pi*-pro}
Let $S$ be a scheme, $\bar s \in S$ a geometric point and $\pi_1^\et(S,\bar s) \to G$ a continuous homomorphism.
There is a canonical functor $\pi^*: \Fin_G \to \Et_S$ inducing a geometric morphism $\pi_G^*: \Shv_\et(BG) \to \Shv(S_\et)$.
Taking solidifications (Remark \ref{rmk:solidification-G} and \cite[Example 3.3.11]{barwick2019pyknotic}) we obtain an induced functor \[ \pi^*: \Shv_\proet(BG) \to \Shv(S_\proet). \]
\end{exm}

The ``stalk'' functor is particularly well-behaved in the pro-$G$ setting:
\begin{lem} \label{lem:pro-stalk}
The functor $\varpi^*: \Shv_\proet(BG) \to \Shv_\proet(B\{e\})$ is conservative and admits a left adjoint.
Similarly for spectra.
\end{lem}
\begin{proof}
We can view $G$ as an object of $\pFin_G$ and then $G \to *$ is a covering.
All results follow from the identification $(\pFin_G)_{/G} \wequi \pFin_{\{e\}}$.
\end{proof}

\begin{lem} \label{lemm:spc-condensed-ff}
The canonical functor $\Spc \to \Shv_\proet(B\{e\})$ is fully faithful.
The same is true for $\ell$-completed and/or stabilized versions.
\end{lem}
\begin{proof}
Let $\scr W \subset \pFin$ denote the full subcategory on $w$-contractible objects.
Then $\Fin \subset \scr W$ and $\Spc \to \Shv_\proet(B\{e\})$ identifies with $\PSh_\Sigma$ applied to this inclusion.
Write \[ c: \Spc \wequi \PSh_\Sigma(\Fin) \adj \PSh_\Sigma(\scr W) \wequi \Shv_\proet(B\{e\}): c^* \] for the adjunction.
Then $\id \to c^*c$ is an isomorphism on $\Fin$ and hence on all of $\Spc$, since $c^*$ preserves sifted colimits and $\Spc$ is generated under sifted colimits by $\Fin$.
This is the first fully faithfulness statement.
For the $\ell$-complete version, it suffices to prove that if $X \in \Spc$ is $\ell$-complete then $X \wequi c^* L_\ell c X$.
This follows from Proposition \ref{prop:p-comp-PSigma}.
The argument in the stable case is similar.
\end{proof}

We record for future reference the following two results.
\begin{lem} \label{lemm:BG-proet-finite-vcd}
Suppose that $G$ has finite virtual $\ell$-cohomological dimension.
Then \[ \ShvSp(B_\et G)_\ell^\comp \to \ShvSp(B_\proet G)_\ell^\comp \] is fully faithful, with essential image those objects which have classical mod $\ell$ homotopy sheaves.
\end{lem}
\begin{proof}
Repeat the proofs of \cite[Lemma 3.4, Proposition 3.5]{bachmann-SHet}.
\end{proof}
\begin{lem} \label{lemm:classical-proet-locally}
Let $F \in \Shv_\proet(BG)_{\le 0}$.
If $F$ is classical pro-étale locally (e.g. $\varpi^*(F)$ is classical) then so is $F$.
\end{lem}
\begin{proof}
Repeat the proof of \cite[Lemma 5.1.4]{bhatt2013pro}.
\end{proof}

\subsection{Main results} \label{subsec:Ket-main}
We denote by $K \in \PSh(\Sch)$ or $\PSh_\ShvSp(\Sch)$ the presheaf of $K$-theory spaces (or connective spectra).
We also write $K$ for its restriction to any subcategory of $\Sch$, so e.g. we write $K \in \PSh(\Sm_S)$, and so on.
We denote by $K^\circ \in \PSh(\Sch)$ the rank zero summand of $K$, or in other words the connected component of the identity.

\begin{lem} \label{lem:K-extended}
Consider $K \in \PSh(\Sm_S)$.
We have $L_\Zar \nu^*(K) \wequi K \in \PSh(\Sch_S)$.
In particular $\nu^*(L_\et K) \wequi L_\proet(K)$ when viewing $\nu^*$ as a functor $\Shv_\et(\Sm_S) \to \Shv_\proet(\Sch_S)$.

The same statements hold for $K^\circ$ in place of $K$.
\end{lem}
\begin{proof}
The second statement follows from the first, and the first statement follows from the fact that algebraic $K$-theory of commutative rings is left Kan extended from smooth affine schemes \cite[Example A.0.6(1)]{EHKSY3} and Zariski descent.
The analogous statements for $K^\circ$ follow since the relevant functors are exact and so commute with taking connected covers.
\end{proof}

From now on we suppose that $1/\ell \in S$.
In this case $\mu_{\ell^n}$ is étale over $S$ and these extensions induce the cyclotomic character \[ \pi^*: \Shv_\et(B\Z_\ell^\times) \to \Shv(S_\et) \to \Shv_\et(\Sm_S), \] as well as stable or similar variants.
Recall that evaluating on the stalk and $\ell$-completing yields a fully faithful embedding with understood image (Lemma \ref{lem:describe-SHetp-BG}) \[ \ShvSp(B_\et \Z_\ell^\times)_\ell^\comp \to \ShvSp^{B(\mu \times \Z)}. \]
Recall also that by Goerss--Hopkins obstruction theory \cite[Corollary 7.7]{goerss2004moduli} we have \[ \Map_{\CAlg(\ShvSp)}(\KU_\ell^\comp, \KU_\ell^\comp) \wequi \Hom_{\CAlg^\heart}(\pi_* \KU_\ell^\comp, \pi_* \KU_\ell^\comp) \wequi \Z_\ell^\times. \]

\begin{cnstr}
Using the equivalence $\CAlg(\ShvSp^{BG}) \wequi \CAlg(\ShvSp)^{BG}$ and the above discreteness statement, we may lift $\KU_\ell$ to $\CAlg(\ShvSp^{B(\mu \times \Z)})$.
By construction, the resulting object lies in the essential image of the $\ell$-completed stalk functor.
We consequently obtain \[ \KU_\ell \in \CAlg(\ShvSp(B_\et\Z_\ell^\times)_\ell^\comp). \]
Taking the connective cover, we obtain \[ \ku_\ell \in \CAlg(\ShvSp(B_\et\Z_\ell^\times)_\ell^\comp). \]
\end{cnstr}

\begin{prop} \label{prop:K-pulled-back}
There is a canonical equivalence \[ \pi^*(\ku_\ell)_\ell^\comp \wequi L_\et(K)_\ell^\comp \in \CAlg(\ShvSp(S_\et)). \]
\end{prop}
\begin{proof}
There is a canonical map $\pi^*(\KU_\ell) \to L_{K(1)} K \in \CAlg(\ShvSp(S_\et))$ constructed in \cite[Theorem 3.9]{bhatt2020remarks}.\NB{construct first $\KU_\ell \to \pi_* L_{K(1)} K$}
Recall the pro-$\ell$-truncation transformation from \S\ref{subsec:pro}.
Consider the following diagram
\begin{equation*}
\begin{CD}
\pi^*(\ku_\ell)_\ell^\comp @>a>> \pi^*(\KU_\ell)_\ell^\comp @>c>> L_\et(L_{K(1)} K)_\ell^\comp @<b<< L_\et(K)_\ell^\comp \\
@AfAA                                        @AAA                                     @AAA                                    @AgAA             \\
\tau_{\ge 0}^\ell(\pi^*(\ku_\ell)) @>>> \tau_{\ge 0}^\ell(\pi^*(\KU_\ell)) @>>> \tau_{\ge 0}^\ell(L_\et(L_{K(1)} K)) @<<< \tau_{\ge 0}^\ell(L_\et(K)). \\
\end{CD}
\end{equation*}
This is a diagram of hypercomplete, $\ell$-adic étale sheaves of spectra.
The map $c$ is the completion of the canonical comparison map.
The maps $a$ and $b$ are obtained from $\ku_\ell \to \KU_\ell$ and $K \to L_{K(1)} K$.
The bottom row is obtained from the top one by taking pro-$\ell$-(connective covers).
Since all the relevant functors are lax symmetric monoidal, this is even a diagram of $\scr E_\infty$ rings.
We shall show that the maps $f$ and $g$ are equivalences, as are all the maps in the bottom row.
The circuit from top left to top right via the bottom then yields the desired equivalence.
Let $n \ge 0$ and consider the diagram
\begin{equation*}
\begin{CD}
\pi^*(\ku_\ell)/\ell^n @>>> \pi^*(\KU_\ell)/\ell^n @>>> L_\et(L_{K(1)} K)/\ell^n @<<< L_\et(K)/\ell^n \\
@AAA                                        @AAA                                     @AAA                                    @AAA             \\
\tau_{\ge 0}(\pi^*(\ku_\ell)/\ell^n) @>>> \tau_{\ge 0}(\pi^*(\KU_\ell)/\ell^n) @>>> \tau_{\ge 0}(L_\et(L_{K(1)} K)/\ell^n) @<<< \tau_{\ge 0}(L_\et(K)/\ell^n) \\
\end{CD}
\end{equation*}
obtained in a similar way.
By definition, the original diagram is the limit over $n$ of the new diagrams; it hence suffices to prove the analogous claims for the new diagram.
Since we are dealing with hypercomplete sheaves (by definition), it suffices to check the equivalences on stalks.
Passage to stalks commutes with $\tau_{\ge 0}$ and with $/\ell^n$, and ignores $L_\et$.
It also commutes with $L_{K(1)}(\ph)/\ell^n$, since this functor commutes with filtered colimits.\NB{ref?}
For the same reason passage to stalks commutes with $K$ \cite[Proposition 7.10]{blumberg2013universal}.
If $A$ is a strictly Henselian local ring, stalk of some (smooth) $S$-scheme, the induced diagram on this stalk thus is
\begin{equation*}
\begin{CD}
\ku/\ell^n @>>> \KU/\ell^n @>>> L_{K(1)} K(A)/\ell^n @<<< K(A)/\ell^n \\
@AAA                @AAA                     @AAA                                    @AAA             \\
\tau_{\ge 0}(\ku/\ell^n) @>>> \tau_{\ge 0}(\KU/\ell^n) @>>> \tau_{\ge 0}(L_{K(1)} K(A))/\ell^n) @<<< \tau_{\ge 0}(K(A)/\ell^n). \\
\end{CD}
\end{equation*}
We have $K(A)/\ell^n \wequi \ku/\ell^n$ by the theorems of Gabber \cite{gabber-rigidity} and Suslin \cite{suslin-rigidity,SUSLIN1984301}, which implies the desired equivalences (recall that $L_{K(1)} \ku \wequi \KU$\NB{ref}).
\end{proof}

Consider $\ku_\ell^\comp \in \ShvSp^{B(\mu \times \Z)}$.
This is a connective object in the usual $t$-structure, and $\pi_0$ is given by the Eilenberg-MacLane spectrum $\Z_\ell^\comp$ with the trivial action.
Note that (a) the induced $\Z_\ell^\times$-action is continuous, so this is in the essential image of $\ShvSp(B_\et\Z_\ell^\times)_\ell^\comp$ and (b) this is just the $\ell$-adic completion of the constant sheaf $\Z$ in $\ShvSp^{B(\mu \times \Z)}$, whence the same is true in $\ShvSp(B_\et \Z_\ell^\times)$.
All in all we have constructed a morphism \[ \ku_\ell \to \Z_\ell^\comp \in \CAlg(\ShvSp(B_\et\Z_\ell^\times)_\ell^\comp). \]

\begin{dfn} \label{dfn:bul}
We put \[ \bu_\ell = \fib(\ku_\ell \to \Z_\ell^\comp) \in \ShvSp(B_\et \Z_\ell^\times), \quad \widetilde\BU_\ell = \Omega^\infty(\bu_\ell) \in \CMon(\Shv_\et(B\Z_\ell^\times)) \] and \[ \BU_\ell := \Omega^\infty(\nu^*(\bu_\ell)_\ell^\comp) \in \CMon(\Shv_\proet(B\Z_\ell^\times)). \]
\end{dfn}
The commutative monoid structures arise just because we are looking at the infinite loop space of a sheaf of spectra.

\begin{lem} \label{lemm:SH-BZl-connectivity}
Let $E \in \ShvSp(B_\et \Z_\ell^\times)_\ell^\comp$ such that the completed stalk $\varpi^*(E)$ is in $\ShvSp_{\ge 0}$.
Then $E \in \ShvSp(B_\et\Z_\ell^\times)_{\ge -1}$.
\end{lem}
\begin{proof}
We prove the claim on stalks; i.e. that $\colim_H \map(\Z_\ell^\times/H, E) \in \ShvSp_{\ge -1}$, where the colimit is over open subgroups.
By construction $\mu \times \Z \to \Z_\ell^\times$ is dense and hence $\mu \times \Z \to \Z_\ell^\times/H$ is surjective.
It follows that $\Z_\ell^\times/H \wequi (\mu \times \Z)/((\mu \times \Z) \cap H)$ as $(\mu \times \Z)$-sets.
Using \cite[Lemma 3.3]{bhatt2020remarks} we thus get $\map(\Z_\ell^\times/H, E) \wequi \varpi^*(E)^{h((\mu \times \Z) \cap H)}$.
For a cofinal family of subgroups $H$ we get $(\mu \times \Z) \cap H = \ell^n \Z \wequi \Z$, which has cohomological dimension $1$.
The result follows.
\end{proof}

\begin{lem} \label{lemm:bu-l-connectivity}
We have $\bu_\ell \in \ShvSp(B_\et \Z_\ell^\times)_{\ge 1}$.
\end{lem}
\begin{proof}
Immediate from Lemma \ref{lemm:SH-BZl-connectivity}, sinc $\bu_\ell^\comp \in \ShvSp_{\ge 2}$.
\end{proof}

\begin{cor} \label{cor:BU-ell-uncompl}
We have $L_\ell \nu^* \widetilde\BU_\ell \wequi \BU_\ell$.
\end{cor}
\begin{proof}
$\Omega^\infty(\nu^*(\bu_\ell)_\ell^\comp) \wequi L_\ell \Omega^\infty(\nu^* \bu_\ell)$ by Corollary \ref{cor:htpy-dim-0-compl}(2); now use that $\nu^*$ and $\Omega^\infty$ commute.
\end{proof}

Note that $K^\circ$ is a commutative monoid since it is an infinite loop sheaf; i.e. we are considering the additive monoid structure.
Recall also that $\pi^*: \Fin_{\Z_\ell^\times} \to \Et_S$ induces $\pi^*: \Pro(\Fin_{\Z_\ell^\times}) \to \Pro(\Et_S) \to S_\proet$ (see Example \ref{ex:pi*-pro}).
\begin{prop} \label{prop:etale-Kcirc-pulledback}
There are canonical equivalences \[ \pi^*(\nu^*\bu_\ell)_\ell^\comp \wequi L_\proet(K_{\ge 1})_\ell^\comp \in \ShvSp_\proet(\Sch_S) \] and \[ L_\ell \pi^*(\BU_\ell) \wequi \Omega^\infty L_\proet(K_{\ge 1})_\ell^\comp \wequi L_\ell L_\proet(K^\circ) \in \CMon(\Shv_\proet(\Sch_S)), \] and similarly in $\Shv(S_\proet)$.
\end{prop}
\begin{proof}
By definition and since $\pi^*, \nu^*$ commute we have \[ \pi^*(\nu^*\bu_\ell)_\ell^\comp \wequi \nu^*(\fib(\pi^*(\ku_\ell)_\ell^\comp \to L_\et(\Z)_\ell^\comp))_\ell^\comp. \]
We have $\pi^*(\ku_\ell)_\ell^\comp \wequi L_\et(K)_\ell^\comp$ by Proposition \ref{prop:K-pulled-back}, and hence $\fib(\pi^*(\ku_\ell)_\ell^\comp \to L_\et(\Z)_\ell^\comp) \wequi L_\et(K_{\ge 1})_\ell^\comp$.
Lemma \ref{lem:K-extended} shows that $\nu^*(L_\et(K_{\ge 1})) \wequi L_\proet(K_{\ge 1})$ and hence $\nu^*(\pi^* \bu_\ell)_\ell^\comp \wequi L_\proet(K_{\ge 1})_\ell^\comp$.
This is the first equivalence.

Since $\Omega^\infty$ commutes with $\pi^*$, by definition we get \[ L_\ell \pi^*(\BU_\ell) = L_\ell \pi^* \Omega^\infty(\nu^*(\bu_\ell)_\ell^\comp) \wequi L_\ell \Omega^\infty \pi^*(\nu^*(\bu_\ell)_\ell^\comp). \]
By Lemma \ref{lemm:bu-l-connectivity} we have $\bu_\ell \in \ShvSp(B_\et \Z_\ell^\times)_{\ge 1}$ and hence $\nu^*(\bu_\ell) \in \ShvSp(B_\proet\Z_\ell^\times)_{\ge 1}$.
Corollary \ref{cor:htpy-dim-0-compl}(2) now shows that $\nu^*(\bu_\ell)_\ell^\comp \in \ShvSp(B_\proet\Z_\ell^\times)_{\ge 1}$ and so $\pi^*(\nu^*(\bu_\ell)_\ell^\comp) \in \ShvSp_\proet(\Sch_S)_{\ge 1}$.
Using Corollary \ref{cor:htpy-dim-0-compl}(2) again we find $L_\ell \Omega^\infty \pi^*(\nu^*(\bu_\ell)_\ell^\comp) \wequi \Omega^\infty (\pi^*(\nu^*\bu_\ell)_\ell^\comp)$.
All in all we have found that \[ L_\ell \pi^*(\BU_\ell) \wequi \Omega^\infty(L_\proet(K_{\ge 1})_\ell^\comp), \] which is the second equivalence.
By one final application of Corollary \ref{cor:htpy-dim-0-compl}(2) this is equivalent to $L_\ell \Omega^\infty L_\proet K_{\ge 1} \wequi L_\ell L_\proet K^\circ$.
\end{proof}

\begin{rmk}
Combining the first and third equivalence we see that $L_\ell L_\proet(K^\circ) \wequi \Omega^\infty \pi^*(\nu^*\bu_\ell)_\ell^\comp$.
Proposition \ref{prop:strongly-finite-ff}(0) now shows that these sheaves are $\A^1$-invariant.
\end{rmk}

\section{Étale algebraic cobordism} \label{sec:etale-MGL}
\localtableofcontents

\bigskip
Let $S$ be a scheme with $1/\ell \in S$.
Under the equivalence $\SH_\et(S)_\ell^\comp \wequi \ShvSp(S_\et)_\ell^\comp$ of \cite{bachmann-SHet}, for any $E \in \SH(S)$, its $\ell$-adic étale localization $L_\et(E)_\ell^\comp$ determines (and is determined by) an étale sheaf of spectra on $S$.
For example, taking $E=\KGL$ the algebraic $K$-theory spectrum, the results of the previous section easily show that $L_\et(\KGL)_\ell^\comp \wequi \pi^* \KU_\ell$.
In this section we would like to repeat the discussion for $E=\MGL$, the algebraic cobordism spectrum.
That is, we would like to construct an object $\MU_\ell \in \ShvSp(B_\et \Z_\ell^\times)$ together with an equivalence $L_\et(\MGL)_\ell^\comp \wequi \pi^*(\MU_\ell)$.
We would in fact like to do this using the results from the previous section: we want to define $\MU_\ell$ as the Thom spectrum of $\BU_\ell$; then $\pi^*\MU_\ell$ should be the Thom spectrum of $\pi^* \BU_\ell \wequi \BGL_\ell^\comp$, which should be $\MGL$ (up to appropriate sheafification and completion).

While this strategy works in principle, some rather delicate issues arise.
To begin with, in order to ``take the Thom spectrum'' of a ``space'' $\BU_\ell$, we need a (commutative monoid) map $J_{\Z_\ell^\times}: \BU_\ell \to \BGL_1$.
Here the right hand side denotes the étale sheaf classifying spherical fibrations, which one verifies easily is constant (i.e. has trivial $\Z_\ell^\times$-action).
The underlying map of $J_{\Z_\ell^\times}$ is the classical $J$-homomorphism, and we are (essentially) asking to give this a $\Z_\ell^\times$-equivariant refinement.
The existence of such a refinement is closely related to the \emph{stable Adams conjecture} \cite{stable-adams-1,stable-adams-2}.
For $\ell$ odd, the group $\Z_\ell^\times \wequi \Z/(\ell-1) \times \Z_\ell$ is topologically generated by one element.
This roughly speaking means that in order to produce a genuine $\Z_\ell^\times$-action on $J_{\Z_\ell^\times}$ we need only act by a single automorphism with no further relations (see Lemma \ref{lemm:BGl-ff} for details).
This we can achieve by utilizing a motivic proof of the stable Adams conjecture (see Remark \ref{rmk:action-alternative-construction}).
For $\ell=2$ though, the group $\Z_2^\times \wequi \Z/2 \times \Z_2$ is not topologically generated by one element, and analogous constructions are much more difficult: even if we can make $J$ equivariant for the two generators, we would also have to account for the fact that the generators commute.
We are saved by Theorem \ref{thm:etale-fund-gps} which ensures that the $2$-adic cyclotomic characters actually always factor through the group $G_2 := \Z/2 *_2 \Z_2$.
Replacing $\Z_2^\times$ by $G_2$, the two generators now do not interact anymore, and we need not account for their commutation.
This allows us to make the map $J$ equivariant by similar arguments as before.

Even with this issue out of the way, we still need to prove that $\pi^*(J)$ is homotopic to an appropriate version of the motivic $J$-homomorphism, ideally as a highly structured commutative monoid map.
While we believe this to be true, arguments we know for proving this would lead too far afield.
For this reason we present a simple argument showing that the two maps are homotopic as $\scr E_0$-monoid maps (i.e. maps of pointed spaces).
From this we can deduce that $L_\et(\MGL)_\ell^\comp$ is equivalent to $\pi^*(\MU_\ell)$ as $\scr E_0$-rings, which (with some care) is enough for our desired applications.

\subsubsection*{Organization}
We begin in \S\ref{subsec:spherical-fib} by defining the pro-étale sheaf classifying $\ell$-adic spherical fibrations, determining its homotopy groups and showing that it is constant.
Using this, in \S\ref{subsec:etale-mot-j} we show that an étale sheafified version of the $\ell$-adic motivic $J$-homomorphism factors through the $\ell$-completion of its source, namely $\pi^*(\BU_\ell)$.
Then in \S\ref{subsec:equivariant-j} we construct a $G_\ell$-equivariant analog of the $J$-homomorphism.
We prove in \S\ref{subsec:J-E0} that the $\ell$-adic motivic $J$-homomorphism is obtained by pullback from the $G_\ell$-equivariant one, at least as $\scr E_0$-monoid map.
From this we deduce in \S\ref{subsec:etale-MGL} our main result, that $\ell$-adic étale algebraic cobordism is pullback from $\MU_\ell$ (see Theorem \ref{thm:identify-MGL}).

We conclude with two complementary sections.
In \S\ref{subsec:more-MUl} we sketch a theory of $G_\ell$-equivariant complex orientations, and show that $\MU_\ell$ satisfies a universal property with respect to those.
We deduce that $\pi^*(\MU_\ell) \wequi L_\et(\MGL)_\ell^\comp$ as homotopy rings.
Finally in \S\ref{subsec:MGL-regular-primes} we show how to strengthen this equivalence to one of $\scr E_\infty$-rings in certain restricted cases.

\subsubsection*{Standing assumptions}
Throughout we assume that $S$ is a scheme on which $\ell$ is invertible.

\subsection{Spherical fibrations} \label{subsec:spherical-fib}
\begin{dfn}
We denote by \[ \BGL_1(\1_{\proet,\ell}^\comp) \in \PSh(\Sch_S) \] the subpresheaf of $(\SH_\proet(\ph)_\ell^\comp)^\wequi$ corresponding to those objects which are pro-étale locally equivalent to the unit.
\end{dfn}
\begin{prop} \label{prop:BGL1-complete}
Let $\scr L'$ be the set of primes dividing $\ell-1$ and $\scr L = \scr L' \cup \{\ell\}$.

$\BGL_1(\1_{\proet,\ell}^\comp)$ is an unstably $\scr{L}$-complete pro-étale hypersheaf.
In fact it is equivalent to $L_\scr{L} \pi^*(\BGL_1(\1_\ell^\comp))$, where $\pi^*: \Spc \to \Shv_\proet(\Sch_S)$ is the unique cocontinuous functor.
Moreover \[ \BGL_1(\1_{\proet,\ell}^\comp) \wequi L_{\scr L'} \BGL_1(\1_{\proet,\ell}^\comp) \times L_\ell \BGL_1(\1_{\proet,\ell}^\comp) \quad\text{and}\quad L_{\scr L'} \BGL_1(\1_{\proet,\ell}^\comp) \wequi K(\Z/(\ell-1), 1). \]
\end{prop}
For the proof and some of the following arguments, we use the following fact.\NB{move this elsewhere?}
For a scheme $X$, write $|X|$ for the underlying topological space.
For another topological space $Y$, write $C(|X|, Y)$ for the set of continuous maps from $|X|$ to $Y$.
Now if $A$ is a discrete abelian group, then the associated pro-étale sheaf satisfies \begin{equation}\label{eq:aproet} (a_\proet A)(X) = C(|X|, A). \end{equation}
Indeed this follows from Lemma \ref{prop:htpy-dim-0-PSigma} in case $X$ is $w$-contractible and \cite[Lemma 3.8]{bachmann-SHet} in general.
\begin{proof}
Since $\SH_\proet(\ph)_\ell^\comp$ is a pro-étale hypersheaf by Lemma \ref{lem:SH-proet-descent}, to see that the subpresheaf $\BGL_1(\1_{\proet,\ell}^\comp)$ is a hypersheaf it suffices to see that the defining condition is pro-étale local, which it is by construction.
The functor $\ShvSp_\ell^\comp \to \SH_\proet(S)_\ell^\comp$ induces a canonical map $\pi^* (\BGL_1(\1_\ell^\comp)) \to \BGL_1(\1_{\proet,\ell}^\comp)$, which we shall show is an $\scr L$-completion.
By Corollary \ref{cor:htpy-dim-0-compl} it suffices to show that this holds on sections over $W \in \Sch_S$ $w$-contractible.

The homotopy groups of $\BGL_1(\1_\ell^\comp)$ are given by $*$ in degree $0$, $\Z_\ell^\times$ in degree $1$, and $\pi_i \1_\ell^\comp = \pi_i^s \otimes \Z_\ell$ in degree $i+1$, $i \ge 1$.
Consequently by \eqref{eq:aproet}, the homotopy groups of $\pi^* (\BGL_1(\1_\ell^\comp))(W)$ are given by $*$, $C(|W|, d\Z_\ell^\times)$ and $C(|W|, \pi_i^s \otimes \Z_\ell)$, where $d\Z_\ell^\times$ denotes $\Z_\ell^\times$ with the discrete topology.
Except for $\pi_1$, all of these are bounded $\scr L$-torsion and hence derived $\scr L$-complete.
Even on $\pi_1$ the torsion is bounded, and so its derived $\scr L$-completion coincides with the classical one.
Using exactness of $C(|W|, \ph)$ on discrete abelian groups (see e.g. \cite[proof of Corollary 2.18]{bachmann-eta}) this works out as \[ \lim_n C(|W|, d\Z_\ell^\times)/s^n \wequi \lim_n C(|W|, d\Z_\ell^\times/s^n) \wequi C(|W|, \Z_\ell^\times), \] where $s=\ell(\ell-1)$.
Thus we have shown that the homotopy groups of $L_\scr{L} \pi^* (\BGL_1(\1_\ell^\comp))(W)$ are given by $*, C(|W|, \Z_\ell^\times)$ and $C(|W|, \pi_i^s \otimes \Z_\ell)$, and we need to prove that the same is true for the homotopy groups of $\BGL_1(\1_{\proet,\ell}^\comp)(W)$.
$w$-contractibility implies that this is just $B\Aut_{\SH_\proet(W)_\ell^\comp}(\1)$, which by Proposition \ref{prop:strongly-finite-ff}(1) is the same as $B := B\Aut_{\ShvSp(W_\proet)_\ell^\comp}(\1)$.
Let $X = \End_{\ShvSp(W_\proet)}(\1)$.
By Corollary \ref{cor:htpy-dim-0-conn}(4) and \eqref{eq:aproet} we deduce that $\pi_i X = C(|W|, \pi_i^s)$.
Again by bounded $\ell$-torsion, from this we get $\pi_i X_\ell^\comp \wequi C(|W|, \pi_i^s \otimes \Z_\ell)$.
Since $\pi_0 B = *$, $\pi_1 B = \pi_0(X_\ell^\comp)^\times$ and $\pi_{i+1}B = \pi_i X_\ell^\comp$ for $i \ge 1$, we see that $B$ has the desired homotopy groups.

For the final statement, using Lemma \ref{lem:L-splitting}, it suffices to prove that $\BGL_1(\1_\ell^\comp) \wequi K(\Z/(\ell-1), 1) \times L_\ell \BGL_1(\1_\ell^\comp)$, which is clear.
\end{proof}

\begin{rmk} \label{rmk:BGL-htpy-sheaves}
In fact we have determined (in the course of the proof) all the homotopy sheaves of $\BGL_1(\1_{\proet,\ell}^\comp)$: $\ul{\pi}_0 = *$, $\ul{\pi}_1 \wequi C(|\ph|, \Z_\ell^\times)$ and $\ul{\pi}_{i+1} \wequi C(|\ph|, \pi_i^s \otimes \Z_\ell)$ for $i \ge 1$.
\end{rmk}

\begin{cor} \label{cor:BGL1-trunc}
We have \[ \tau_{\ge 2} \BGL_1(\1_{\proet,\ell}^\comp) \wequi L_\ell \pi^* \tau_{\ge 2} \BGL_1(\1) \in \Shv_\proet(\Sch_S). \]
\end{cor}
\begin{proof}
Compare homotopy sheaves using Remark \ref{rmk:BGL-htpy-sheaves} and Corollary \ref{cor:htpy-dim-0-compl}.
\end{proof}

\subsection{Étale motivic $j$-homomorphisms} \label{subsec:etale-mot-j}
The motivic $j$-homomorphism \cite[\S16.2]{norms} induces a map \[ K \to \SH(\ph) \in \PSh(\Sch_S). \]
Composing with $\SH(\ph) \to \SH_\proet(\ph)_\ell^\comp$ and passing to the rank zero summand in $K$-theory we obtain \[ K^\circ \to \BGL_1(\1_{\proet,\ell}^\comp) \in \PSh(\Sch_S). \]

\begin{prop} \label{prop:J-factors}
The map $K^\circ \to \BGL_1(\1_{\proet,\ell}^\comp) \in \CMon(\PSh(\Sch_S))$ factors (uniquely up to homotopy) through $\tau_{\ge 2}\BGL_1(\1_{\proet,\ell}^\comp)$.
\end{prop}
\begin{proof}
We know that $\ul\pi_0 K^\circ = * = \ul\pi_0 \BGL_1(\1_{\proet,\ell}^\comp)$, whereas $\ul\pi_1 K^\circ = \Gm$ and $\ul\pi_1 \BGL_1(\1_{\proet,\ell}^\comp) = C(|\ph|, \Z_\ell^\times)$.
It follows that $\Map(K^\circ, \tau_{\le 1} \BGL_1(\1_{\proet,\ell}^\comp))$ is discrete and \[ \Map(K^\circ, \tau_{\ge 2}\BGL_1(\1_{\proet,\ell}^\comp)) \to \Map(K^\circ, \BGL_1(\1_{\proet,\ell}^\comp)) \] is a monomorphism onto the subspace of those maps which are trivial on $\ul\pi_1$.
There will thus be a lift, unique up to homotopy, if (and only if) the $j$-homomorphism is trivial on $\ul\pi_1$.
Given $u \in \Gm(X) = \scr O_X(X)^\times$ the $j$-homomorphism yields $J(u) \in [\1, \1]_{\SH_\proet(X)_\ell^\comp}$.
We need to prove that the image of $J(u)$ in $C(|X|, \Z_\ell^\times)$ is $1$.
For any topological space $Y$ and topological abelian group $A$, we have $C(Y,A) \hookrightarrow \prod_{y \in Y}A$.
Naturality of $J(u)$ thus implies that we may assume that $X$ is the spectrum of an algebraically closed field.
We have a factorization \[ J: \Gm(X) \xrightarrow{J_0} [\1, \1]_{\SH(X)} \to C(X, \Z_\ell^\times). \]
It is consequently enough to show that $J_0(u) = 1 \in \pi_0(\1_{\SH(X)})$.
But $\pi_0(\1_{\SH(X)}) = GW(X) = \Z$ since $X$ is the spectrum of an algebraically closed field, and $J_0(u) = \lra{u}=1$ essentially by construction.
\end{proof}

Since the new target is an $\ell$-complete hypersheaf (e.g. by Corollary \ref{cor:BGL1-trunc}), we may sheafify and complete the source.
This way we obtain \[ J_{\proet,\ell}^\comp: L_\ell \pi^* \BU_\ell \wequi L_\ell L_\proet(K^\circ) \to \tau_{\ge 2}\BGL_1(\1_{\proet,\ell}^\comp) \wequi L_\ell \pi^* \tau_{\ge 2} \BGL_1(\1) \in \CMon(\Shv_\proet(\Sch_S)). \]
By Corollary \ref{cor:htpy-dim-0-compl} we have $L_\ell L_\proet(K^\circ) \wequi \Omega^\infty L_\proet(K_{\ge 1})_\ell^\comp$.
Since $\Omega^\infty: \ShvSp_\proet(\Sch_S) \to \CMon(\Shv_\proet(\Sch_S))$ is fully faithful onto the grouplike objects (see e.g. \cite[\S3.1.4]{EHKSY}), $J_{\proet,\ell}^\comp$ thus takes the form $\Omega^\infty j_{\proet,\ell}^\comp$, for a morphism \[ j_{\proet,\ell}^\comp: \pi^*(\bu_\ell)_\ell^\comp \wequi L_\proet(K_{\ge 1})_\ell^\comp \to \pi^*((\bgl_1 \1)_{\ge 2})_\ell^\comp \in \ShvSp_\proet(\Sch_S). \]
Here $\bgl_1 \1 \in \ShvSp$ denotes the classical unit spectrum.

\begin{rmk} \label{rmk:J-stable-basechange}
The morphisms $J_{\proet,\ell}^\comp$ and $j_{\proet,\ell}^\comp$ are stable under base change, since their source and target are and the motivic $j$-homomorphisms are natural.
\end{rmk}

Let $S_0=\Spec \Z[1/\ell]$.
By Proposition \ref{prop:strongly-finite-ff}(2), the functor $\nu^*: \ShvSp((S_0)_\et)_\ell^\comp \to \ShvSp_\proet(\Sch_{S_0})_\ell^\comp$ is fully faithful.
It follows that $j_{\proet,\ell}^\comp$ must be of the form $\nu^* j_{\et,\ell,S_0}^\comp$.
Via base change we obtain \[ j_{\et,\ell}^\comp: \pi^*(\bu_\ell)_\ell^\comp \to \pi^*((\bgl_1 \1)_{\ge 2})_\ell^\comp \in \ShvSp(S_\et). \]
By Remark \ref{rmk:J-stable-basechange}, we have $j_{\proet,\ell}^\comp \wequi \nu^*j_{\et,\ell}^\comp$ over any $S$.

\subsection{Equivariant $j$-homomorphisms} \label{subsec:equivariant-j}
Recall the groups $G_\ell$ from Definition \ref{def:modified-cyc-char}.
Let $G_2' = \Z/2 * \Z$, and $G_\ell' = \Z$ for $\ell$ odd.
We have canonical maps $G_\ell' \to G_\ell$ (at odd $\ell$ corresponding to the diagonal $\Z \to \Z/(\ell -1) \times \Z_\ell \wequi \Z_\ell^\times$, and at $\ell=2$ corresponding to the map to the pro-$2$-completion).
\begin{lem} \label{lemm:BGl-ff}
The canonical symmetric monoidal functor \[ \ShvSp(B_\et G_\ell)_\ell^\comp \to (\ShvSp^{BG_\ell'})_\ell^\comp \] is fully faithful, with essential image those spectra $E$ such that the action of $G_\ell'$ on $\pi_*(E/\ell)$ extends continuously to $G_\ell$.
Similarly for $G'$ a free group on finitely many generators, and $G$ its pro-$\ell$ completion.
\end{lem}
\begin{proof}
For $\ell$ odd, the first statement is Example \ref{ex:SH-BG}(4).

For $\ell=2$, we verify the assumptions of Lemma \ref{lem:describe-SHetp-BG}.
Write $q: P=\Z/2 * \Z \to \Z/2 *_2 \Z_2=G$ for the canonical map.
Since this is a pro-$2$-completion, it suffices to exhibit a cofinal family of subgroups $K_i' \subset P$ of $2$-power indices such that $BK_i'$ is a finite complex of bounded dimension, and $\colim_i H^*(BK_i', \Z/2) = 0$ for $*>0$.

Note that $S^1 \vee \RR\P^\infty$ admits $S^1 \vee S^1$ as a covering space of degree $2$ (attach two copies of $S^1$ at antipodal points of $S^n$ and let $n \to \infty$).
Equivalently, $\Z/2 * \Z$ admits a subgroup of index $2$ isomorphic to $\Z * \Z$.\NB{Surely generated by $b$ and $aba$, where $a,b$ are the generators of $\Z/2,\Z$...?}
We may thus replace $P$ by $\Z * \Z$.
Thus we have arrived at a special case of the last assertion.

In the case of the last assertion, any subgroup of finite index is finitely generated free (by the Nielsen--Schreier formula), and hence it suffices to prove the following: if $X = S^1 \vee \dots \vee S^1$ is a finite bouquet of circles, and $e_i \in H^1(X, \Z/\ell)$ is one of the canonical basis elements, then there exists a covering space $X' \to X$ of $\ell$-power degree such that $H^*(X, \Z/\ell) \to H^*(X', \Z/\ell)$ annihilates $e_i$.
Think of $X$ as obtained from $S^1$ corresponding to $e_i=e$ by attaching $n$ copies of $S^1$ around it.
Let $X'$ be obtained from $S^1$ by attaching $\ell n$ copies of $S^1$.
Give $X'$ the $C_\ell$ action which acts by the $\ell$-th power on the central circle and ``moves along'' the attached circles.\NB{better description...}
Then $X' \to X'/C_\ell \wequi X$ is an $\ell$-fold cover which induces the $\ell$-th power map on the central circle and hence annihilates $e$.
This concludes the proof.
\end{proof}

Recalling (e.g. from \cite[Lemma 3.8]{dwyer1980simplicial}) that $BG \vee BK \wequi B(G * K)$ we see that an object of $(\ShvSp^{BG_2'})_2^\comp$ just consists of a spectrum $E \in \ShvSp_2^\comp$ together with maps $B\Z/2, B\Z \to B\Aut(E)$, i.e. a $\Z/2$-action and an automorphism of $E$.
Similarly an object of $(\ShvSp^{BG_\ell'})_\ell^\comp$ is just $E \in \ShvSp_\ell^\comp$ together with an automorphism.
\begin{cnstr}
Consider the classical stable $j$-homomorphism \[ j: \bu \to \bgl_1 \1 \in \ShvSp. \]
We can give this a $\Z/2$-action coming from complex conjugation.
Moreover, after completing at a prime $\ell$, we can give it a $\Z$-action coming from the stable Adams conjecture.\footnote{That is, we want to act by the Adams operation on the source and trivially on the target, and the stable Adams conjecture shows that $j$ then enhances to an equivariant map---i.e. $j \circ \psi \wequi j$.}
Taking thom spectra, we obtain \[ \MU_\ell^\comp \in \CAlg((\ShvSp^{BG_\ell})^\comp). \]
\end{cnstr}
\begin{rmk} \label{rmk:action-alternative-construction}
An alternative construction of these actions and spectra is as follows.
In the following diagram
\begin{equation*}
\begin{tikzcd}
G_\ell' \ar[r, dotted] \ar[d] & Gal(\QQ) \ar[ld, "\tilde\pi"] \\
G_\ell
\end{tikzcd}
\end{equation*}
a lift exists as indicated.\NB{Using that $\pi$ is surjective and complex conjugation lifts the $\Z/2$}
Now consider the $Gal(\QQ)$-equivariant realization of $j_{\et,\ell}^\comp$ (respectively of $\MGL$) into $(\ShvSp^{B Gal(\QQ)})_2^\comp$ and restrict along the lift.
(In fact this is one way of \emph{proving} the stable Adams conjecture.)
\end{rmk}
Noting that $\bu \in \ShvSp_{\ge 2}$, the classical $j$ factors through $(\bgl_1 \1)_{\ge 2}$.
Using Lemma \ref{lemm:BGl-ff}, we thus obtain \begin{equation} \label{eq:jG} j_{G_\ell,\ell}^\comp: \bu_\ell \to \pi^*(\bgl_1\1)_{\ge 2})_\ell^\comp \in \ShvSp(B_\et G_\ell). \end{equation}
Moreover there is at most one object \begin{equation} \label{eq:MUl} \MU_\ell \in \CAlg(\ShvSp(B_\et G_\ell)_\ell^\comp) \end{equation} inducing the given actions on $\MU_\ell^\comp \in \CAlg(\ShvSp_\ell^\comp)$.
In fact this exists: take the Galois equivariant realization of $\MGL$ and restrict to $G_\ell'$, as in Remark \ref{rmk:action-alternative-construction}.

\subsection{$J_{\proet,\ell}^\comp$ as an $\scr E_0$-algebra map} \label{subsec:J-E0}
We denote by $\Z/\ell^n(1) \in \Shv_\et(B\Z_\ell^\times)$ the sheaf corresponding to $\mu_{\ell^n}$.
We also put \[ \Z/\ell^\infty(1) = \colim_n \Z/\ell^n(1) \in \Shv_\et(B\Z_\ell^\times), \] \[ H\Z_\ell(1) = \Sigma^{-1} (H\Z/\ell^\infty(1))_\ell^\comp \in \ShvSp(B_\et \Z_\ell^\times) \] and \[ \Z_\ell(1) = \lim_n \nu^* \Z/\ell^n(1) \in \Shv_\proet(B\Z_\ell^\times). \]
Beware that in contrast with $H\Z_\ell(1) \in \ShvSp(B_\proet\Z_\ell^\times)$, the object $H\Z_\ell(1) \in \ShvSp(B_\et \Z_\ell^\times)$ is not actually obtained by applying the functor $H$ to any sheaf of abelian groups on $B\Z_\ell^\times$.
\begin{lem} \label{lem:HZl-basics} \hfill
\begin{enumerate}
\item $L_\ell B \nu^*\Z/\ell^\infty(1) \wequi B^2 \Z_\ell(1) \in \CMon(\Shv_\proet(B\Z_\ell^\times))$
\item $\nu^*(H\Z_\ell(1))_\ell^\comp \wequi H\Z_\ell(1) \in \ShvSp(B_\proet\Z_\ell^\times)$
\item Under the embedding $\ShvSp(B_\et\Z_\ell^\times)_\ell^\comp \hookrightarrow \ShvSp^{B(\mu \times \Z)}$ of Lemma \ref{lem:describe-SHetp-BG}, the spectrum $H\Z_\ell(1)$ just corresponds to $H\Z_\ell$ with the weight one action.
\item $H\Z_\ell(1) \in \ShvSp(B_\et \Z_\ell^\times)_{\ge -1}$.
\item The canonical map $\pi^* \Z/\ell^\infty(1) \wequi \mu_{\ell^\infty} \to \Gm$ induces \[ L_\ell \pi^*\nu^* B\Z/\ell^\infty(1) \wequi L_\ell B\Gm \in \CMon(\Shv_\proet(\Sch_S)). \]
\end{enumerate}
\end{lem}
\begin{proof}
(1) The following argument is happening entirely in pro-étale sheaves.
We have natural projections $\Z_\ell(1) \to \Z/\ell^n(1)$, fitting into sequences of sheaves $\Z_\ell(1) \xrightarrow{\ell^n} \Z_\ell(1) \to \Z/\ell^n(1)$ which we claim are exact.
Indeed this may be checked after passing to a pro-étale covering whence we may choose an isomorphism $\Z_\ell(1) \wequi \Z_\ell$, inducing $\Z/\ell^n(1) \wequi \Z/\ell^n$, and the claim is clear.
Taking the colimit over $n$ of these exact sequences, we get an exact sequence $\Z_\ell(1) \to \Z_\ell(1)[1/\ell] \to \Z/\ell^\infty(1)$.
The result follows by completing and taking classifying spaces.

(2) Immediate from (1).

(3) Immediate from the definitions.

(4) Immediate from (3) and Lemma \ref{lemm:SH-BZl-connectivity}.

(5) Using Corollary \ref{cor:htpy-dim-0-compl}(1), we need to prove that if $W$ is $w$-contractible, then \[ L_\ell B \mu_{\ell^\infty}(W) \wequi L_\ell B \Gm(W). \]
Since $W$ has no non-trivial étale covers, $\Gm(W)$ is $\ell$-divisible, and $\mu_{\ell^\infty}(W)$ is its $\ell^\infty$-torsion subgroup.
Thus we have an exact sequence $0 \to \mu_{\ell^\infty}(W) \to \Gm(W) \to C \to 0$ where $C$ is uniquely $\ell$-divisible, and the result follows.
\end{proof}

\begin{lem} \label{lemm:construct-b}
There is a canonical map\NB{name?} \[ b: B\Z/\ell^\infty(1) \to \widetilde{\BU}_\ell \in \Shv_\et(B\Z_\ell^\times)_* \] such that the induced map \[ L_\ell B\Gm \stackrel{L.\ref{lem:HZl-basics}(6)}{\wequi} L_\ell \pi^* \nu^* B\Z/\ell^\infty(1) \xrightarrow{\pi^*\nu^*b} L_\ell \pi^* \BU_\ell \stackrel{P.\ref{prop:etale-Kcirc-pulledback}}{\wequi} L_\ell L_\proet(K^\circ) \] is the canonical one.
\end{lem}
\begin{proof}
It is equivalent to give a map $\Sigma^\infty B\Z/\ell^\infty \to \bu_\ell^\comp \in \ShvSp^{B(\mu \times \Z)}$.
In this category we have $\Sigma^\infty_+ B\Z/\ell^\infty[\beta^{-1}]_\ell^\comp \wequi \KU_\ell^\comp$ (by Snaith's theorem, see e.g. \cite[Theorem 3.6]{bhatt2020remarks}).
Hence we obtain the desired map as \[ \Sigma^\infty B\Z/\ell^\infty \wequi \tau_{\ge 1}  \Sigma^\infty B\Z/\ell^\infty \hookrightarrow \tau_{\ge 1} \Sigma^\infty_+ B\Z/\ell^\infty \to \tau_{\ge 1} \Sigma^\infty_+ B\Z/\ell^\infty[\beta^{-1}]_\ell^\comp \wequi \tau_{\ge 1} \KU_\ell^\comp \wequi \bu_\ell^\comp. \]
Equivalently, this map arises from $\Sigma^\infty B\Z/\ell^\infty \to \Sigma^\infty_+ B\Z/\ell^\infty[\beta^{-1}]_\ell^\comp \wequi \KU_\ell^\comp$ by taking pro-$\ell$-covers.

To show that the induced map is the canonical one, consider the commutative diagram in $\ShvSp_\proet(\Sch_S)$
\begin{equation*}
\begin{tikzcd}
\Sigma^\infty B\pi^*\nu^* \Z/\ell^\infty(1) \ar[r, "\wequi"] \ar[d] & \Sigma^\infty B\Gm \ar[d] \ar[rd, "c"] \\
\Sigma^\infty_+ B\pi^*\nu^* \Z/\ell^\infty(1) \ar[r]\ar[d] & \Sigma^\infty_+ B\Gm \ar[r] & K \ar[r, "w"] & L_{K(1)} K. \\
\pi^*\Sigma^\infty_+ B\nu^* \Z/\ell^\infty(1)[\beta^{-1}] \ar[urrr, "\wequi", bend right=10]
\end{tikzcd}
\end{equation*}
The maps marked $\wequi$ become equivalences upon $\ell$-adic completion.
The map $c$ is the canonical one.
The map $w$ becomes an equivalence when applying pro-$\ell$-covers; this was established in the proof of Proposition \ref{prop:K-pulled-back}.
Applying $\tau_{\ge 1}^\ell$ we obtain a new commutative diagram in which $w$ is an equivalence and the circuit from top left to $\tau_{\ge 1}^\ell L_\proet(K)_\ell^\comp \wequi \tau_{\ge 1}^\ell L_{K(1)} K$ (see again the proof of Proposition \ref{prop:K-pulled-back} for this equivalence) via the bottom and left is $\pi^*\nu^*(b)$.
The desired result follows.
\end{proof}

\begin{rmk} \label{rmk:uncomp-BZl}
We have the sheaves of spaces \[ \widetilde\BU_\ell, \widetilde{B^2\Z_\ell(1)} := \Omega^\infty \Sigma^2 H\Z_\ell(1) \in \Shv_\et(B\Z_\ell^\times). \]
Applying the functor $L_\ell \nu^*$ we obtain $\BU_\ell$ and $B^2 \Z_\ell(1)$; the former equivalence is Corollary \ref{cor:BU-ell-uncompl} and the latter is proved in the same way, using Lemma \ref{lem:HZl-basics}(4).
Now let $\varpi: \{e\} \to \Z_\ell^\times$ be the unique homomorphism.
Since $\Shv(B\{e\}) = \Spc$ has homotopy dimension $0$, we similarly get \[ L_\ell \varpi^* \widetilde\BU_\ell \wequi \Omega^\infty \varpi^*(\bu_\ell)_\ell^\comp \stackrel{(*)}{\wequi} \Omega^\infty \bu_\ell^\comp \wequi \BU_\ell^\comp, \] where the middle equivalence is by construction of $\bu_\ell$.
By the same reasoning we get $L_\ell \varpi^* \widetilde{B^2\Z_\ell(1)} \wequi B^2 \Z_\ell$.
\end{rmk}

We have the homomorphism $G_\ell \to \Z_\ell^\times$ (which is an isomorphism for $\ell$ odd) which we can use to transport the above constructions to $G_\ell$.
We do this without further comment from now on.

For a pointed space (or presheaf of spaces) $X$, denote by $F_*X$ the free monoid on $X$ with $*=0$.
\begin{thm} \label{thm:equiv-splitting}
There is a pointed map \[ s: \BU_\ell \to L_\ell F_* B^2 \Z_\ell(1) \in \Shv_\proet(BG_\ell)_* \] such that the composite \[ c: \BU_\ell \xrightarrow{s} L_\ell F_* B^2 \Z_\ell(1) \wequi L_\ell F_* B \nu^* \Z/\ell^\infty \xrightarrow{\nu^*(b^\dagger)} \BU_\ell \in \Shv_\proet(BG_\ell) \] is the identity.
(Here by $b^\dagger$ we mean the map induced by $b$ using that the target is a commutative monoid.)
\end{thm}
Before carrying out the proof, let us recall the classical analog: there exists a map \[ s_{cl}: \BU_\ell^\comp \to L_\ell F_* B^2 \Z_\ell \in \Spc_* \] such that the composite \[ \BU_\ell^\comp \xrightarrow{s_{cl}} L_\ell F_* B^2 \Z_\ell \xrightarrow{b_{cl}^\dagger} \BU_\ell^\comp \] is the identity.
For a construction see \cite[Theorems 2.1 and 3.2]{MR539791}.
\begin{proof}
Note first that it suffices to construct $s$ such that $c$ is an equivalence; indeed then we can replace $s$ by $sc^{-1}$.

Since $B^2 \Z_\ell(1)$ is connected, we have $F_* B^2 \Z_\ell(1) \wequi \Omega^\infty\Sigma^\infty B^2 \Z_\ell(1)$; this is just the identification of $\Omega^\infty\Sigma^\infty$ in a topos (see e.g. \cite[\S3.1.4]{EHKSY}).
Hence $L_\ell F_* B^2 \Z_\ell(1) \wequi \Omega^\infty \Sigma^\infty(B^2 \Z_\ell(1))_\ell^\comp$ by Corollary \ref{cor:htpy-dim-0-compl}.
The map $s$ thus corresponds to $s^\dagger: \Sigma^\infty \BU_\ell \to \Sigma^\infty(B^2 \Z_\ell(1))_\ell^\comp$.
By Remark \ref{rmk:uncomp-BZl}, to construct $s^\dagger$ it suffices to construct \[ s_0: \Sigma^\infty \widetilde \BU_\ell \to \Sigma^\infty(\widetilde{B^2 \Z_\ell(1)})_\ell^\comp \in \ShvSp(B_\et G_\ell). \]
Suppose we have found such a map.
By Lemma \ref{lem:pro-stalk} the functor \[ \varpi^*: \Shv_\proet(BG_\ell) \to \Shv_\proet(B\{e\}) \] is conservative and admits a left adjoint; since it also has a right adjoint it preserves $L_\ell$-complete objects by Example \ref{ex:biadj-pres-ell-comp}.
Hence in order to prove that the composite $c$ is an equivalence, it suffices to show the same for $\varpi^*(c) = \varpi^*(c)_\ell^\comp$.
Note that $\varpi^*(s_0)$ induces \[ \varpi^*(s_0)^\dagger: L_\ell \varpi^*(\widetilde \BU_\ell) \to \Omega^\infty \Sigma^\infty \varpi^*(L_\ell \widetilde{B^2 \Z_\ell(1)})_\ell^\comp. \]
Again by Remark \ref{rmk:uncomp-BZl}, the source is $\BU_\ell^\comp$ and the target is \[ \Omega^\infty \Sigma^\infty(B^2\Z_\ell)_\ell^\comp \wequi L_\ell F_* B^2 \Z_\ell \in \Spc. \]
In particular we can form the composite \[ c_0: \BU_\ell^\comp \xrightarrow{\varpi^*(s_0)^\dagger} L_\ell F_* B^2 \Z_\ell \xrightarrow{\varpi^*(b^\dagger)} \BU_\ell^\comp \in \Spc. \]
Note that $\nu^*(c_0) \wequi \varpi^*(c)$.
In particular, it suffices to show that $c_0$ is an equivalence.
By construction the map $\varpi^*(b^\dagger)$ is the canonical one, and hence it is enough to construct $s_0$ such that $\varpi^*(s_0)^\dagger$ is the standard splitting map, i.e. such that $\varpi^*(s_0)$ is the standard map.
The completed stalk functor \[ \varpi^*(\ph)_\ell^\comp: \ShvSp(B_\et G_\ell)_\ell^\comp \to (\ShvSp^{BG_\ell'})_\ell^\comp \] is fully faithful by Lemma \ref{lemm:BGl-ff}, and hence we arrive at the following: in order to prove the theorem, it is sufficient to lift the standard splitting map \[ s_{cl}^\dagger: \Sigma^\infty(\BU)_\ell^\comp \to \Sigma^\infty(B^2\Z)_\ell^\comp \] along the forgetful functor $\ShvSp^{BG_\ell'} \to \ShvSp$.
In other words, we need to make the classical splitting map $G_\ell'$-equivariant.

To do this, observe that by Lemma \ref{lemm:lift-splitting} below, there is a map $\Sigma^\infty B \GL \to \Sigma^\infty B \Gm \in \SH(\QQ)$ whose complex realization is $s_{cl}^\dagger$.
By Lemma \ref{lemm:galois-equiv-enhancement} we thus obtain a $G$-equivariant enhancement $s_G$ of the classical splitting map, where $G=Gal(\QQ)$.
To conclude we restrict along a lift of the cyclotomic character, as in Remark \ref{rmk:action-alternative-construction}.
\end{proof}

\begin{lem} \label{lemm:lift-splitting} \NB{say something about kleen?}
There is a map $s_{mot}^\dagger: \Sigma^\infty B \GL \to \Sigma^\infty B \GL_1 \in \SH(\QQ)$ such that $r_\CC(s_{mot}^\dagger) \wequi s_{cl}^\dagger$.
\end{lem}
\begin{proof}
We shall build the map by imitating Snaith's construction from \cite{MR539791}.
Write $NT_n \subset \GL_n$ for the normalizer of the maximal torus.
By \cite[\S1 and \S2]{MR4094419} we have Becker--Gottlieb transfer maps \[ \tau_n: \Sigma^\infty_+ B\GL_n \to \Sigma^\infty_+ BNT_n. \]
(Here by $BNT_n$ we mean the étale classifying space.)
These transfers are compatible as $n$ varies, see e.g. \cite[Remark 2.5]{carlsson2024motivic}.
For $n=0$, this shows in particular that we may pass to unpointed suspension spectra, and then to the colimit along $n$.
Thus we obtain \[ \tau: \Sigma^\infty B\GL \to \Sigma^\infty BNT. \]
Note that \[ BNT_n \wequi L_\et (B\GL_1)^{\times n}_{h\Sigma_n} \wequi D_n(B\GL_1), \] where $D_n$ is the motivic extended power functor (for the last equivalence, see \cite[Proposition 5.23]{bachmann-colimits}).
Using the normed structure on the infinite loop spectrum $\Omega^\infty \Sigma^\infty B\GL_1$ (e.g. use \cite[Lemma 7.11]{bachmann-MGM} for $G=e$) we obtain compatible maps \[ BNT_n \wequi D_n(B\GL_1) \to D_n(\Omega^\infty \Sigma^\infty B\GL_1) \to \Omega^\infty \Sigma^\infty B\GL_1. \]
Taking the colimit over $n$ and adjoining over produces \[ \Sigma^\infty BNT \to \Sigma^\infty B\GL_1. \]
Composing with $\tau$ yields the desired map.
\end{proof}

Given a category $\scr C$ and $X \in \scr C$, we denote by $\tilde \Sigma X \in \scr C_*$ the \emph{unreduced suspension}; in other words we form the pushout
\begin{equation*}
\begin{CD}
X @>>> * \\
@VVV @VVV \\
* @>>> \tilde\Sigma X
\end{CD}
\end{equation*}
in $\scr C$ and note that this has a canonical base point (coming from one of the two $*$ in the diagram).
This construction has the property that if $X \in \scr C_*$ then $\tilde\Sigma X \wequi S^1 \wedge X \in \scr C_*$.

Let $F: \scr C_* \to \scr D$ be a symmetric monoidal functor such that $F \tilde \Sigma X$ is invertible.
Then there is an induced morphism of $\scr E_1$-monoids \[ \Aut_{\scr C}(X) \xrightarrow{F\tilde\Sigma} \Aut_{\scr D}(F\tilde \Sigma X) \wequi \Aut_{\scr D}(\1_\scr{D}). \]
If moreover $X$ is a grouplike $\scr E_1$-monoid, then left multiplication induces $\Map_{\scr C}(*, X) \to \Aut_{\scr C}(X)$ and hence \[ \Map_{\scr C}(*, X) \to \Aut_{\scr D}(\1_\scr{D}) \in \Mon_{\scr E_1}(\Spc). \]

\begin{exm}
Take $\scr C = \Spc(S)$, $X = \Gm$, $\scr D = \SH(S)$.
Then $\tilde \Sigma \Gm \wequi S^{2,1}$ and so we obtain a morphism of $\scr E_1$-monoids $\Map_{\Spc(S)}(*, \Gm) \to \Aut_{\SH(S)}(\1)$.
This is natural in $S$, and hence yields \[ u_\Gm: L_\mot \Gm \to \ul{\Aut}(\1) \in \Mon_{\scr E_1}(\PSh(\Sm_S)). \]
\end{exm}

\begin{exm} \label{ex:u1(1)}
Take $\scr C = \Shv_\proet(B\Z_\ell^\times),$ $X = B\Z_\ell(1)$ and $\scr D = \ShvSp(B_\proet\Z_\ell^\times)_\ell^\comp$.
Pro-étale locally $X$ is equivalent to $B\Z_\ell \wequi L_\ell S^1$, and so invertibility holds.
Moreover since $\Z_\ell(1)$ is a commutative monoid so is $B\Z_\ell(1)$, and so we obtain \[ u_{B\Z_\ell(1)}: B\Z_\ell(1) \to \ul{\Aut}(\1_\ell^\comp) \in \Mon_{\scr E_1}(\Shv_\proet(B\Z_\ell^\times)). \]
\end{exm}

\begin{lem} \label{lem:pullback-fib}
The following diagram in $\Shv_\proet(\Sch_S)_*$ commutes
\begin{equation*}
\begin{tikzcd}
B\Gm \ar[r] & K^\circ \ar[r, "j"] & \BGL_1(\1_{\proet,\ell}^\comp) \\
B\mu_{\ell^\infty} \ar[u] \ar[r] & B^2 \pi^* \Z_\ell(1) \ar[ur, "Bu_{B\Z_\ell(1)}" swap].
\end{tikzcd}
\end{equation*}
\end{lem}
\begin{proof}
Since $B\mu_{\ell^\infty}$ and $\BGL_1(\1_{\proet,\ell}^\comp)$ are connected and grouplike, it suffices by \cite[Lemma 7.2.2.11(1)]{HTT} to prove that we obtain a commutative diagram of $\scr E_1$-monoids after applying $\Omega$.
Since the construction $u_X$ is natural in the monoid $X$, it is then enough to show that $u': \Gm \to \Omega K^\circ \to \ul{\Aut}(\1)$ is given by $u_\Gm$ and that $\Sigma^\infty B\Z_\ell(1) \wequi \Sigma^\infty \Gm \in \SH_\proet(S)_\ell^\comp$.
Unravelling the definitions, $u'$ sends a unit $a$ to the automorphism of $\1$ induced by multiplication by $a$ on $Th(\scr O) \wequi \tilde \Sigma \Gm$, which is also precisely what $u_\Gm$ also does; this proves the first claim.
The second claim follows from \cite[Theorem 6.5]{bachmann-SHet}.
\end{proof}

Let \begin{equation}\label{eq:JG} J_{G_\ell,\ell}^\comp: \BU_\ell \to \tau_{\ge 2}\BGL_1(\1_\ell^\comp) \in \CMon(\Shv_\proet(BG_\ell)) \end{equation} denote the map $\Omega^\infty \nu^*(j_{G_\ell,\ell}^\comp)_\ell^\comp$, where $j_{G_\ell,\ell}^\comp$ is the map from \eqref{eq:jG}.
\begin{cor} \label{cor:pointed-pullback}
We have \[ L_\ell \tilde\pi^* J_{G_\ell,\ell}^\comp \wequi J_{\proet,\ell}^\comp \in \Shv_\proet(\Sch_S)_*. \]
\end{cor}
\begin{proof}
By Lemma \ref{lemm:BUl-conn} below, we may replace $\tau_{\ge 2} \BGL_1(\1_\ell^\comp)$ by $\BGL_1(\1_\ell^\comp)$.
We may also assume that $S=\Spec(\Z[1/\ell])$.
Observe that in the diagram
\begin{equation*}
\begin{tikzcd}
L_\ell \pi^* \BU_\ell \ar[r,"\tilde\pi^* s"] \ar[rr, "\id" swap, bend left] & L_\ell \pi^* F_* B^2\Z_\ell(1) \ar[r] \ar[rr, "\pi^* (Bu_{B\Z_\ell(1)})^\dagger", bend right] &  L_\ell \pi^* \BU_\ell \ar[r, "J_{\proet,\ell}^\comp"] & \BGL_1(\1_{\proet,\ell}^\comp) \\
\end{tikzcd}
\end{equation*}
the top cell commutes by Theorem \ref{thm:equiv-splitting} and the bottom cell commutes by Lemma \ref{lem:pullback-fib}.
Taking Galois equivariant realization and restricting to $G_\ell'$ shows that the diagram remains commutatitve if we replace $J_{\proet,\ell}^\comp$ by $\tilde\pi^* J_{G_\ell,\ell}^\comp$.
We deduce that \[ J_{\proet,\ell}^\comp \wequi \pi^* (Bu_{B\Z_\ell(1)})^\dagger \tilde\pi^* s \wequi \tilde\pi^* J_{G_\ell,\ell}^\comp, \] as desired.
\end{proof}

\begin{lem} \label{lemm:BUl-conn} \NB{could move this earlier}
We have $\BU_\ell \in \Shv_\proet(B\Z_\ell^\times)_{\ge 2}$
\end{lem}
\begin{proof}
Since $\BU_\ell \wequi \Omega^\infty \nu^*(\bu_\ell)_\ell^\comp$, it suffices to prove the spectral analog.
Using Corollary \ref{cor:htpy-dim-0-conn}(3) several times, for this it is enough to show that $\nu^*(\bu_\ell)/\ell^n \in \ShvSp(B_\proet\Z_\ell^\times)_{\ge 2}$ and $\ul\pi_2 \nu^*(\bu_\ell)/\ell^{n+1} \to \ul\pi_2 \nu^*(\bu_\ell)/\ell^n$ is surjective.\NB{indeed these conditions then hold on sections over $W$, hence the limit is $\ge 2$ on sections over $W$, and hence so was the whole thing}
Since $\nu^*$ is exact this reduces to the same problem for $\bu_\ell^\comp \in \ShvSp(B_\et\Z_\ell^\times)$.
This may be checked on the (completed) stalk, which is just ordinary topological $\bu_\ell^\comp$, where the claim is clear.
\end{proof}

\subsection{Étale $\MGL$} \label{subsec:etale-MGL}
We can view $L_\proet(K^\circ)$ as an object of $\Shv_\proet(\Sch_S)_{/\tau_{\ge 2}\BGL_1(\1_{\proet,\ell}^\comp)}$, via Proposition \ref{prop:J-factors}.
Since $\tau_{\ge 2}\BGL_1(\1_{\proet,\ell}^\comp)$ is $\ell$-complete, we can also $\ell$-complete the source (essentially obtaining $J_{\proet,\ell}^\comp$).
\begin{prop} \label{prop:Kcirc-comp-fiberwise}
The map \[ L_\proet(K^\circ) \to L_\ell L_\proet(K^\circ) \wequi L_\ell \pi^* \BU_\ell \in \Shv_\proet(\Sch_S)_{/\tau_{\ge 2}\BGL_1(\1_{\proet,\ell}^\comp)} \] is an unstable $\ell$-completion.\footnote{Beware that this is a morphism in the slice category, whereas the $L_\ell$ refers to completion in the non-slice category.}
\end{prop}
\begin{proof}
Using Propositions \ref{prop:htpy-dim-0-PSigma} and \ref{prop:p-comp-slice}, it suffices to show that for $W \in \Sch_S$ w-contractible, the $\ell$-completion of $L_\proet(K^\circ)(W) \to (\tau_{\ge 2}\BGL_1(\1_{\proet,\ell}))(W)$ is fiberwise.
Via Lemma \ref{lemm:fiberwise-comp-trick}, for this it is enough to show that the target has vanishing $\pi_1$.
This follows from Corollary \ref{cor:htpy-dim-0-conn}(1).
\end{proof}

We can deduce our first major result of this section.
Recall the notion of motivic colimits from \S\ref{sec:motivic-colimits}.
\begin{thm} \label{thm:identify-MGL}
The image of $\MGL$ in $\SH_\proet(S)_\ell^\comp$ is given by the motivic colimit \[ M_{\tau_{\ge 2}\BGL_1(\1_{\proet,\ell}^\comp)}(L_\ell \pi^* \BU_\ell). \]
\end{thm}
\begin{proof}
Consider the motivic colimit functor \[ \PSh(\Sch_S)_{/\tau_{\ge 2}\BGL_1(\1_{\proet,\ell}^\comp)} \to \PSh(\Sch_S)_{/\BGL_1(\1_{\proet,\ell})} \to \SH_\proet(S)_\ell^\comp. \]
By Proposition \ref{prop:mot-colim-basics}(1) and Corollary \ref{cor:BGL1-trunc} it factors through $\Shv_\proet(\Sch_S)_{/\tau_{\ge 2}\BGL_1(\1_{\proet,\ell})}$.
Since it is cocontinuous by construction, by presentability and Proposition \ref{prop:completion-permanence} it preserves unstable $\ell$-equivalences.
Since all unstable $\ell$-equivalences in $\SH_\proet(S)_\ell^\comp$ are equivalences (Example \ref{ex:p-completion-stable}), we deduce from Proposition \ref{prop:Kcirc-comp-fiberwise} that \[ M_{\tau_{\ge 2}\BGL_1(\1_{\proet,\ell}^\comp)}(L_\ell \pi^* \BU_\ell) \wequi M_{\BGL_1(\1_{\proet,\ell})}(K^\circ). \]
This coincides with the image of $\MGL$ by Proposition \ref{prop:mot-colim-basics}(2,3) (i.e. functoriality of motivic colimits), Lemma \ref{lem:K-extended} (stability of $K$ under left Kan extension) and the definition of $\MGL$.
\end{proof}

We are now ready to prove the main theorem of this section.
The spectrum $\MU_\ell \in \CAlg(\ShvSp(B_\et G_\ell)_\ell^\comp)$ was constructed in \eqref{eq:MUl}.
\begin{thm} \label{thm:etale-cobordism-weak}
For any scheme $S$ with $1/\ell \in S$, there exists an equivalence \[ L_\et(\MGL)_\ell^\comp \wequi \tilde\pi^* \MU_\ell \in \Alg_{\scr E_0}(\SH_\et(S)_\ell^\comp). \]
If $\ell=2$ or if $S$ admits a morphism to a strictly henselian local scheme, then the equivalence upgrades to one of $\scr E_\infty$-rings.
\end{thm}
\begin{proof}
We may assume that $S = \Spec(\Z[1/\ell])$, and so is in particular strongly locally $\ell$-étale finite.
Hence by Proposition \ref{prop:strongly-finite-ff}, it suffices to prove that $\nu^* L_\et(\MGL)_\ell^\comp \wequi \pi^*\nu^*\MU_\ell$.
But the right hand side is just $M(\tilde\pi^* J_{G_\ell,\ell}^\comp)$ (by definition), which coincides with the image of $\MGL$ by Theorem \ref{thm:identify-MGL} and Corollary \ref{cor:pointed-pullback}.
To upgrade this to an equivalence of $\scr E_\infty$-rings, we need to upgrade the homotopy $\tilde\pi^* J_{G_\ell,\ell}^\comp \wequi J_{\proet,\ell}^\comp$ of pointed morphisms to one of $\scr E_\infty$-morphisms.
This follows from Proposition \ref{prop:identity-stable-j-trick} at the end of this section.
\end{proof}

\subsection{More about $\MU_\ell$} \label{subsec:more-MUl}
In $\Shv_\proet(B\Z_\ell^\times)$ we have the sheaf $\Z_\ell(1),$ represented by the pro-finite set $\lim_n \Z/\ell^n(1)$.
We let \[ \1_\ell^\comp(1)[1] = \Sigma^\infty(B \Z_\ell(1) )_\ell^\comp \in \ShvSp(B_\proet\Z_\ell^\times)_\ell^\comp. \]
\begin{lem} \label{lem:Zl-twisting}
The object $\1_\ell^\comp(1)$ is invertible and lies in the essential image of the fully faithful functor $\ShvSp(B_\et\Z_\ell^\times)_\ell^\comp \to \ShvSp(B_\proet\Z_\ell^\times)_\ell^\comp$, and $\pi^*(\1_\ell^\comp(1)) \wequi \1_\ell^\comp(1)$.
\end{lem}
\begin{proof}
The functor is fully faithful by Lemma \ref{lemm:BG-proet-finite-vcd}, and $\1_\ell^\comp(1)$ lies in the essential image by Lemma \ref{lemm:classical-proet-locally}.
It is invertible by Example \ref{ex:u1(1)}.
The assertion about $\pi^*(\1_\ell^\comp(1))$ holds essentially by construction \cite[\S3]{bachmann-SHet}.
\end{proof}

We have an unstable map \[ \Sigma B\Z_\ell(1) \to L_\ell(\nu^* B\Z/\ell^\infty(1)) \in \Shv_\proet(B\Z_\ell^\times), \] corresponding to the equivalence $\Omega L_\ell(\nu^* B\Z/\ell^\infty(1)) \wequi B\Z_\ell(1)$ from Lemma \ref{lem:HZl-basics}(1).
Stabilizing this yields a map \begin{equation} \label{eq:unitBZell} \Sigma^{2,1} \1_\ell^\comp \wequi \Sigma \1_\ell^\comp(1)[1] \to \Sigma^\infty(B\Z/\ell^\infty)_\ell^\comp \in \ShvSp(B_\et \Z_\ell^\times)_\ell^\comp \subset \ShvSp(B_\proet\Z_\ell^\times)_\ell^\comp. \end{equation}
\begin{lem}
Under the equivalences from Lemma \ref{lem:HZl-basics}, $\pi^*$ of the above map coincides with the $\ell$-completion of the stabilization of the canonical map $S^{2,1} \wequi \P^1 \to B\Gm \wequi \P^\infty$.
\end{lem}
\begin{proof}
This boils down to the fact that the canonical map $\P^1 \wequi \Sigma \Gm \to B\Gm$ is also adjoint to the identity map $\Gm \to \Omega B \Gm \wequi \Gm$.
\end{proof}

We now work in the category $\ShvSp(B_\et G_\ell)_\ell^\comp$, so that we have access to the ring spectrum $\MU_\ell$.
\begin{dfn}
We call a homotopy ring spectrum $E \in \ShvSp(B_\et G_\ell)_\ell^\comp$ \emph{complex oriented} if we are supplied with a map $\Sigma^\infty B\Z/\ell^\infty(1) \to \Sigma^{2,1} E$ extending the suspension of the unit $\Sigma^{2,1} \1 \to \Sigma^{2,1} E$ through \eqref{eq:unitBZell}.
\end{dfn}

\begin{prop}
$\MU_\ell$ has a canonical complex orientation, corresponding under the equivalence $\tilde\pi^*\MU_\ell \wequi L_\et(\MGL_\ell)$ to the canonical orientation.
\end{prop}
\begin{proof}
Taking (pro-étale) Thom spectra of the map $B\Z/\ell^\infty(1) \to \BU_\ell$ from Lemma \ref{lemm:construct-b} yields a map \[ o: M(B\Z/\ell^\infty(1)) \to \MU_\ell. \]
Using the embedding into $\ShvSp^{BG_\ell'}$ from Lemma \ref{lemm:BGl-ff}, one checks that $M(B\Z/\ell^\infty(1)) \wequi \Sigma^{-2,-1}\Sigma^\infty B\Z/\ell^\infty(1)$.\NB{this actually seems a bit annoying. better argument?}
Hence $\Sigma^{2,1}o$ defines a complex orientation of $\MU_\ell$.
Applying $\tilde\pi^*$ we obtain the canonical orientation, since the latter arises as $\MGL(1) \wequi Th(\gamma-1) \wequi \Sigma^{-2,-1}B\Gm \to \MGL$.
\end{proof}

\begin{cor} \label{cor:htpy-ring-MUl}
There exists an equivalence of homotopy ring spectra $e_2: L_\et(\MGL)_\ell^\comp \wequi \tilde\pi^*(\MU_\ell)$.
If the previous equivalence $e_1: L_\et(\MGL)_\ell^\comp \wequi \tilde\pi^*(\MU_\ell)$ (constructed in the proof of Theorem \ref{thm:etale-cobordism-weak}) was an equivalence of homotopy ring spectra, it coincides with the $e_2$ (but we do not know this in general).
\end{cor}
\begin{proof}
Any oriented motivic spectrum receives a canonical ring map from $\MGL$.
We hence obtain a unique ring map $e_2: L_\et(\MGL)_\ell^\comp \to \tilde\pi^*(\MU_\ell)_\ell^\comp$ compatible with the orientations.
We may check this is an equivalence on stalks.
In this case $e_1: L_\et(\MGL)_\ell^\comp \xrightarrow{\wequi} \tilde\pi^*(\MU_\ell)_\ell^\comp$ is an equivalence of ring spectra by Theorem \ref{thm:etale-cobordism-weak}, and so $e_1 \wequi e_2$ by uniqueness.
In particular $e_2$ is an equivalence, as needed.
\end{proof}

\begin{prop} \label{prop:cx-oriented-equiv}
Let $E \in \ShvSp(B_\et G_\ell)_\ell^\comp$ be a complex oriented ring spectrum.
\begin{enumerate}
\item There is a canonical isomorphism $E^{**}(B\Z/\ell^\infty(1)) \wequi E^{**}\fpsr{t}$, with $|t|=(2,1)$.
  The group structure on $B\Z/\ell^\infty(1)$ induces a formal group law on $E^{2*,*}$.
\item An equivalence of complex oriented ring spectra induces a strict isomorphism of the associated formal group laws.
\item The cosimplicial ring $\pi_{2*,*} \CB(\MU_\ell)$ coincides with the cosimplicial ring corresponding to a Hopf algebroid $(\pi_{2*,*}\MU_\ell, (\MU_\ell)_{2*,*}\MU_\ell)$.
  The formal group law on $\pi_{2*,*}\MU_\ell$ and the strict isomorphism of formal group laws on $\pi_{2*,*}(\MU_\ell \wedge \MU_\ell)$ corresponding to the switch map classify a morphism of Hopf algebroids $(\MU_*,\MU_*\MU) \to (\pi_{2*,*}\MU_\ell, (\MU_\ell)_{2*,*}\MU_\ell)$.
  The following diagram commutes
\begin{equation*}
\begin{tikzcd}
(\MU_*, \MU_*\MU) \ar[r] \ar[rd] & (\pi_{2*,*} L_\et(\MGL)_\ell^\comp, \pi_{2*,*} L_\et(\MGL \wedge \MGL)_\ell^\comp) \\
             & \ar[u] (\pi_{2*,*}\MU_\ell, (\MU_\ell)_{2*,*}\MU_\ell), \\
\end{tikzcd}
\end{equation*}
  where the horizontal map is the canonical one.
\end{enumerate}
\end{prop}
\begin{proof}
We constantly use the factorization \[ \ShvSp(B_\et G_\ell)_\ell^\comp \xrightarrow{\tilde\pi^*} \SH_\et(\Spec(\Z[1/\ell]))_\ell^\comp \to (\ShvSp^{BG_\ell'})_\ell^\comp. \]
Since the composite is fully faithful, $\tilde\pi^*$ is faithful and conservative.

(1) The class $t$ is the complex orientation.
We can use it to build a map $B\Z/\ell^\infty(1) \wedge E \to \bigoplus_{n \ge 0} \Sigma^{2n,n} E$.
It suffices to prove that this is an equivalence, which we may check after applying $\tilde\pi^*$, whence we reduce to the motivic analog.
For this see e.g. \cite[Proposition 6.2]{naumann2009motivic}.

(2) We need to prove that certain equations hold, which we may check after applying the faithful functor $\tilde\pi^*$; this reduces to a well-known result.

(3) For the first claim, it suffices to prove that $\MU_\ell \wedge \MU_\ell \wequi \bigoplus \Sigma^{2*,*} \MU_\ell$.
It suffices to prove this in $(\ShvSp^{BG_\ell'})_\ell^\comp$, where it follows from the same claim for $\MU$.
The rest is formal; see \cite[Corollary 6.7]{naumann2009motivic}.
\end{proof}

\subsection{Regular primes} \label{subsec:MGL-regular-primes}
\begin{prop} \label{prop:regular-prime-ff}
Let $\ell=2$ or an odd regular prime and $S$ the spectrum of a strictly henselian local ring.
The following functors are fully faithful:
\begin{enumerate}
\item $\pi^*_\mu: \ShvSp(B_\et G_\ell^\mu)_\ell^\comp \to \ShvSp(\Spec(\Z[1/\ell,\mu_\ell])_\et)_\ell^\comp$
\item $\pi^*: \ShvSp \to \ShvSp(S_\et)$
\end{enumerate}
\end{prop}
\begin{proof}
Let us denote all the functors by $\pi^*$.
Each of them preserves colimits and a compact-rigid generating family (a finite virtual cohomological dimension argument implies that the rigid generating family of Lemma \ref{lem:rigid-generation} consists of compact objects), and hence admits a cocontinuous right adjoint $\pi_*$.
It suffices to prove that $\pi_*\pi^* \1 \wequi \1$ (see e.g. \cite[Lemma 22]{bachmann-hurewicz}).
The functor $\pi^*$ preserves connective objects, and $\pi_*$ does because of finite virtual cohomological dimension.
Since for such objects $(\ph) \wedge H\Z$ is conservative, it suffices to prove that $H\Z \wequi H\Z \wedge \pi_*\pi^* \1$.
Observe that $H\Z \wedge \pi_*\pi^* \1 \wequi \pi_*\pi^* H\Z$ (see e.g. \cite[Lemma 3.2]{bachmann-topmod}).
We wish to show that $[X, H\Z]_* \wequi [X, \pi_*\pi^* H\Z]_*$ for every generator $X$.
Since the generators are rigid, for this we may as well prove that $[\1, H\Z \wedge X]_* \wequi [\1,  \pi_*\pi^*(H\Z \wedge X)]_*$.

First we prove (2).
We take $X = \1$ the sphere spectrum.
We thus need to show that $H^*_\et(S, \Z) = H^*_{sing}(pt, \Z)$, which is clear.

Now we prove (1).
We can choose generators of the form $X = \Sigma^\infty_+(A)/\ell$, where $A$ is a finite discrete $G_\ell^\mu$-set.
Then $X \wedge H\Z \wequi \Sigma^\infty_+(A) \wedge H\Z/\ell$ is concentrated in the heart of the $t$-structure, corresponding to the continuous $G_\ell^\mu$-representation $M=\F_\ell[A]$.
Since $G_\ell^\mu$ is a pro-$\ell$-group, $M$ can be obtained as a finite extension of copies of the trivial representation $\F_\ell$ (recall that non-trivial representations of $\ell$-groups over $\F_\ell$ are reducible \cite[\S8.3, Corollary of Proposition 26]{serre1977linear}).
It is thus enough to prove that $[\1, \pi_*\pi^* H\F_\ell]_* \wequi [\1, H\F_\ell]_*$.
This holds by construction (see Theorem \ref{thm:etale-fund-gps}).
\end{proof}

\begin{lem} \label{lem:mul-ext-faithful}
The canonical functor \[ \ShvSp(S_\et)_\ell^\comp \to \ShvSp(S[\mu_\ell]_\et)_\ell^\comp \] is faithful.
\end{lem}
\begin{proof}
Denote the functor by $\lambda^*$ and its right adjoint by $\lambda_*$.
Let $X, Y \in \ShvSp(S_\et)_\ell^\comp$.
By Galois descent we get $\Map(X,Y) \wequi \Map(X,\lambda_*\lambda^*Y)^{hG}$, where $G = Gal(S[\mu_\ell]/S) \hookrightarrow \Z/(\ell-1)$.
Since $|G|$ is invertible in $\ShvSp(S_\et)_\ell^\comp$, we thus get \[ [X,Y] \wequi [X,\lambda_*\lambda^*Y]^G \hookrightarrow [\lambda^* X, \lambda^* Y]. \]
\end{proof}

\begin{prop} \label{prop:identity-stable-j-trick}
Let $\ell=2$, or $S$ defined over the spectrum of a strictly henselian local ring.
Then $\tilde\pi^*(j_{G_\ell,\ell}^\comp)_\ell^\comp \wequi j_{\et,\ell}^\comp \in \ShvSp(S_\et)_\ell^\comp$.

If $\ell$ is an odd regular prime, the same is true if\NB{and only if?} the following holds: for any $\sigma \in Gal(\QQ)$ with $\pi(\sigma) = 1 \in \Z_\ell^\times$, the map induced by equivariant realization \[ \Sigma \bu_\ell^\comp \to \bu_\ell^\comp/0 \wequi \bu_\ell^\comp/(\sigma-1) \to \bgl_1 \1 \in \ShvSp \] is null.
\end{prop}
\begin{proof}
Let $S_0 = \Spec \Z[1/\ell,\mu_\ell]$ and $G = \Z/2 * \Z$ if $\ell=2$, $G$ the free group on $(\ell+1)/2$ generators if $\ell$ odd.
Consider the composite \[ F: \ShvSp((S_0)_\et)_\ell^\comp \to \ShvSp(\Spec(\QQ(\mu_\ell))_\et)_\ell^\comp \to (\ShvSp^{BGal(\QQ(\mu_\ell))})_\ell^\comp \to (\ShvSp^{BG})_\ell^\comp, \] where the last map is induced by lifting along $\pi_\mu$ similar to Remark \ref{rmk:action-alternative-construction}.
The composite $F\pi^*_\mu$ is fully faithful by Lemma \ref{lemm:BGl-ff}.
Since $\pi^*_\mu$ is also fully faithful (Proposition \ref{prop:regular-prime-ff}), in order to show that $\pi^*_\mu(j_{G_\ell,\ell}^\comp)_\ell^\comp \wequi j_{\et,\ell}^\comp$ it suffices to prove the same after applying $F$.
If $\ell=2$ the equivalence holds essentially by construction; see Remarks \ref{rmk:action-alternative-construction} and \ref{rmk:pi-factorization}.
For $\ell$ odd, we arrive at the following problem.
The image of either of the two maps in $(\ShvSp^{BG})_\ell^\comp$ is determined by $(\ell+1)/2$ $\Z$-equivariant maps $\bu_\ell^\comp \to \bgl_1(\1)$.
For $j_{\et,\ell}^\comp$, the equivariant maps arise from the equivariance of the $Gal(\QQ)$-equivariant realization corresponding to the images of the generators of $G$ in $Gal(\QQ)$, whereas for $\pi_\mu^* j_{G_\ell,\ell}^\comp$ they arise from the images of the generators of $G$ under $G \to Gal(\QQ) \to \Z_\ell^\times \to \Z_\ell \to Gal(\QQ)$.
These equivariances will be the same if for $\sigma \in Gal(\QQ)$ the equivariance only depends on $\pi(\sigma) \in \Z_\ell^\times$, which holds under the stated assumption.\NB{details if we care}

We have thus proved the result for $S=S_0$.
Since the claim is stable under base change, the case $\ell=2$ is settled.
For $\ell$ odd, the result for $S_0$ implies the one for $\Z[1/\ell]$ by Lemma \ref{lem:mul-ext-faithful}, settling this case as well.

The argument for strictly henselian local rings is similar, noting that the cyclotomic character factors as \[ \tilde\pi: S_\et \to B\{e\} \to BG_\ell, \] and using the functor \[ F: \ShvSp(S_\et) \to \ShvSp(\Spec(\bar k)_\et) \wequi \ShvSp, \] where $\bar k$ is the residue field of $S$.
\end{proof}

\section{Adams summands} \label{sec:adams-summands}
\localtableofcontents

\bigskip
In this section we will split $\MGL_{(\ell)}$ and $\KGL_{(\ell)}$ into $\ell-1$ summands.
We do it by mimicking the topological arguments.
This section is largely independent from the previous ones.

Let $\ell$ be an odd prime.
By determining explicitly $[\KU_{(\ell)}, \KU_{(\ell)}]$, Adams wrote down idempotents in this ring, which split off summands $e_0^{(\ell)}\KU_{(\ell)}$ from the spectrum $\KU_{(\ell)}$.
Taking infinite loop spaces, we obtain a summand of $\BU_{(\ell)}$, and taking thom spectra we obtain an $\scr E_\infty$-idempotent of $\MU_{(\ell)}$, also referred to as Adams summand.
We can perform largely the same construction motivically.
At least for reasonable bases, $[\KGL_{(\ell)}, \KGL_{(\ell)}]$ is in bijection with $[\KU_{(\ell)}, \KU_{(\ell)}]$ and so the idempotent from topology immediately lifts to the motivic world.
Now as before we take infinite loop spaces to obtain a summand of $\BGL_{(\ell)}$, and hence an $\scr E_\infty$-idempotent of $\MU_{(\ell)}$.

\subsubsection*{Organization}
In \S\ref{subsec:adams-topology} we recall the topological picture.
Next in \S\ref{subsec:local-j-hom} we establish some results about the interaction between motivic $j$-homomorphisms and inverting the prime $2$.
We use this in \S\ref{subsec:motivic-adams-summand} to construct the motivic Adams summand $e_0^{(\ell)} \MGL_{(\ell)}$ of $\MGL_{(\ell)}$.
Finally in \S\ref{subsec:Adams-etale} we study the interaction of our construction with étale realization: we show that under the equivalence $L_\et(\MGL)_\ell^\comp \wequi \pi^*\MU_\ell$, the summand $L_\et(e_0^{(\ell)}\MGL_{(\ell)})_\ell^\comp$ is obtained by pullback from a summand of $\MU_\ell$.

\subsubsection*{Standing assumptions}
We will throughout assume that $\ell$ is \emph{odd}.

\subsection{Recollections from topology} \label{subsec:adams-topology}
Let $\scr L$ be a set of primes.
The map \[ [\KU_{\scr L^{-1}}, \KU_{\scr L^{-1}}] \to [\KU_\QQ, \KU_\QQ] \] is injective (see e.g. \cite[Lemma 4]{adams1969lectures}).
The well-known splitting $\KU_\QQ \wequi \bigoplus_n \Sigma^{2n} H\QQ$ yields idempotents $\{e_i \mid i \in \Z\}$ on the right hand side.
Let $\ell$ be a prime.
For $\alpha \in \Z/(\ell-1)$ we set \[ e_\alpha^{(\ell)} = \sum_{[n] = \alpha} e_n, \] where the sum is over all $n \in \Z$ in the residue class $\alpha$.
One may show that this is a well-defined idempotent of $[\KU_\QQ, \KU_\QQ]$ and in fact lies in $[\KU_{(\ell)}, \KU_{(\ell)}]$ \cite[Theorem 5]{adams1969lectures}\NB{This ref says: primes $p$ occurring in denominators are $\not\equiv 1 \pmod{\ell-1}$, so $p \ne \ell$.}.
We obtain a splitting \[ \KU_{(\ell)} \wequi \bigoplus_{\alpha \in \Z/(\ell-1)} e_\alpha^{(\ell)} \KU_{(\ell)}. \]
Essentially by construction we have \[ \pi_* e_\alpha^{(\ell)} \KU_{(\ell)} \wequi \begin{cases} \Z_{(\ell)} & *=2n, [n] = \alpha \\ 0 & \text{else} \end{cases}. \]
Passing to connective covers, we obtain similar idempotents and splittings for $\ku_{(\ell)}, \bu_{(\ell)}$.

The $j$-homomorphism $\bu \to \bgl_1 \1 \to \bgl_1(\1_{(\ell)})$ factors through $\tau_{\ge 2} \bgl_1(\1_{(\ell)}) \wequi \tau_{\ge 2}(\bgl_1 \1)_{(\ell)}$ and hence induces the local $j$-homomorphism \[ j_{(\ell)}: \bu_{(\ell)} \to \tau_{\ge 2}(\bgl_1 \1)_{(\ell)} \to \bgl_1(\1_{(\ell)}). \]
Taking the Thom spectrum of $e_0^{(\ell)}\bu_{(\ell)} \to \bu_{(\ell)}$ we obtain a morphism of $\scr E_\infty$-rings \[ e_0^{(\ell)}\MU_{(\ell)} := M(e_0^{(\ell)}\bu_{(\ell)}) \to \MU_{(\ell)}. \]
\begin{lem}
The map $e_0^{(\ell)}\MU_{(\ell)} \to \MU_{(\ell)}$ is a summand inclusion.
In fact this is the inclusion of Adams' summand.
\end{lem}
\begin{proof}
Write $\MU_0 \subset \MU_{(\ell)}$ for Adams' summand.
It suffices to show that the composite $e_0^{(\ell)} \MU_{(\ell)} \to \MU_{(\ell)} \to \MU_0$ is an equivalence, which we may check in homology.
By construction $H_* e_0^{(\ell)} \MU_{(\ell)} \wequi H_* e_0\BU_{(\ell)}$.
This is the same as the homology of $\MU_0$ \cite[p. 110]{adams1969lectures}.
\end{proof}
In particular \cite[Theorem 19]{adams1969lectures} \[ \pi_* e_0^{(\ell)} \MU_{(\ell)} \wequi \Z_{(\ell)}[x_1, x_2, \dots] \quad\text{where}\quad |x_i| = 2i(\ell-1). \]

\subsection{Local motivic $j$-homomorphisms} \label{subsec:local-j-hom}
Recall the splitting of $\1[1/2]$ into $\1[1/2]^+ = \1[1/2]_\eta^\comp$ and $\1[1/2]^- = \1[1/2,1/\eta]$, on which $\eta$ acts by $0$ respectively as a unit (see e.g. \cite[\S2.7.3]{bachmann-eta}).
Denote by $\BGL_1 \1[1/2]^+ \in \PSh(\Sm_S)$ the subgroupoid of $\SH(\ph)^\wequi$ on objects that are Nisnevich locally equivalent to $\1[1/2]^+$, and similarly for other localizations of the sphere.
This is a Nisnevich sheaf.

\begin{lem}
The $j$-homomorphism \[ K^\circ \to \BGL_1 \1 \to \BGL_1 \1[1/2]^+ \in \Shv_\Nis(\Sm_S) \] factors canonically through $\tau_{\ge 2}\BGL_1 \1[1/2]^+$.
\end{lem}
\begin{proof}
As in the proof of Proposition \ref{prop:J-factors}, we need to show that $\ul\pi_1 K^\circ \to \ul\pi_1 \BGL_1 \1[1/2]^+$ is null.
For this it is enough to show that if $a \in \scr O^\times(X)$ then the endomorphism $\lra{a}: \1[1/2]^+ \to \1[1/2]^+ \in \SH(X)$ is homotopic to the identity.
This follows from the fact that $\lra{a} = 1+\eta[a]$.\NB{ref?}
\end{proof}

Let $\scr L$ be a set of primes containing $1/2$.
By consideration of homotopy sheaves (and Corollary \ref{cor:loc-topos}), we see that $\tau_{\ge 2}\BGL_1 \1_{\scr L^{-1}}^+$ is $\scr L$-local.
Consequently we obtain the \emph{$\scr L$-local motivic $j$-homomorphism} \[ j_{\scr L^{-1}}: L_{\scr L^{-1}}K^\circ \to \tau_{\ge 2}\BGL_1 \1_{\scr L^{-1}}^+ \to \BGL_1 \1_{\scr L^{-1}}^+ \] and an associated Thom spectrum functor \[ M=M_{\scr L^{-1}}^+: \PSh(\Sm_S)_{/L_{\scr L^{-1}}K^\circ} \to \SH(S)_{\scr L^{-1}}^+. \]
\begin{lem}
Suppose that $S$ has finite dimension.\NB{Not really necessary.}
Given $A \to K^\circ \in \Shv_\Nis(\Sm_S)$ with stalkwise nilpotent homotopy sheaves, the associated map \[ A \to L_{\scr L^{-1}}A \in \Shv_\Nis(\Sm_S)_{/\tau_{\ge 2}\BGL_1 \1_{\scr L^{-1}}^+} \] is an $\scr L$-localization.
In particular $M(A)_{\scr L^{-1}}^+ \wequi M_{\scr L^{-1}}^+(L_{\scr L^{-1}}A)$.
\end{lem}
\begin{proof}
$\Shv_\Nis(\Sm_S)_{/\tau_{\ge 2}\BGL_1 \1_{\scr L^{-1}}^+}$ is an $\infty$-topos with enough points, so by Corollary \ref{cor:loc-topos} we may check the claim on stalks.
These are given by pairs $(X, \alpha)$ where $X$ is an essentially smooth, henselian local $S$-scheme and $\alpha: X \to \tau_{\ge 2}\BGL_1 \1_{\scr L^{-1}}^+$.
Given $\scr F \to \tau_{\ge 2}\BGL_1 \1_{\scr L^{-1}}^+$, we get \[ \Map_{\Shv_\Nis(\Sm_S)_{/\tau_{\ge 2}\BGL_1 \1_{\scr L^{-1}}^+}}((X, \alpha), \scr F) \wequi \scr F(X) \times_{\tau_{\ge 2}\BGL_1 \1_{\scr L^{-1}}^+(X)} \{\alpha\}. \]
In other words we need to check that the $\scr L$-localization of $A(X) \to \tau_{\ge 2}\BGL_1 \1_{\scr L^{-1}}^+(X)$ is fiberwise.
This is clear from the formulas for localization of nilpotent spaces, since $\tau_{\ge 2}\BGL_1 \1_{\scr L^{-1}}^+(X)$ is simply-connected.

The last claim follows since $M_{\scr L^{-1}}^+$ preserves $\scr L$-equivalences (by Proposition \ref{prop:localization-permanence}).
\end{proof}

\subsection{Motivic Adams summands} \label{subsec:motivic-adams-summand}
Recall the complex realization functor $r_\CC: \SH(\CC) \to \ShvSp$ induced by sending a complex variety to its classical space of complex points.
If $R$ is a subring of $\CC$, we denote also by $r_\CC$ the composite $\SH(R) \to \SH(\CC) \to \ShvSp$.
\begin{lem} \label{lem:lifting-endomorphisms}
Let $\scr L$ be a set of primes, and $R \subset \CC$ regular.
\begin{enumerate}
\item If $K_1(R)$ is finite, complex realization induces a bijection \[ \KGL_{\scr L^{-1}}^{2*,*}\KGL_{\scr L^{-1}} \wequi \KU_{\scr L^{-1}}^{2*}\KU_{\scr L^{-1}}. \]
\item If all primes not in $\scr L$ are invertible in the Dedekind domain $R$ with vanishing Picard group, complex realization induces a bijection \[ \MGL_{\scr L^{-1}}^{2*,*}\MGL{(\scr L)} \wequi \MU_{\scr L^{-1}}^{2*}\MU_{\scr L^{-1}}. \]
\end{enumerate}
In particular (1) holds if $R=\Z$, and (2) holds if $\scr L = (\ell)$ and $R=\Z[1/\ell]$.
\end{lem}
\begin{proof}
(1) We have $\KGL \wequi \colim \Sigma^\infty (\BGL \times \Z)$, yielding an exact sequence \[ 0 \to \limone \KGL_{\scr L^{-1}}^{2*-1,*}(\BGL \times \Z) \to \KGL_{\scr L^{-1}}^{2*,*}\KGL_{\scr L^{-1}} \to \lim \KGL_{\scr L^{-1}}^{2*,*}(\BGL \times \Z) \to 0. \]
There is a similar sequence in topology.
Since $\KGL$ is orientable we can compute $\KGL^{**}\BGL$ just as in topology.
The $\limone$ term disappears by our assumption (since we are looking at a system of finite groups), and it always does in topology.
Since $r_\CC(\BGL) = \BGL$ and (thus) $r_\CC(\KGL) = \KU$, the result follows.
See also \cite[Corollary 5.2.4]{riou2010algebraic}.

(2) By the Thom isomorphism we have \[ \MGL_{\scr L^{-1}}^{2*,*}\MGL \wequi \MGL_{\scr L^{-1}}^{2*,*}\BGL \wequi \MGL^{2*,*}[c_1, c_2, \dots], \] just as in topology.
It thus suffices to show that $\MGL^{2*,*}_{\scr L^{-1}} \wequi \pi_{2*} \MU_{\scr L^{-1}}$.
Since $R$ is a Dedekind domain on which all primes outside $\scr L$ are invertible, we have $\ul\pi_{2*,*}\MGL_{\scr L^{-1}} \wequi \pi_{2*} \MU_{\scr L^{-1}}$ and $\ul\pi_{2*+1,*}\MGL_{\scr L^{-1}} \wequi \scr O^\times \otimes M_*$, where $M_*$ is a free $\Z_{\scr L^{-1}}$-module \cite[Corollaries 7.4 and 7.5]{SpitzweckMGL}.
To conclude, using the local-to-global spectra sequence, it will be enough to show that $H^1(\Spec(R), \ul\pi_{2*+1,*}\MGL_{\scr L^{-1}}) = 0$, which follows from the assumption on the Picard group.
\end{proof}

\begin{dfn} \label{def:adams-summands}
We can consequently use the idempotent $e_0^{(\ell)}$ of \S\ref{subsec:adams-topology} to define an idempotent of $\KGL$ over $\Z$, and consequently over any scheme by base change.
We obtain \[ e_0^{(\ell)}\KGL_{(\ell)} \in \SH(S). \]
Now suppose that $S$ is regular.
Then $\Omega^\infty \KGL \wequi K$, and so we obtain summands \[ e_0^{(\ell)}K_{(\ell)} \subset K_{(\ell)} \quad\text{and}\quad e_0^{(\ell)}K_{(\ell)}^\circ \subset K_{(\ell)}^\circ. \]
Applying the $(\ell)$-local motivic Thom spectrum functor of \S\ref{subsec:local-j-hom}, we obtain\NB{probably even a normed spectrum} \[ e_0^{(\ell)} \MGL_{(\ell)} := M_{(\ell)}^+(e_0^{(\ell)}K_{(\ell)}^\circ) \in \CAlg(\SH(S)). \]
\end{dfn}

Here are the basic properties of these constructions.
\begin{prop} \label{prop:e0MGL-basics}
\begin{enumerate}
\item Formation of $e_0^{(\ell)}\KGL_{(\ell)}, e_0^{(\ell)}K_{(\ell)}^\circ, e_0^{(\ell)}\MGL_{(\ell)}$ is stable under base change (among regular schemes for the last two).
\item We have $r_\CC(e_0^{(\ell)}\KGL_{(\ell)}) \wequi e_0^{(\ell)}\KU_{(\ell)}$, $r_\CC(e_0^{(\ell)}K^\circ_{(\ell)}) \wequi e_0^{(\ell)}\BU_{(\ell)}$ and $r_\CC(e_0^{(\ell)}\MGL_{(\ell)}) \wequi e_0^{(\ell)}\MU_{(\ell)}$.
\item Let $S = \Spec(\Z[1/\ell])$.
  Then $e_0^{(\ell)}\MGL_{(\ell)}$ is the summand of $\MGL_{(\ell)}$ corresponding to the idempotent $e_0^{(\ell)}$ obtained via Lemma \ref{lem:lifting-endomorphisms}.
\item Let $S$ be essentially smooth over a Dedekind scheme, and $1/\ell \in S$.
  Then \[ s_* e_0^{(\ell)} \KGL_{(\ell)} \wequi H\Z_{(\ell)}[b, b^{-1}], |b| = (2(\ell-1),\ell - 1) \] and \[ s_* e_0^{(\ell)} \MGL_{(\ell)} \wequi H\Z_{(\ell)}[x_1, x_2, \dots], |x_i| = (2i(\ell-1), i(\ell-1)). \]
\item If $1/\ell \in S$ we have \[ e_0^{(\ell)} \MGL_{(\ell)} \otimes e_0^{(\ell)} \MGL_{(\ell)} \wequi e_0^{(\ell)} \MGL_{(\ell)}[b_1, b_2, \dots], |b_i| = (2i(\ell-1), i(\ell-1)). \]
\end{enumerate}
\end{prop}
\begin{proof}
(1) The statement about $\KGL$ is clear.
Since formation of $K^\circ \wequi Gr$ is stable under base change, so is the formation of the summand corresponding to the idempotent $e_0$ (which itself is stable under base change by definition).
The statement about $\MGL$ follows from Proposition \ref{prop:mot-colim-basics}.

(2) Since $r_\CC(\KGL) = \KU$, the first statement holds by design.
The functor $r_\CC$ does not commute with $\Omega^\infty$ in general, but it does when applied to $\KGL$, by inspection.
This implies that $r_\CC(e_0^{(\ell)}\BU_{(\ell)})$ is the summand of $r_\CC(L_{(\ell)}\BGL) \wequi \BU_{(\ell)}$ corresponding to $e_0$, whence the second statement.
The third is again a formal consequence.

(3) Denote by $\MGL_0$ the summand of $\MGL_{(\ell)}$ corresponding to $e_0$.
Then we have a canonical map $e_0\MGL_{(\ell)} \to \MGL_{(\ell)} \to \MGL_0$ which we shall show is an equivalence.
By construction this is a morphism of connective spectra on which $\eta=0$, and hence it suffices to prove that we have an equivalence after $\otimes H\Z$\NB{ref?}.
By the Thom isomorphism we get $e_0\MGL_{(\ell)} \otimes H\Z \wequi \Sigma^\infty_+ e_0K^\circ_{(\ell)} \otimes H\Z$.
By construction this is a summand of $\Sigma^\infty_+ K^\circ_{(\ell)} \otimes H\Z \wequi \MGL_{(\ell)} \otimes H\Z$.
Over our base, maps between such motives are equivalences if and only if they are after complex realization, see Lemma \ref{lemm:conservative-realization} below.
By (2), we have thus reduced to topology, where the result holds by definition.

(4) We may assume that $S=\Spec(\Z[1/\ell])$.
We have $s_* \KGL_{(\ell)} = H\Z_{(\ell)}[\beta,\beta^{-1}]$ and $s_* \MGL_{(\ell)} = H\Z_{(\ell)} \otimes L_*$.
The slices of $e_0\KGL_{(\ell)}$ and $e_0 \MGL_{(\ell)}$ are summands thereof.
For $\KGL$ and $\MGL$, the slices are turned by complex realization into the homotopy groups (by inspection).
It follows that the correct summands are ``the same'' as in topology, whence the claim follows from \S\ref{subsec:local-j-hom}.

(5) As in (3) it suffices to prove that the canonical comparison map induces an isomorphism on $\otimes H\Z$, which can be checked after complex realization.
\end{proof}

\begin{lem}\NB{I feel like I have written this somewhere before} \label{lemm:conservative-realization}
Let $R \subset \CC$ be a Dedekind domain with trivial Picard group and $A$ any ring.
Write $\scr C \subset \Mod_{HA}(\SH(R))$ for the subcategory on summands of sums of objects of the form $\Sigma^{2*,*} HA$.
Then complex realization on $\scr C$ is conservative.
\end{lem}
\begin{proof}
We claim that complex realization is in fact fully faithful on the homotopy category.
For this it suffices to show that $\pi_{2*,*}HA = 0$ for $* \ne 0$ and $=A$ else.
Since $\Spec(R)$ has dimension $1$, this is clear for $* \ne -1$ from the construction of $HA$ in \cite{spitzweck2012commutative}.
For $*=-1$ we have $\pi_{-2,-1}HA = Pic(R) \otimes A$, which vanishes by assumption.
\end{proof}

\subsection{Étale realization} \label{subsec:Adams-etale}
Recall that we have the functors \[ \ShvSp(B_\et \Z_\ell^\times)_\ell^\comp \xrightarrow{\pi^*} \SH_\et(\QQ)_\ell^\comp \xrightarrow{\alpha} (\ShvSp^{B\Z})_\ell^\comp, \] where $\alpha$ comes from a map $\Z \to Gal(\QQ)$ lifting a topological generator of $\Z_\ell^\times$.
Then the composite $\alpha\pi^*$ is fully faithful, and $\alpha L_\et(\MGL)_\ell^\comp \wequi \alpha\pi^* \MU_\ell$.
We thus find that $\alpha L_\et(e_0\MGL_{(\ell)})_\ell^\comp$ defines an object $\alpha\pi^*(e_0 \MU_\ell)$ with \[ e_0 \MU_\ell \in \CAlg(\ShvSp(B_\et \Z_\ell^\times)_\ell^\comp), \] and similarly we obtain \[ e_0 \KU_\ell \in \ShvSp(B_\et\Z_\ell^\times)_\ell^\comp. \]
By construction, these are summands of $\MU_\ell$ and $\KGL_\ell$, respectively.

\begin{prop} \label{prop:e0-etale-reln}
Let $S \in \Sch_{\Z[1/\ell]}$.
Then we have an equivalence of $\scr E_0$-rings \[ L_\et(e_0 \MGL_{(\ell)})_\ell^\comp \wequi \pi^*(e_0\MU_\ell), \] as well as an equivalence of unstructured spectra \[ L_\et(e_0 \KGL_{(\ell)})_\ell^\comp \wequi \pi^*(e_0 \KU_\ell). \]
\end{prop}
\begin{proof}
Since both sides are summands of $L_\et(\MGL)_\ell^\comp$ by Theorem \ref{thm:etale-cobordism-weak}, we have a canonical comparison morphism $L_\et(e_0 \MGL_{(\ell)})_\ell^\comp \to \pi^*(e_0\MU_\ell)$ which we need only prove is an equivalence.
This we may do after pulling back to the étale stalks, i.e. we may assume that $S$ is the spectrum of a separably closed field.
But now $\SH_\et(S)_\ell^\comp \wequi \ShvSp_\ell^\comp$, and both spectra identify with the classical Adams summand.\NB{Over general bases $L_\et(\MGL)_\ell^\comp$ could have more complicated endomorphisms, so couldn't immediately identify the idempotents.}
The same argument works for $\KGL$.
\end{proof}

\section{Quantifying eventual étale locality} \label{sec:CQL}

\localtableofcontents

\bigskip
The Bloch--Kato conjecture states that motivic cohomology coincides with truncated étale cohomology.
Put differently, the map $H\Z/\ell \to L_\et H\Z/\ell \in \SH(k)$ induces an isomorphism on $\pi_{**}$ in certain specific degrees, and in fact the source is zero outside of these degrees.
When considering $\KGL$ instead of $H\Z$, we have the Quillen--Lichtenbaum conjecture.
It states that algebraic $K$-theory coincides with its étale localization beyond a certain degree, depending on the étale cohomological dimension of $k$.
Contrary to the situation for motivic cohomology, algebraic $K$-theory does not vanish outside of the étale local range.
This can be explained as follows.
$\KGL$ is built out of shifted copies of $H\Z$, namely copies of the form $\Sigma^{2*,*} H\Z$.
$H\Z$ is obtained ($\ell$-adically) from its étale localization, by truncating into the range of $\pi_{p,q}$ with $p \ge q$ and $q \le 0$.
Since this range is not invariant under replacing $(p,q)$ by $(p+2,q+1)$, there is no similar exact formula for $\pi_{**}(\KGL/\ell)$.

We can obtain a way out of this problem by describing $H\Z$, perhaps approximately, as obtained from its étale localization by performing a kind of truncation invariant under $(p,q) \mapsto (p+2,q+1)$.
If so, then $\KGL$ will automatically inherit the same description, and so will any other spectrum that can be built out of copies of $\Sigma^{2*,*} H\Z$.
To find the appropriate kind of truncation, note that the transformation $(p,q) \mapsto (p+2,q+1)$ leaves the quantity $c(p,q) = p-2q$ invariant.
It is called the \emph{Chow degree}.
We say that $E \in \SH(k)$ satisfies the \emph{Chow--Quillen--Lichtenbaum} property of degree $s$ if the fiber of $E_\ell^\comp \to L_\et(E)_\ell^\comp$ has homotopy groups concentrated in Chow degrees $<s-1$.
Careful inspection of the truncation implicit in the Bloch--Kato conjecture shows that $H\Z$ satisfies the Chow--Quillen--Lichtenbaum property of degree $\cd_\ell(k)-1$.
By the above argument, the same holds for $\MGL$; this is the main result in this section (Example \ref{ex:CQL-sc-KGL}).

\subsubsection*{Organization}
We begin in \S\ref{def:quillen-lichtenbaum} by defining the Chow--Quillen--Lichtenbaum property, establishing basic facts and producing basic examples (i.e., $H\Z$).
Then in \S\ref{subsec:etale-slice} we show that under appropriate assumptions, the étale localized slice filtration of a spectrum converges to the étale localization of the spectrum itself.
From this we deduce that $\MGL$ satisfies the Chow--Quillen--Lichtenbaum property.

We conclude with two complementary sections.
In \S\ref{subsec:prime-2} we strengthen our results in the case $\ell=2$, showing that in some situations (notably excluding $H\Z$, but including $\MGL$) the assumptions on the étale cohomological dimension of $k$ can be replaced by assumptions on the \emph{virtual} étale cohomological dimension.
Finally in \S\ref{subsec:SH-et-nh} we establish an application: by using the Chow--Quillen--Lichtenbaum property of the sphere itself, we prove that any motivic spectrum satisfying étale descent automatically satisfies étale hyperdescent.

\subsection{Quillen--Lichtenbaum properties} \label{def:quillen-lichtenbaum}
\begin{dfn}
Given $(i,j) \in \Z^2$, we call $c(i,j) = i-2j$ the associated \emph{Chow degree}.
\end{dfn}

\begin{dfn}[Chow--Quillen--Lichtenbaum property]
Let $E \in \SH(S)$.
We say that $E$ satisfies $\CQL_s(\ell, S)$ (or $\CQL_s(\ell)$ for short) if the map \[ E_\ell^\comp \to L_\et(E)_\ell^\comp \in \SH(S) \] induces an isomorphism on $\pi_{i,j}$ in Chow degrees $c(i,j) \ge s$, and a monomorphism in Chow degree $c(i,j)=s-1$.
\end{dfn}

\begin{rmk} \label{rmk:extensions}
Equivalently, the homotopy groups of the fiber vanish in Chow degrees $\ge s-1$.
\end{rmk}

\begin{rmk} \label{rmk:CQL-local}
For this, it suffices that the fiber is hypercomplete (e.g. $S$ of finite dimension) and its homotopy \emph{sheaves} (restricted to $S_\Nis$) vanish in Chow degrees $\ge s-1$.
(Though beware that unless $E$ is $\ell$-power torsion, these homotopy sheaves may not be particularly well-behaved.)
\end{rmk}

\begin{wrn}
Write $\Omega^\infty: \SH(S) \to \PSh_\ShvSp(\Sm_S)$ for the canonical functor.
It is usually not the case that $L_\et \Omega^\infty E \wequi \Omega^\infty L_\et E$, whence the above Quillen--Lichtenbaum property only loosely relates to the usual Quillen--Lichtenbaum conjecture.
\end{wrn}

\begin{exm}
$L_\et(\KGL)_\ell^\comp$ represents $K_1$-local $K$-theory, and $L_\et(\Omega^\infty\KGL)_\ell^\comp \to \Omega^\infty L_\et(\KGL)_\ell^\comp$ is a pro-$\ell$ connective cover (see the proof of Proposition \ref{prop:K-pulled-back}).
In particular $\KGL$ satisfies $\CQL_s(\ell, S)$ (for some $s \ge 0$) if and only $\ell$-adic (homotopy) algebraic $K$-theory over $S$ is étale local in degrees $\ge s$.
\end{exm}

\begin{lem} \label{lemm:CQL-basics}
\begin{enumerate}
\item Spectra satisfying $\CQL_s(\ell)$ are closed under extensions.
\item If $S$ $\ell$-étale finite (\cite[Definition 5.8]{bachmann-SHet}) and $\cd_\ell(S) < \infty$ (e.g. $S$ noetherian, $\dim S < \infty$, and $\pcd_\ell S < \infty$), then they are also closed under coproducts.
\item $E$ satisfies $\CQL_s(\ell)$ if and only if $E_\ell^\comp$ does.
\item If $E \in \SH(S)$ and $E/\ell$ satisfies $\CQL_s(\ell)$, then so do $E/\ell^n$ and $E$.
\end{enumerate}
\end{lem}
\begin{proof}
(1) Immediate from Remark \ref{rmk:extensions}.

(2) Using Lemma \ref{lemm:complete-sums} it suffices to prove that the functor $\map(\Sigma^{i,j}\1, L_\et(\ph)_\ell^\comp): \SH(S) \to \ShvSp_\ell^\comp$ preserves colimits.
It is the composite of the functors $\SH(S) \to \SH_\et(S)_\ell^\comp$ (which preserves colimits) and $\map(\Sigma^{i,j}\1, \ph): \SH_\et(S)_\ell^\comp \to \ShvSp_\ell^\comp$.
The latter functor is right adjoint to $e: \ShvSp_\ell^\comp \to \SH_\et(S)_\ell^\comp$ induced by $\1 \mapsto \Sigma^{i,j} \1$.
The functor $e$ preserves the compact generator $\1/\ell$ by \cite[Corollary 5.7]{bachmann-SHet}, whence its right adjoint preserves colimits, as needed.
For the characterization of $\ell$-étale finiteness and $\cd_p(S) < \infty$, see \cite[Example 5.9, Theorem 2.10]{bachmann-SHet}.

(3) Write $F(E) = \fib(E_\ell^\comp \to L_\et(E)_\ell^\comp)$.
Then $(*)$ $F(E) = \lim_n F(E/\ell^n)$.
By Remark \ref{rmk:extensions}, $E$ satisfies $\CQL_s(\ell)$ if and only if $F(E)$ has homotopy groups concentrated in degrees $<s-1$.
As we have just seen, this condition only depends on $E/\ell^n$ (for all $n$), and hence is unchanged replacing $E$ by $E_\ell^\comp$.

(4) By Remark \ref{rmk:extensions} and $(*)$, $E$ satisfies $\CQL_s(\ell)$ if $E/\ell^n$ does for all $n$.
This holds if $E/\ell$ satisfies $\CQL_s(\ell)$, by stability under extension.\NB{And $E/\ell^{n-1} \to E/\ell^n \to E/\ell$.}
\end{proof}

\begin{prop} \label{prop:HZ-CQL}
Let $S$ be essentially smooth over a Dedekind scheme and $1/\ell \in S$.
Then $H\Z$, $H\Z_\ell$ as well as $H\Z/\ell^n$ satisfy $\CQL_s(\ell, S)$ provided that one of the following holds:
\begin{enumerate}
\item $S$ is henselian local with residue field $k$ and $\cd_\ell(k) \le s+1$.
\item $\pcd_\ell(S) \le s + 1$.
\end{enumerate}
In either case we have $\pi_{i,w} H\Z/\ell = 0$ for $i < -(s+1+\cd_\Nis S)$.
\end{prop}
\begin{proof}
By Lemma \ref{lemm:CQL-basics}(4) it suffices to prove that $H\Z/\ell$ satisfies $\CQL_s(\ell)$.

(1) By definition (see \cite{spitzweck2012commutative}) and since $S$ is local, for $w \ge 0$ we have \[ \map(\Sigma^{0,-w}\1, H\Z/\ell) \wequi \tau_{\ge -w} \Gamma_\et(S, \mu_\ell^{\otimes w}), \] and $\map(\Sigma^{0,-w}\1, H\Z/\ell) = 0$ for $w < 0$.
It follows that if $F = \fib(H\Z/\ell \to L_\et H\Z/\ell)$ then \[ \map(\Sigma^{0,-w}\1, F) \wequi \Sigma^{-1} \tau_{<-w} \Gamma_\et(S, \mu_\ell^{\otimes w}) \text{ or } 0. \]
We have $\Gamma_\et(S, \mu_\ell^{\otimes *}) \wequi \Gamma_\et(k, \mu_\ell^{\otimes *})$ by rigidity \cite[Theorem 1]{gabber1994affine}.
It follows that if $w>s$ then $F=0$.
Also $F=0$ if $w < 0$.
Finally for any $w$ we have $F \in \ShvSp_{\le -w-2}$.
It follows that $\pi_{**} F$ is concentrated in degrees $(n,-w)$ with $n \le -w-2$ and $0 \le w \le s$, whence \[ c(n,-w) = n+2w \le w-2 \le s-2. \]
Thus $H\Z/\ell$ satisfies $\CQL_s(\ell)$ by Remark \ref{rmk:extensions}.

The vanishing of $\pi_{i,w} H\Z/\ell$ for $i < -\cd_\ell(k)$ is immediate from the fact that this group is either zero or coincides with $H^{-i}_\et(k, \mu_\ell^{\otimes w})$.

(2) By using (1), we see that $\ul{\pi}_{**}(F)|_{S_\Nis}$ is concentrated in Chow degrees $\le s-2$.
The first result thus follows from Remark \ref{rmk:CQL-local}, and the second follows from the same assertion for (1) by using e.g. the descent spectral sequence.
\end{proof}

\begin{exm}
Proposition \ref{prop:HZ-CQL}(2) applies with $s=1$ if $S$ is an open subset of $\Spec \scr O_K$, where $K$ is a number field without real embeddings.\NB{ref?}
\end{exm}

\subsection{The étale slice filtration} \label{subsec:etale-slice}
\begin{lem} \label{lemm:preserve-limit}
Let $\scr C$ be a presentable stable $\infty$-category and $G: \scr C \to \ShvSp$ an exact functor.
Suppose given a sequence of full subcategories $\dots \subset \scr C_i \subset \scr C_{i-1} \subset \dots \subset \scr C$ such that
\begin{enumerate}
\item $\scr C_i$ is closed under extensions,
\item if $\dots\to E^2 \to E^1 \to E^0 \in \scr C_i$ is a tower, then $\lim_i E^i \in \scr C_{i-1}$, and
\item $G(\scr C_i) \subset \ShvSp_{\ge d_i}$, with $\lim_i d_i = \infty$.
\end{enumerate}

Now let \[ \dots \to F^2 \to F^1 \to F^0 \to F^{-1} \to \dots \in \scr C \] be a tower.
Assume that $\fib(F^{i+1} \to F^i) \in \scr C_{e_i}$, with $\lim_i e_i = \infty$.
Then \[ G(\lim_i F^i) \wequi \lim_i G(F^i). \]
\end{lem}
\begin{proof}
Let $F' = \lim_i F^i$ and $F_i = \fib(F' \to F^i)$.
We have a fiber sequence \[ \lim_i G(F_i) \to G(F') \to \lim_i G(F^i), \] whence we need to prove that $\lim_i G(F_i) = 0$.
Put $F_i^n = \fib(F^n \to F^i)$, so that $F_i = \lim_n F_i^n$.
Let $N > 0$.
By (3), it suffices to show that $F_i \in \scr C_N$ for $i$ sufficiently large.
By (2), for this it is enough to show that $F_i^n \in \scr C_{N+1}$ for $i$ sufficiently large.
By (1), for this it suffices that $F_i^{i+1} \in \scr C_{N+1}$ for $i$ sufficiently large.
This holds by assumption.
\end{proof}

\begin{prop} \label{prop:etale-slice}
Let $S$ be essentially smooth over a Dedekind scheme, $1/\ell \in S$, $\pvcd_\ell(S) < \infty$.
Suppose $E \in \SH(S)$ such that
\begin{enumerate}
\item $s_n(E/\ell) = \bigoplus_j \Sigma^{e_j^{(n)},n} H\Z/\ell$, and
\item $\lim_{n \to \infty} \inf\{e_j^{(n')} \mid n' \ge n\} = \infty.$\NB{actually could replace $\inf\{e_j\}$ here by $\inf\{|e_j|\}$, is that useful?}
\end{enumerate}
Then \[ L_\et(\lim_n f^n E)_\ell^\comp \wequi \lim_n L_\et(f^n E)_\ell^\comp. \]
\end{prop}
\begin{proof}

If $p: S' \to S$ is smooth, then $p^*$ commutes with limits, colimits and effective covers and preserves generating étale equivalences and étale local objects.
It follows that $p^*$ commutes with $L_\et$, $f^n$ and $(\ph)_\ell^\comp$.
If $f$ is itself an étale covering, then $p^*$ is conservative on étale local objects.
It follows that we may replace $S$ by an étale covering and hence we may assume both that $\pcd_\ell(S) < \infty$ and that there is a good $\tau$-self map mod $\ell$ \cite[\S4.5]{bachmann-bott}.
Moreover we may assume that either $\ell$ is odd or $\rho$ is nilpotent.
It now follows from \cite[Theorem 7.4, \S2.2]{bachmann-bott} that for any $F \in \SH(S)$ we have \begin{equation}\label{eq:invert-tau-Let} L_\et(F/\ell) \wequi (F/\ell[\tau^{-1}])^{(+)}, \end{equation} where $(\ph)^{(+)}$ means passage to the $+$-part if $\ell$ is odd.

We seek to show that $L_\et(\lim_n f^n E)/\ell \to \lim_n L_\et(f^n E)/\ell$ is an equivalence.
It suffices to check this on $\map(p_\# \Sigma^{i,j} \1_{S'}, \ph)$ for all $p: S' \to S$ smooth.
As before since $p^*$ commutes with all our operations (and $\pcd_\ell(S') < \infty$, e.g. by \cite[Lemma 2.12(2)]{bachmann-SHet}), we may replace $S$ by $S'$, and hence it suffices to prove that we obtain an isomorphism on $\pi_{**}$.
For this we apply Lemma \ref{lemm:preserve-limit}, with $\scr C = \SH(S)$, $F^i = f^i(E)$ and $G = \map(\Sigma^{0,*}\1, L_\et(\ph)/\ell)$.
We let $\scr C_i \subset \SH(S)$ consist of those spectra such that $\pi_{p,w}(E/\ell) = 0$ for $p < i, w \in \Z$.
Property (1) is clear and (2) follows from the Milnor exact sequence.\NB{ref?}
Property (3) follows from \eqref{eq:invert-tau-Let}.

We thus need to establish the following: for $i \in \Z$ there exists $N$ such that for all $n>N$, $j<i$ and $w \in \Z$ we have $\pi_{j,w}(s_{n} E/\ell) = 0$.
By assumptions (1) and (2), for this it is enough to prove that there exists $M$ such that for all $i<M$ and $w \in \Z$ we have $\pi_{i,w}(H\Z/\ell) = 0$.
This follows from the last statement of Proposition \ref{prop:HZ-CQL}.
\end{proof}

\begin{cor} \label{cor:CQL-HZ-slices}
Let $S$ be essentially smooth over a Dedekind scheme, $1/\ell \in S$, $\pcd_\ell(S) < \infty$.
Suppose $E \in \SH(S)$ such that
\begin{enumerate}
\item $s_n(E/\ell) = \bigoplus_j \Sigma^{2n-e_j^{(n)},n} H\Z/\ell$, $e_j^{(n)} \ge 0$, and
\item $\lim_{n \to \infty} \inf\{2n'-e_j^{(n')} \mid n' \ge n\} = \infty.$
\end{enumerate}
If $H\Z/\ell$ satisfies $\CQL_s(\ell,S)$ then so does $\lim_n f^n E$.
\end{cor}
\begin{proof}
Note that our assumptions are stronger than what is asked in Proposition \ref{prop:etale-slice}.

By Lemma \ref{lemm:CQL-basics}(4), Remark \ref{rmk:extensions} and Proposition \ref{prop:etale-slice}, it suffices to prove that \[ \lim_n \fib(f^n E/\ell \to L_\et f^n E /\ell) =: \lim_n F_n \] has homotopy groups concentrated in Chow degrees $<s-1$.
It is enough to prove this for $F_n$.
For $m \le n$ let $E_n^m = f_m f^n E$, so that $f^n E \wequi \colim_m E_n^m$.\NB{ref?}
As in the proof of Lemma \ref{lemm:CQL-basics}(2), the functor $\map(\Sigma^{p,q} \1, L_\et(\ph)/\ell)$ preserves colimits\NB{make this separate statement?}.
It follows that $F_n \wequi \colim_m F_n^m$, where $F_n^m = \fib(E_n^m/\ell \to L_\et E_n^m/\ell)$.
Thus it suffices to show that $F_n^m$ has homotopy groups concentrated in Chow degrees $<s-1$.
We prove this by induction on $m$, the case $n=m$ being trivial since $F_m^m=0$.
The cofiber sequence $E_n^m \to E_n^{m-1} \to s_{m-1}E$ shows that it suffices to prove that $s_{m-1}E/\ell$ satisfies $\CQL_s(\ell,S)$.
By assumption, $H\Z/\ell$ satisfies $\CQL_s(\ell,S)$, and hence so does $\Sigma^{2n-e,n}H\Z/\ell$ for any $e \ge 0, n \in \Z$.
By Lemma \ref{lemm:CQL-basics}(2), so does any sum thereof, and hence so does $s_{m-1}E/\ell$ by assumption (1).
\end{proof}

\begin{exm} \label{ex:CQL-sc-KGL}
Proposition \ref{prop:etale-slice} and Corollary \ref{cor:CQL-HZ-slices} apply whenever $s_n E$ is a sum of copies of $\Sigma^{2n,n} H\Z$ ($\ell$-adically), so for example to $E = \MGL$ or $E = \KGL$ (by Remark \ref{rmk:MGL-slices}).
They also apply to $E=\1$, by \cite[Theorem 2.12]{rondigs2016first} (and \cite[Theorem 5.1.23(a)]{GreenBook}).\NB{The last two references together show that $s_n \1$ is a sum of $\Sigma^{a,n}H\Z/?$, where $a$ ranges between $n$ and $2n$}
\end{exm}

\begin{exm} \label{exm:CQL-HZ-MGL}
In the situation of Corollary \ref{cor:CQL-HZ-slices}, we find that $\MGL$ satisfies $\CQL_s(\ell,S)$ as soon as $H\Z/\ell$ does.
This follows from Example \ref{ex:CQL-sc-KGL} and Lemma \ref{lem:MGL-sc} (slice completeness of $\MGL$).
\end{exm}

\subsection{Extensions for the prime $2$} \label{subsec:prime-2}
Throughout this subsection we assume that $1/2 \in S$.
Recall the $b$-topology and localization from \S\ref{subsec:b-top}.
We can use it to define the $b$-Chow--Quillen--Lichtenbaum property, by just replacing references to the étale topology in the definition of the Chow--Quillen--Lichtenbaum property by the $b$-topology.
In this subsection we will do two things.
Firstly we will show that for some spectra (those whose $C_2$-equivariant realizations are $2$-adically Borel), the $b$-Chow--Quillen--Lichtenbaum property is equivalent to the usual Chow--Quillen--Lichtenbaum property.
Secondly we will show that the $b$-Chow--Quillen--Lichtenbaum property holds for the same spectra as previously ($H\Z$, $\MGL$, etc.), replacing assumptions on $\cd_2$ by assumptions on $\vcd_2$.
Combining these two facts we learn that some spectra (e.g. $\KGL$ and $\MGL$) satisfy the Chow--Quillen--Lichtenbaum property as soon as a bound $\vcd_2$ (instead of $\cd_2$) holds.

\subsubsection{Comparison and $L_b$ and $L_\et$}
If $x$ is the spectrum of a real closed field, there is a canonical equivalence $\SH_b(x)_2^\comp \wequi \GenSp(BC_2)_2^\comp$ \cite[Theorem 4.20]{elmanto2019scheiderer}.
We write \[ r_{C_2}: \SH(x) \to \SH_b(x)_2^\comp \wequi \GenSp(BC_2)_2^\comp \] for the composite of $b$-localization and the above equivalence.
\begin{lem}
Let $1/2 \in S$, $\dim S < \infty$ and $E \in \ShvSp(S)$.
Then $L_b(E)_2^\comp \wequi L_\et(E)_2^\comp$ if and only if, for every morphism $f_x: x \to S$ where $x$ is the spectrum of a real closed field, the genuine $C_2$-sepectrum \[ r_{C_2}(f_x^* E)_2^\comp \in \GenSp(BC_2) \] is Borel.
\end{lem}
\begin{proof}
Since the $b$-topology is finer than the étale topology, we have $L_\et E \wequi L_\et L_b E$.
It is thus sufficient (and necessary) to prove that $L_b(E)_2^\comp$ satisfies étale descent.
The result now follows by the exact same argument is in the proof of \cite[Theorem 6.10]{elmanto2019scheiderer}.
\end{proof}

\begin{exm} \label{ex:MUR-borel}
This result applies when $E = \MGL$, $E=\KGL$ or $E=\1$, in which case $r_{C_2} = \MU_\RR$, $\KU_\RR$ or $\1$, and the Borel property holds by \cite[Corollary 7.5]{MR2240234} \cite[Theorem 4.1]{hu-kriz-MUR} \cite{lins-theorem}\NB{more explicit reference for last?}.
\end{exm}

\subsubsection{$b$-motivic cohomology}
Recall from Lemma \ref{lemm:b-pullback-square} that there is a pullback square
\begin{equation} \label{eq:b-pullback-HZ2}
\begin{CD}
L_b H\Z/2 @>>> L_\et H\Z/2 \\
@VVV        @VVV \\
H\Z/2[\rho^{-1}] @>>> L_\et(H\Z/2)[\rho^{-1}].
\end{CD}
\end{equation}

\begin{lem} \label{lem:b-mot-first}
Let $S$ be essentially smooth over a Dedekind scheme.
\begin{enumerate}
\item We have \[ H^{**}_\et(S, \Z/2)[\rho^{-1}] \wequi H^*_\ret(S, \Z/2)[\rho^{\pm}, \tau^{\pm}] \] and \[ H^{**}(S, \Z/2)[\rho^{-1}] \wequi H^*_\ret(S, \Z/2)[\rho^{\pm}, \tau]. \]
\item If $H^*_\ret(S, \Z/2) = 0$ for $*>0$, then $\pi_{p,q} L_b H\Z/2 \wequi \pi_{p,q} L_\et H\Z/2$ for $p \ge q$.
\item If $H^*_\et(S[\sqrt{-1}], \Z/2) = 0$ for $*>d$, then $\pi_{p,q} L_b H\Z/2 \wequi \pi_{p,q} H\Z/2[\rho^{-1}]$ for $p <-d$.
\end{enumerate}
\end{lem}
\begin{proof}
(1) $H^{**}_\et(S, \Z/2)[\rho^{-1}] \wequi H^*_\et(S, \Z/2)[\tau^{\pm 1}, (-1)^{-1}]$, where $(-1) = \tau^{-1} \rho \in H^1_\et(S, \Z/2)$.
The first claim now follows from \cite[Corollary 7.19]{real-and-etale-cohomology}.
For the second claim, since both sides (upgrade to statements about sheaves of spectra which) satisfy Nisnevich descent, we may assume that $S$ is henselian local.
Then $H^{**}(S, \Z/2)$ is a truncation of $H^{**}_\et(S, \Z/2)$, and the claim follows by inspection.

(2) By the long exact sequence associated with \eqref{eq:b-pullback-HZ2}, it suffices to prove that for $p \ge q$ we have $\pi_{p,q} H\Z/2[\rho^{-1}] \wequi \pi_{p,q} L_\et(H\Z/2)[\rho^{-1}]$.
This follows from (1).

(3) Again using the long exact sequence, it suffices to prove that $\pi_{p,q} L_\et H\Z/2 \to \pi_{p,q} L_\et H\Z/2[\rho^{-1}]$ is an isomorphism for $p < -d$ and surjective for $p = -d$.
This will hold if $(-1): H^*_\et(S, \Z/2) \to H^{*+1}_\et(S, \Z/2)$ is surjective for $*=d$ and an isomorphism for $*>d$.
This follows from the Gysin sequence \[ \dots \to H^*_\et(S, \Z/2) \to H^*_\et(S[\sqrt{-1}], \Z/2) \to H^*_\et(S, \Z/2) \xrightarrow{(-1)} H^{*+1}_\et(S, \Z/2) \to \cdots . \]
(The Gysin sequence can be seen as arising from the Hochschild--Serre spectral sequence; see e.g. \cite[Proof of Lemma 7.2]{bachmann-bott}.)
\end{proof}

\begin{prop} \label{prop:HZ-bCQL-pre}
Let $S$ be essentially smooth over a Dedekind scheme and $1/2 \in S$.
Consider the fiber sequence \[ F \to H\Z/2 \to L_b H\Z/2. \]
Then $\pi_{p,q}F = 0$ in the following cases:
\begin{enumerate}
\item $S$ henselian local with residue field $k$ and $q \le -\vcd_2(k)$
\item $q \le -\pvcd_2(S)$
\item $S$ henselian local with residue field $k$ and $c(p,q) \ge \vcd_2(k) - 2$
\item $c(p,q) \ge \pvcd_2(S) - 2$
\end{enumerate}
\end{prop}
\begin{proof}
(1) Since $\cd_\ret(S) = 0$ \cite[Proposition 19.2.1]{real-and-etale-cohomology}, Lemma \ref{lem:b-mot-first} implies that $\pi_{p,q} L_b H\Z/2 \wequi \pi_{p,q} L_\et H\Z/2$ for $p \ge q$, and $\pi_{p,q} L_b H\Z/2 \wequi 0$ for $p < \min\{q,-d\}$ where $d=-\vcd_2(k)$.
Thus if $q \le -d$ then the vanishing holds for $p<q$, and the result follows.

(2) Using the descent spectral sequence, it suffices to prove this for $S$ replaced by one of its henselizations.
Now we are in situation (1).

(3) Since $L_\et L_b H\Z/2 \wequi L_\et H\Z/2$, the map $\pi_{**} H\Z/2 \to \pi_{**} L_b H\Z/2$ is a monomorphism.
We thus need to prove that $\pi_{p,q} H\Z/2 \to \pi_{p,q} L_b H\Z/2$ is an isomorphism in Chow degrees $\ge \vcd_2(k) - 1$.
The map is an isomorphism whenever $p \ge q$ by Lemma \ref{lem:b-mot-first}(2); hence we need to prove that $\pi_{p,q} L_b H\Z/2 = 0$ whenever $p<q$ and $p-2q=c(p,q) \ge \vcd_2(k)-1$.
We can rewrite this is as $0 > p-q \ge q+\vcd_2(k)-1$, which implies that $q < -\vcd_2(k) + 1$ and hence $p < -\vcd_2(k)$.
The desired vanishing thus follows from Lemma \ref{lem:b-mot-first}(3).

(4) We have $\ul{\pi}_{p,q}(F)|_{S_\Nis} = 0$ for $c(p,q) \ge \pvcd_2(S) - 2$ by (3).
The result follows as in Remark \ref{rmk:CQL-local} (i.e. using a descent spectral sequence).
\end{proof}

\subsubsection{Slice filtrations}
\begin{prop} \label{prop:ret-slice}
Let $S$ be essentially smooth over a Dedekind scheme, $1/2 \in S$.
Suppose $E \in \SH(S)$ such that
\begin{enumerate}
\item $s_n(E/2) = \bigoplus_j \Sigma^{e_j^{(n)},n} H\Z/2$, and
\item $\lim_{n \to \infty} \inf\{e_j^{(n')}-n' \mid n' \ge n\} = \infty.$
\end{enumerate}
Then \[ (\lim_n f^n E)[\rho^{-1}]/2 \wequi \lim_n (f^n(E)[\rho^{-1}])/2. \]
If in addition $\pvcd_2(S) < \infty$, then this is also equivalent to \[ (\lim_n L_b(f^n E))[\rho^{-1}]/2. \]
\end{prop}
\begin{proof}
As all our assumptions and operations being compatible with smooth base change, it suffices to prove that we get isomorphisms on $\pi_{**}$.
Our assumptions being stable under replacing $E$ by $\Sigma^{p,q} E$, we need only establish an isomorphism on $\pi_0$.
We apply Lemma \ref{lemm:preserve-limit} with $\scr C = \SH(S)$ and $G(F) = \map(\1, F[\rho^{-1}]/2)$.
Let $\scr C_i$ be the category of those spectra $F$ such that $\pi_{n-a,-a}(F/2) = 0$ for $a \ge 0$, $n < i$.
Conditions (1) and (2) are clear, and (3) holds because $\rho \in \pi_{-1,-1} \1$.

For the first equivalence we take $F^i = f^i E$.
We thus need to prove that given $N$, $s_n E \in \scr C_N$ for $n$ sufficiently large.
In other words \[ 0 = \pi_{m-a,-a} \Sigma^{e_j^{(n)},n} H\Z/2 = \pi_{m + (n-e_j^{(n)})-a-n,-a-n} H\Z/2 \] for $m<N,a \ge 0$ and $n$ sufficiently large.
Since $n-e_j^{(n)}$ becomes arbitrarily small by assumption and $\dim S < \infty$, this follows from the fact that $\ul{\pi}_{p+*,*} H\Z/2 = 0$ for $p<0$ and $* \in \Z$.

For the second equivalence we take $F^i = L_b f^i E$.
Since the real étale topology is finer than the $b$-topology, $L_b(F)[\rho^{-1}] \wequi F[\rho^{-1}]$.
We thus need to prove that given $N$, $L_b s_n E \in \scr C_N$ for $n$ sufficiently large.
In other words \[ 0 = \pi_{m-a,-a} \Sigma^{e_j^{(n)},n} L_b H\Z/2 = \pi_{m + (n-e_j^{(n)})-a-n,-a-n} L_b H\Z/2 \] for $m<N,a \ge 0$ and $n$ sufficiently large.
Since $\vcd_2(S) < \infty$, we get $\pi_{p,*} L_b H\Z/2 \wequi \pi_{p,*} H\Z/2[\rho^{-1}]$ for $p$ sufficiently small and $* \in \Z$ arbitrary, by Lemma \ref{lem:b-mot-first}(2).
Since $m + (n-e_j^{(n)})-a-n$ becomes arbitrarily small, we may replace $L_b H\Z/2$ by $H\Z/2[\rho^{-1}]$, and the desired vanishing follows as in the previous case.
\end{proof}

\begin{prop} \label{prop:b-slice}
Let $S$ be essentially smooth over a Dedekind scheme, $1/2 \in S$, $\pvcd_2(S) < \infty$.
Suppose $E \in \SH(S)$ such that
\begin{enumerate}
\item $s_n(E/2) = \bigoplus_j \Sigma^{e_j^{(n)},n} H\Z/2$, and
\item $\lim_{n \to \infty} \inf\{e_j^{(n')}-n' \mid n' \ge n\} = \infty.$
\end{enumerate}
Then \[ L_b(\lim_n f^n E)_2^\comp \wequi \lim_n L_b(f^n E)_2^\comp. \]
\end{prop}
\begin{proof}
We seek to show that $L_b(\lim_n f^n E)/2 \wequi \lim_n L_b(f^n E)/2$.
It suffices to prove this respectively after killing and inverting $\rho$.
Recall (Corollary \ref{cor:Lb-mod-rho}) that for any $F \in \SH(S)$ we have $L_b(F)/(2,\rho) \wequi L_\et(F)/(2,\rho)$ and $L_b(F)/2[\rho^{-1}] \wequi F/2[\rho^{-1}]$.
Consequently after killing $\rho$, the required equivalence is \[ L_\et(\lim_n f^n E)/(2,\rho) \wequi \lim_n L_\et(f^n E)/(2, \rho). \]
This holds by Proposition \ref{prop:etale-slice}.
For inverting $\rho$, we get \[ L_b(\lim_n f^n E)/2[\rho^{-1}] \wequi (\lim_n f^n E)/2[\rho^{-1}] \wequi (\lim_n L_b(f^n E))[\rho^{-1}]/2; \] here the second equivalence is by Proposition \ref{prop:ret-slice}.
\end{proof}

\subsubsection{Chow--Quillen--Lichtenbaum at $\ell=2$}
\begin{dfn}[$b$-Chow--Quillen--Lichtenbaum property]
Let $E \in \SH(S)$.
We say that $E$ satisfies $\bCQL_s(S)$ (or $\bCQL_s$ for short) if the map \[ E_2^\comp \to L_b(E)_2^\comp \] induces an isomorphism on $\pi_{i,j}$ in Chow degrees $\ge t$, and a monomorphism in Chow degree $t-1$.
\end{dfn}

\begin{lem} \label{lemm:bCQL-basics}
\begin{enumerate}
\item Spectra satisfying $\bCQL_s$ are closed under extensions.
\item If $\dim S < \infty$, then they are also closed under coproducts.
\item $E$ satisfies $\bCQL_s$ if and only if $E_\ell^\comp$ does.
\item If $E \in \SH(S)$ and $E/\ell$ satisfies $\bCQL_s$, then so do $E/\ell^n$ and $E$.
\end{enumerate}
\end{lem}
\begin{proof}
The proof of Lemma \ref{lemm:CQL-basics} goes through essentially unchanged.
The only non-trivial input is that $L_b(\ph)_\ell^\comp$ preserves colimits, for which see Corollary \ref{cor:Lb-colim}.
\end{proof}

\begin{cor} \label{cor:HZ-bCQL}
Let $S$ be essentially smooth over a Dedekind scheme and $1/2 \in S$.
Then $H\Z$, $H\Z_2^\comp$ and $H\Z/2^n$ satisfy $\bCQL_s$ in the following cases:
\begin{enumerate}
\item $S$ henselian local with residue field $k$ and $\vcd_2(k) \le s+1$
\item $\pvcd_2(S) \le s+1$.
\end{enumerate}
\end{cor}
\begin{proof}
By Lemma \ref{lemm:bCQL-basics}(3,4) it suffices to prove the claim for $H\Z/2$.
This is a reformulation of Proposition \ref{prop:HZ-bCQL-pre}.
\end{proof}

\begin{cor} \label{cor:bCQL-HZ-slices}
Let $S$ be essentially smooth over a Dedekind scheme, $1/2 \in S$, $\pvcd_2(S) < \infty$.
Suppose $E \in \SH(S)$ such that
\begin{enumerate}
\item $s_n(E/2) = \bigoplus_j \Sigma^{2n-e_j^{(n)},n} H\Z/2$, $e_j^{(n)} \ge 0$, and
\item $\lim_{n \to \infty} \inf\{n'-e_j^{(n')} \mid n' \ge n\} = \infty.$
\end{enumerate}
If $H\Z/2$ satisfies $\bCQL_s(S)$ then so does $\lim_n f^n E$.
\end{cor}
\begin{proof}
The proof of Corollary \ref{cor:CQL-HZ-slices} goes through essentially unchanged.
\end{proof}

\begin{cor} \label{cor:CQL-MGL}
Let $S$ be essentially smooth over a Dedekind scheme, $1/\ell \in S$.
Then $\MGL$ and $\KGL$ satisfy $\CQL_s(\ell, S)$, provided that $\pvcd_\ell(S) \le s+1$.
If $S$ is henselian local with residue field $k$, it also suffices that $\vcd_\ell(k) \le s+1$.
\end{cor}
\begin{proof}
If $\ell \ne 2$ this follows from Example \ref{exm:CQL-HZ-MGL} and Proposition \ref{prop:HZ-CQL}.
If $\ell = 2$ this follows from Example \ref{ex:MUR-borel} (saying that $\CQL_s = \bCQL_s$ for $\MGL, \KGL$), Lemma \ref{lem:MGL-sc} and Remark \ref{rmk:MGL-slices} (saying that these spectra are slice complete and determining their slices), and Corollaries \ref{cor:bCQL-HZ-slices} and \ref{cor:HZ-bCQL}.
\end{proof}

\subsection{Application: Hypercompleteness of motivic invariants} \label{subsec:SH-et-nh}
\newcommand{\nh}{\mathrm{nh}}
For once, we shall discuss non-hypercompleted topoi.
Denote by $\SH_\et^\nh(S)$ the localization of $\SH(S)$ at the (iterated $\P^1$-desuspensions of suspension spectra of) Čech nerves of étale coverings.
Recall that we write $\SH_\et(S)$ for the hypercompleted version.

\begin{lem}\label{lem:i*-Let}
Let $p: T \to S$ be an integral morphism of schemes.
Then $p_*: \PSh_\Sigma(\Sm_T) \to \PSh_\Sigma(\Sm_S)$ commutes with $L_\et^\nh$.
(The same is true for $L_\et$ if $p$ is finite.)

In the same situation $p_*: \SH(T) \to \SH(S)$ commutes with $L_\et^\nh$ (respectively $L_\et$).
\end{lem}
\begin{proof}
Clearly $p_*$ preserves étale sheaves, so we need only prove that it preserves étale equivalences.
The same argument as in \cite[Proposition 2.11]{norms} works.
(The statement about $L_\et$ can be checked on stalks directly.)

Now consider the case of $\SH(\ph)$.
We model $\SH(T)$ as a localization of the category of $\P^1$-prespectra, that is, sequences $(E^0,E^1, \dots)$ with $E^i \in \PSh_\Sigma(\Sm_T)$ together with maps $E^i \to \Omega_{\P^1} E^{i+1}$.
The localization $\SH(T)$ is obtained by restricting to prespectra which are $\Omega$-spectra (i.e. $\Omega_{\P^1} E^{i+1} \wequi E^i$) and such that each $E^i$ is motivically local.
Note that on the level of prespectra, $p_*E = (p_*(E^i))^i$, and the functor $p_*$ preserves $\Omega$-spectra and Nisnevich/$\A^1$/étale local objects.
To prove what we want, it suffices to prove that $p_*$ at the level of prespectra preserves generating étale (hyper)equivalences.
(Again argue as in \cite[Proposition 2.11]{norms}.)
These are maps $(E^i)^i \to (F^i)^i$ such that each $E^i \to F^i$ is an étale (hyper)equivalence.
We have seen that $p_*$ preserves these in the first part of the argument.
\end{proof}

If $\scr P$ is a set of primes and $\scr C$ stable $\infty$-category, write $\scr C_{(\scr P)} \subset \scr C$ for the subcategory of $\scr P$-local objects, i.e. those objects on which each $p \not\in \scr P$ is invertible.
\begin{thm} \label{thm:SH-net-het}
Let $S$ be noetherian\NB{remove noetherian assumption somehow?}, finite dimensional, and assume that $\sup_{s \in S} \vcd_\scr{P}(s) < \infty$.
Then $\SH_\et^\nh(S)_{(\scr P)} \wequi \SH_\et(S)_{(\scr P)}$.
\end{thm}
In other words, if $E \in \SH(S)_{(\scr P)}$ satisfies étale descent, it automatically satisfies hyperdescent.
\begin{proof}
Since the problem is étale local, we can pass to a covering where all residue fields are unorderable \cite[Example 2.14]{bachmann-SHet} and $S$ is qcqs (e.g. affine), so we may assume that every $X \in \Sm_S$ has finite $\scr P$-étale cohomological dimension \cite[Theorem 2.10]{bachmann-SHet}.
It follows that $\scr P$-local étale hypersheaves of spectra on $\Sm_S$ are closed under colimits in presheaves, and that étale hypersheafification is smashing both, for $\scr P$-local presheaves of spectra and in $\SH(S)_{(\scr P)}$.
The former claim is proved in \cite[Corollary 4.40]{clausen2019hyperdescent}, and the latter in \cite[Theorem 1.4]{bachmann-bott}\footnote{In this reference, the same statement is proved after completion at $(\ell, \rho)$. The completion at $\rho$ is vacuous because of our unorderability assumption \cite[Remark 2.5]{bachmann-bott}. The rational statement is well-known, see e.g. \cite[proof of Theorem 7.2]{bachmann-SHet}.}.
We will prove that $L_\et^{nh} \1_{(\scr P)}$ is an étale hypersheaf; since étale hyperlocalization is smashing this will establish the theorem.

Let $\eta \in S$ be a generic point.
Then $\{\eta\}$ is the intersection of all open subsets $U \subset S$ containing $\eta$.
Write $j_U: U \to S$ for the open immersion, $i_U: Z \to S$ for the complementary closed immersion, and $j_\eta: \eta \to S$ for the inclusion of the generic point.
By continuity \cite[Example 4.3.2]{triangulated-mixed-motives} we get\NB{This would be false a priori for $\SH_\et^\nh$.} \[ j_{\eta*}j_\eta^* \1_{(\scr P)} \wequi \colim_U j_{U*}j_U^* \1_{(\scr P)} \in \SH(S). \]
Consider the fiber sequence $F \to \1_{(\scr P)} \to j_{\eta*}j_\eta^*\1_{(\scr P)}$.
Then we have $F \wequi \colim_U i_{U*}i_U^! \1_{(\scr P)}$ by localization (see e.g. \cite[Proposition 2.3.3(3)]{triangulated-mixed-motives}, and hence by Lemma \ref{lem:i*-Let} we get \[ F \wequi \colim_U i_{U*}i_U^! \1_{(\scr P)} \stackrel{L_\et^\nh}{\wequi} \colim_U L_\et^\nh i_{U*}i_U^! \1_{(\scr P)} \wequi \colim_U i_{U*} L_\et^\nh i_U^! \1_{(\scr P)}. \]
By noetherian induction $i_{U*}L_\et^\nh i_U^!\1_{(\scr P)}$ is an étale hypersheaf and hence (étale hyperlocalization being smashing) so is $\colim_U i_{U*} L_\et^\nh i_U^!$.
We deduce that $L_\et^\nh F \wequi L_\et F$.
It remains to prove that $L_\et^\nh j_{\eta*}j_\eta^* \1_{(\scr P)} \wequi L_\et j_{\eta*}j_\eta^* \1_{(\scr P)}$.
Form the fiber sequences \[ G \to j_{\eta*}j_\eta^* \1_{(\scr P)} \to j_{\eta*} L_\et j_\eta^* \1_{(\scr P)} \quad\text{and}\quad K \to L_\et^\nh G \to L_\et G. \]
It suffices to prove that $K \wequi L_\et K$.\NB{Then in fact $K \wequi 0$, but that's not relevant.}
Let $\ell \in \scr P$.
We shall show that $K/\ell \wequi 0$.
Assuming this, we deduce that $K$ is a rational and hence a hypercomplete (see e.g. \cite[Proposition 2.31]{clausen2019hyperdescent}) étale sheaf, as needed.

Suppose first that $char(\eta) = \ell$.
Let $W \subset S$ be the closed subset on which $\ell=0$; in particular $\eta \in W$.
Then $j_{\eta*}$ factors through $i_{W*}$ and hence it is enough to prove that for $E \in \SH(W)$ we have $L_\et^\nh i_{W*} E/\ell \wequi 0$.
By Lemma \ref{lem:i*-Let} this is the same as $i_{W*}L_\et^\nh E/\ell$, which is zero by \cite[Corollary 1.4]{bachmann-linearity}.

Finally assume that $char(\eta) \ne \ell$.
Write $G=(G^0,G^1,\dots)$ as an $\Omega$-$\Gm$-spectrum (so each $G^i$ is a presheaf of spectra on $\Sm_S$).
For any $X \in \Sm_\eta$, the sphere $\1_X$ satisfies $\CQL_s(\ell, X)$ for some $s<\infty$ (depending only on $\dim X$).
This follows from $\ell$-adic slice-completeness of $\1$ (see e.g. \cite[Corollary 5.13]{bachmann-bott}; again completion at $\rho$ is unnecessary because of unorderability), Proposition \ref{prop:HZ-CQL} ($\CQL$ for $H\Z/\ell$) and Example \ref{ex:CQL-sc-KGL}, Corollary \ref{cor:CQL-HZ-slices} (deducing $\CQL$ for $\1$).
We deduce that for any $X \in \Sm_S$, the restriction of $G^i/\ell$ to $\Et_X$ is bounded above.
Note that $G \stackrel{L_\et^\nh}{\wequi} G'$, where \[ G'_i = \colim_j \Omega_{\Gm}^j L_{\A^1}^{pre}L_\et^{\nh,pre}G^{j+i}, \] $L_{\A^1}^{pre}$ denotes the $\A^1$-localization of spectral presheaves (i.e. the singular construction) and $L_\et^{nh,pre}$ denotes the étale localization of spectral presheaves.
Since $G^i/\ell$ is bounded above on every small étale topos, $L_\et^{\nh,pr}G^i/\ell$ is hypercomplete.
Since hypercomplete étale sheaves of spectra are closed under colimits in presheaves, $G'/\ell$ is actually an $\Omega$-$\Gm$-spectrum, motivically local, and étale hyperlocal.
In other words $G'/\ell \wequi L_\et^\nh G/\ell \wequi L_\et G/\ell$.

This concludes the proof.
\end{proof}

\begin{cor}
Let $S \in \Sch_{\Z[1/\ell]}$.
The functor $\ShvSp(S_\et^\nh)_\ell^\comp \to \SH_\et^\nh(S)_\ell^\comp$ is a localization, namely at hypercovers obtained by pullback from finite type schemes over $\Z[1/\ell]$.
\end{cor}
\begin{proof}
Since the functors \[ \ShvSp((\ph)_\et^\nh)_\ell^\comp, \SH_\et^\nh(\ph)_\ell^\comp: \Sch^\op \to Pr^L \] are continuous and satisfy Zariski descent \cite[\S3]{bachmann-SHet2}, we may assume that $S$ is of finite type over $\Z[1/\ell]$.
Thus we have reduced to Theorem \ref{thm:SH-net-het} and \cite[Theorem 6.6]{bachmann-SHet}.
\end{proof}

\section{Galois approximations} \label{sec:galois-approx}
\localtableofcontents

\bigskip
Let $S$ be a scheme.
The canonical symmetric monoidal functor $\ShvSp \to \SH(S)$ together with the invertible object $\Sigma^{0,1} \1 \in \SH(S)$ induces a canonical adjunction (see \S\ref{subsec:graded}) \[ e_\Ss: \ShvSp^\Ss \adj \SH(S): e_\Ss^*. \]
Our main aim in this section is to give, for certain $S$, a reconstruction of \[e_\Ss^*(\1_{\eta,\ell}^\comp) \in \CAlg(\ShvSp^\Ss) \] not involving $\SH(S)$.

The basic strategy for this is as follows.
Choose a ring spectrum $E \in \SH(S)$ such that $\Tot \CB(E)_\ell^\comp \wequi \1_{\ell,\eta}^\comp$.
We mainly have in mind $E = \MGL$ (or one of its variants, like an Adams summand), where the convergence is a standard result.
We intend to provide a reconstruction of $e_\Ss^*(\CB^\bullet(E)_\ell^\comp)$; totalizing this will fulfill our original aim.
To do this, we consider the map $\CB^\bullet(E)_\ell^\comp \to \CB^\bullet L_\et(E)_\ell^\comp$.
Using results from the previous section, we can often argue that the map has fiber with homotopy groups concentrated below a certain Chow degree.
Also we can sometimes show that $\CB^\bullet(E)_\ell^\comp$ has homotopy groups concentrated above a certain Chow degree.
This implies that we can, in good cases, recover $e_\Ss^* \CB^\bullet(E)_\ell^\comp)$ from $e_\Ss^* \CB^\bullet(L_\et E)_\ell^\comp)$ by a strategic truncation.
We thus need only produce a reconstruction of $e_\Ss^* \CB^\bullet L_\et(E)_\ell^\comp$, which we can do using results from the previous sections (in particular \S\ref{sec:etale-MGL}).

\subsubsection*{Organization}
We begin in \S\ref{subsec:chow-conn} by showing that, under appropriate assumptions on $S$, the homotopy groups of certain spectra are concentrated in non-negative Chow degrees.
We will use this to provide several approximations of $e_\Ss^*(\1_{\ell,\eta}^\comp)$ of increasing complexity, valid in increasing generality.
First, in \S\ref{subsec:coh-dim-1} we deal with the case of fields of cohomological dimension $\le 1$, where a simple truncation suffices.
Second, in \S\ref{subsec:coh-dim-2} we deal with cohomological dimension $\le 2$.
Here, in addition to performing a truncation, we also have to modify the lowest homotopy groups (where we know a priory that we must essentially get the Lazard ring, but the pure étale topology construction will produce something else).
Third, in \S\ref{subsec:recons-adams} we modify our strategy by replacing $\MGL$ by its Adams summand $e_0^{(\ell)} \MGL_{(\ell)}$, and employing a slightly different $t$-structure on $\ShvSp^\Ss$.
It turns out that this allows us to deal with fields of cohomological dimension $\le \ell-1$---an improvement over the previous construction as soon as $\ell \ge 5$.

In the final two subsections, namely \S\ref{subsec:cellular-reconstruction} and \S\ref{subsec:AT-reconstruction}, we use our reconstructions of $e_\Ss^*(\1_{\ell,\eta}^\comp)$ to write down \emph{Galois approximations}.
That is, by considering appropriate categories of modules, we provide reconstructions of entire subcategories of $\SH(S)_{\ell,\eta}^\comp$.
We also establish some basic functorialities, and for example show how to obtain a functor $\SH(\CC)_{\ell,\eta}^{\comp\cell} \to \SH(k)_{\ell,\eta}^\comp$ for any field $k$ containing all $\ell$-power roots of unity.

\subsubsection*{Standing notation}
We keep fixed the functor $e: \ShvSp^\Ss \to \SH(S)$ corresponding to $\Sigma^{0,1} \1$.
Lemma \ref{lem:graded-t} supplies us with a $t$-structure on $\ShvSp^\Ss$, coming from the function $\Ss \to \Z, n \mapsto 2n$.
Unless stated otherwise, this is the $t$-structure on $\ShvSp^\Ss$ that we use.

\subsection{Chow connectivity} \label{subsec:chow-conn}
\begin{lem} \label{lemm:chow-conn-permanence}
The category of spectra $E \in \SH(S)$ with $\pi_{**}(E)$ concentrated in Chow degrees $\ge 0$ is closed under colimits, extensions, $\Sigma^{2*,*}$, and completions.
\end{lem}
\begin{proof}
The claim about $\Sigma^{2*,*}$ is clear.
Since $\Sigma^{2n,n} \1 \in \SH(S)$ is compact, all other claims reduce to their analogs for $\ShvSp_{\ge 0}$, which are well-known.
\end{proof}

\begin{lem} \label{lemm:HZ-chow-conn}
Let $S$ be essentially smooth over a Dedekind scheme $D$.
Then $\pi_{**} H\Z_S$ is concentrated in non-negative Chow degrees.
\end{lem}
\begin{proof}
Let $p: S \to D$ be the structure morphism.
Write $z_S^n$ for the Bloch cycle complex of codimension $n$, viewed as a presheaf on $S_\Zar$.
Then $\map(\Sigma^{2n,n} \1, H\Z) \wequi \Gamma_\Zar(S, z_S^n)$ \cite[Theorem 7.18]{spitzweck2012commutative}.
It thus suffices to prove that $L_\Zar z_S^n$ is connective on global sections.
We claim that $p_* L_\Zar z_S^n \wequi L_\Zar p_* z_S^n$.
This may be checked on stalks, whence we may assume that $D$ is local; in this situation the claim follows from the fact that $z_S^n$ already satisfies Zariski descent \cite[Theorem 1.7 and following paragraph]{levine2001techniques}.
Since $z_S^n$ is sectionwise connective and $D$ has Zariski cohomological dimension $\le 1$, it thus suffices to prove that $H^1_\Zar(D, p_* (z_S^n)_0) = 0$.
By construction $(z_S^n)_0(U)$ is the group of codimension $n$ cycles on $U$.
It follows that the presheaf $(z_S^n)_0(\ph)$ is in fact a flasque Zariski sheaf.\NB{ref?}
It follows that the presheaf $p_* (z_S^n)_0$ is a flasque sheaf, and hence has vanishing cohomology as desired \cite[Tag 09SY]{stacks-project}.\NB{Flasque presheaves need \emph{not} have vanishing cohomology. E.g. consider constant presheaves...}
\end{proof}

\begin{cor} \label{cor:MGL-chow-conn}
If $\scr S$ is the set of primes not invertible in $S$, then $\pi_{**} \MGL_S[1/\scr S]$ is concentrated in non-negative Chow degrees.
The same holds for $\pi_{**} (\MGL_S)_\ell^\comp$, for any $\ell$ invertible on $\scr S$.
\end{cor}
\begin{proof}
We need only prove the first statement, since it implies the second by Lemma \ref{lemm:chow-conn-permanence}.
By Lemmas \ref{lem:MGL-sc} and \ref{lemm:sc-nondeg}, it suffices to prove that $\pi_{**} f^i \MGL_S[1/\scr S]$ is concentrated in Chow degrees $\ge 0$.
Via the cofiber sequences $s_i \to f^{i+1} \to f^i$ and Lemma \ref{lemm:chow-conn-permanence} this reduces to proving that $\pi_{**} s_i \MGL_S[1/\scr S]$ is concentrated in Chow degrees $\ge 0$.
This follows from Remark \ref{rmk:MGL-slices} and Lemma \ref{lemm:HZ-chow-conn}.
\end{proof}

To exploit the above properties, we use the following construction.
\begin{dfn}
Denote by $\tau^c_{\ge 0} \SH(S)_\ell^\comp \subset \SH(S)_\ell^\comp$ the subcategory generated under colimits and extensions by the objects $\Sigma^{2n,n} \1_\ell^\comp$ for $n \in \Z$.
This defines a $t$-structure \cite[Proposition 1.4.4.11(2)]{HA}, and we denote the truncation functors by $\tau^c_{\ge 0}, \tau^c_{\le 0}$ and so on.
\end{dfn}

\begin{exm} \label{ex:MGL-cellular}
We have $\MGL_\ell^\comp \in \tau^c_{\ge 0}\SH(S)_\ell^\comp$ by \cite[Theorem A.1]{BKWX}.
\end{exm}
\begin{rmk}
The functor $\tau^c_{\ge 0}: \SH(S)_\ell^\comp \to \SH(S)_\ell^\comp$ is lax symmetric monoidal, $\tau^c_{\ge 0}\SH(S)_\ell^\comp$ being closed under tensor products.
\end{rmk}
\begin{rmk} \label{rmk:tauc-vanish-meaning}
For $E \in \SH(S)_\ell^\comp$, we have $\tau_{\ge 0}^c E = 0$ if and only if $\pi_{**} E$ is concentrated in Chow degrees $<0$.
(Necessity is obvious, and for sufficiency we need to see that $\Map(X, E) = 0$ for all $X \in \tau_{\ge 0}^c \SH(S)_\ell^\comp$.
The collection of such $X$ is closed under colimits and extensions and contains $\Sigma^{2*,*}\1_\ell^\comp$ by assumption, whence is all of $\tau_{\ge 0}^c \SH(S)_\ell^\comp$.)
\end{rmk}

\begin{lem} \label{lem:chow-recover}
Let $E \in \SH(S)_\ell^\comp$.
\begin{enumerate}
\item If $E$ satisfies $\CQL_0(\ell, S)$, then $\tau^c_{\ge 0} E \wequi \tau^c_{\ge 0}L_\et(E)_\ell^\comp$.
\item If $E$ satisfies $\CQL_s(\ell, S)$ then we have a pullback square
\begin{equation*}
\begin{CD}
\tau^c_{\ge 0} E @>>> \tau^c_{\ge 0} L_\et(E)_\ell^\comp \\
@VVV                     @VVV \\
\tau^c_{< s} \tau^c_{\ge 0} E @>>> \tau^c_{< s} \tau^c_{\ge 0} L_\et(E)_\ell^\comp.
\end{CD}
\end{equation*}
\item Let $A \in \Alg(\tau^c_{\ge 0}\SH(S)_\ell^\comp)$ have homotopy groups concentrated in Chow degrees $\ge 0$.
  Let $E \in \Mod_A(\SH(S)_\ell^\comp)$.
  Then $e_\Ss^*(\tau^c_{\ge n} E) \wequi \tau_{\ge n}^\Ss e_\Ss^* E$, and similarly for $\tau^c_{\le n}$.
\end{enumerate}
\end{lem}
\begin{proof}
(1) This is a special case of (2).

(2) Taking vertical fibers, we need to prove that $\tau^c_{\ge s} E \wequi \tau^c_{\ge s} L_\et(E)_\ell^\comp$.
For this it suffices to prove that if $F=\fib(E \to L_\et(E)_\ell^\comp)$, then $F \in \tau_{<s-1}^c \SH(S)_\ell^\comp$ (indeed then $\Sigma F \in \tau_{<s}^c \SH(S)_\ell^\comp$ and we conclude from the fiber sequence $\tau_{\ge s}^c E \to \tau_{\ge s}^c L_\et(E)_\ell^\comp \to \tau_{\ge s}^c F \in \tau_{\ge s}^c \SH(S)_\ell^\comp$).
By Remark \ref{rmk:tauc-vanish-meaning}, this exactly means that $F$ has homotopy groups concentrated in Chow degrees $<s-1$.
This is the definition of $\CQL_s(\ell,S)$.

(3) Consider the forgetful functors \[ \Mod_A(\SH(S)_\ell^\comp) \to \SH(S)_\ell^\comp \quad\text{and}\quad \Mod_A(\SH(S)_\ell^\comp) \to (\ShvSp^\Ss)_\ell^\comp \to \ShvSp^\Ss. \]
All these categories carry evident $t$-structures (generated by the spheres of non-negative Chow degree), and it will suffice to show that all the functors are $t$-exact.
For the last functor, this follows from the fact that $\ShvSp_\ell^\comp \to \ShvSp$ is $t$-exact.\NB{ref?}
For the two other functors, the result follows from the following general fact: 
Consider an adjunction \[ L: \scr C \adj \scr D: R \] of stable presentable categories.
Give $\scr C$ the $t$-structure generated by a set $\scr G$ of objects, and $\scr D$ the $t$-structure generated by $L\scr G$.
Then $L$ is right $t$-exact and (hence) $R$ is left $t$-exact.
In order for $R$ to be $t$-exact, it suffices that $R$ preserves colimits (which will hold as soon as $\scr C$ is compactly generated by a collection $\scr F$ of objects such that $L\scr F$ consists of compact objects) and $RL\scr G \subset \scr C_{\ge 0}$.
This criterion applies in our situation: in order for $\Mod_A(\SH(S)_\ell^\comp) \to \SH(S)_\ell^\comp$ to be $t$-exact we need $A \in \tau_{\ge 0}^c\SH(S)_\ell^\comp$, and in order for $\Mod_A(\SH(S)_\ell^\comp) \to (\ShvSp^\Ss)_\ell^\comp$ to be $t$-exact we need $A$ to have homotopy groups in Chow degrees $\ge 0$.
\end{proof}

\subsection{Cohomological dimension $\le 1$} \label{subsec:coh-dim-1}
Recall from \S\ref{subsec:graded} the $t$-structure on $\ShvSp^\Ss$.
\begin{lem} \label{lem:reconst-from-Let-cd1}
Let $E \in \SH(S)$.
Assume that $E$ satisfies $\CQL_0(\ell, S)$ and $\pi_{**}(E_\ell^\comp)$ is concentrated in Chow degrees $\ge 0$.
Then \[ e_\Ss^*(E_\ell^\comp) \wequi \tau_{\ge 0}^\Ss e_\Ss^*(L_\et(E)_\ell^\comp), \] via the canonical map.
\end{lem}
\begin{proof}
Since $\pi_{**}(E_\ell^\comp)$ is concentrated in Chow degrees $\ge 0$ we have $e_\Ss^*(E_\ell^\comp) \wequi \tau^\Ss_{\ge 0}e_\Ss^*(E_\ell^\comp)$.
Since $\tau^\Ss_{\ge 0}e_\Ss^* \wequi \tau^\Ss_{\ge 0}e_\Ss^* \tau^c_{\ge 0}$ (note that $e_\Ss(\tau^\Ss_{\ge 0} \ShvSp^\Ss)_\ell^\comp \subset \tau^c_{\ge 0}\SH(S)_\ell^\comp$), it will suffice to show that $\tau_{\ge 0}^c(E_\ell^\comp) \wequi \tau_{\ge 0}^c(L_\et(E)_\ell^\comp).$
This holds by Lemma \ref{lem:chow-recover}(1).
\end{proof}

\begin{cor} \label{cor:reconst-1}
Suppose that $\dim S < \infty$, $\MGL$ satisfies $\CQL_0(\ell, S)$ and $\pi_{**}(\MGL_\ell^\comp)$ is concentrated in Chow degrees $\ge 0$.
Then \[ e_\Ss^*(\1_{\ell,\eta}^\comp) \wequi \Tot \tau^\Ss_{\ge 0} e_\Ss^* L_\et(\CB(\MGL))_\ell^\comp \in \CAlg(\ShvSp^\Ss). \]
The same is true with $\MGL$ replaced by any $E \in \CAlg(\SH(S))$ such that $\MGL_\ell^\comp \wequi E_\ell^\comp \in \Alg_{\scr E_0}(\SH(S))$.
\end{cor}
Here the truncation $\tau^\Ss_{\ge 0}$ is applied levelwise to the cosimplicial object $e_\Ss^* L_\et(\CB(\MGL))_\ell^\comp$.
\begin{proof}
We have \[ e_\Ss^*(\1_{\ell,\eta}^\comp) \wequi e_\Ss^* \Tot \CB(\MGL)_\ell^\comp \wequi \Tot\, e_\Ss^* \CB(\MGL)_\ell^\comp \wequi \Tot\, \tau_{\ge 0}^\Ss e_\Ss^* L_\et(\CB(\MGL))_\ell^\comp, \] where the first equivalence is since $\1_\eta^\comp \wequi \Tot \CB(\MGL)$ by \cite[Theorem 2.3]{bachmann-topmod}, the second because $e_\Ss^*$ is a right adjoint, and the third because $\MGL^{\wedge n}$ satisfies the assumptions of Lemma \ref{lem:reconst-from-Let-cd1}.

This argument remains valid when replacing $\MGL$ by $E$.
The main observation is that the map $\1_{\ell,\eta}^\comp \to \Tot \CB(E)_\ell^\comp$ is still an equivalence, since the spectrum underlying the totalization of the cobar construction of an $\scr E_\infty$-ring only depends on the underlying $\scr E_0$-ring structure (because the totalization of a cosimplicial object only depends on the underlying semi-cosimplicial object).
\end{proof}

Suppose that $1/\ell \in S$ and $S$ is locally $\ell$-étale finite.
Then we have an induced adjunction \[ L_\et(e_\Ss)_\ell^\comp: \ShvSp^\Ss \adj \SH_\et(S)_\ell^\comp \wequi \ShvSp(S_\et)_\ell^\comp: \Gamma_\et^\Ss(\ph). \]
Recall also the cyclotomic character \[ \pi: S_\et \xrightarrow{\tilde\pi} BG_\ell \to B\Z_\ell^\times \] from \S\ref{subsec:cyclo-char} and the spectrum $\MU_\ell \in \CAlg(\ShvSp(B_\et G_\ell)_\ell^\comp)$ from \S\ref{subsec:equivariant-j}.

\begin{thm} \label{thm:recons-1}
Suppose that $1/\ell \in S$, $\dim S < \infty$, $\MGL$ satisfies $\CQL_0(\ell, S)$ and $\pi_{**}(\MGL_\ell^\comp)$ is concentrated in Chow degrees $\ge 0$.
Then \[ e_\Ss^*(\1_{\ell,\eta}^\comp) \wequi \Tot \tau^\Ss_{\ge 0} \Gamma_\et^\Ss \tilde\pi^*(\CB(\MU_\ell))_\ell^\comp \in \CAlg(\ShvSp^\Ss). \]
\end{thm}
\begin{proof}
We put \[ E := \tau^c_{\ge 0} \tilde\pi^*(\MU_\ell)_\ell^\comp \in \CAlg(\SH(S)_\ell^\comp), \] where we view \[ \tilde\pi^*(\MU_\ell)_\ell^\comp \in \ShvSp(S_\et)_\ell^\comp \wequi \SH_\et(S)_\ell^\comp \subset \SH(S)_\ell^\comp. \]
We have \[ L_\et(\MGL)_\ell^\comp \wequi \tilde\pi^*(\MU_\ell)_\ell^\comp \in \Alg_{\scr E_0}(\SH_\et(S)_\ell^\comp), \] by Theorem \ref{thm:identify-MGL}.
It follows via Lemma \ref{lem:chow-recover}(1) and Example \ref{ex:MGL-cellular} that $E \wequi \MGL_\ell^\comp$ as $\scr E_0$-algebras.
Hence by Corollary \ref{cor:reconst-1} we find that \[ e_\Ss^*(\1_{\ell,\eta}^\comp) \wequi \Tot \tau^\Ss_{\ge 0} e_\Ss^* L_\et(\CB(E))_\ell^\comp \in \CAlg(\ShvSp^\Ss). \]
The map \[ E = \tau_{c \ge 0}\tilde\pi^*(\MU_\ell)_\ell^\comp \to \tilde\pi^*(\MU_\ell)_\ell^\comp \in \CAlg(\SH(S)_\ell^\comp) \] induces \[ L_\et(E)_\ell^\comp \to \tilde\pi^*(\MU_\ell)_\ell^\comp \in  \CAlg(\SH(S)_\ell^\comp). \]
This is an equivalence, since the underlying map of spectra is equivalent to $L_\et(\MGL)_\ell^\comp \to \tilde\pi^*(\MU_\ell)_\ell^\comp$.
It follows that $L_\et(\CB(E))_\ell^\comp \wequi \tilde\pi^*(\CB(\MU_\ell))_\ell^\comp$ as $\scr E_\infty$-rings, concluding the proof.
\end{proof}

\begin{exm}
The assumptions of Theorem \ref{thm:recons-1} regarding $\MGL$ are satisfied if $S$ is essentially smooth over a Dedekind scheme and $\pvcd_\ell(S) \le 1$.
This follows from Corollaries \ref{cor:MGL-chow-conn} and \ref{cor:CQL-MGL}.
\end{exm}

\subsection{Cohomological dimension $\le 2$} \label{subsec:coh-dim-2}
Recall from Proposition \ref{prop:cx-oriented-equiv} that there is a morphism of graded Hopf algebroids (whence graded cosimplicial rings) $\omega: \pi_{2*}\CB(\MU) \to \pi_{2*,*}\CB(\MU_\ell)$.
\begin{thm} \label{thm:recons-2}
Suppose that $1/\ell \in S$, $\dim S < \infty$, $\MGL$ satisfies $\CQL_1(\ell, S)$ and $\pi_{**}(\MGL_\ell^\comp)$ is concentrated in Chow degrees $\ge 0$.
Then there is a pullback square in $\CAlg(\ShvSp^\Ss)$
\begin{equation*}
\begin{CD}
e_\Ss^*(\1_{\ell,\eta}^\comp) @>>> \Tot \tau^\Ss_{\ge 0} \Gamma_\et^\Ss \tilde\pi^*(\CB(\MU_\ell))_\ell^\comp \\
@VVV                                  @VVV \\
\Tot \pi_{2*}\CB(\MU)_\ell^\comp @>{\omega}>> \Tot \tau^\Ss_{=0} \Gamma_\et^\Ss \tilde\pi^*(\CB(\MU_\ell))_\ell^\comp.
\end{CD}
\end{equation*}
\end{thm}
\begin{proof}
Define $E \in \CAlg(\SH(S))$ by the pullback square
\begin{equation} \label{eq:E-fake-MGL}
\begin{CD}
E @>>> \tau^c_{\ge 0} \tilde\pi^*(\MU_\ell)_\ell^\comp \\
@VVV                        @VVV \\
\tau_{=0}^c(\MGL_\ell^\comp) @>>> \tau^c_{= 0} \tilde\pi^*(\MU_\ell)_\ell^\comp. \\
\end{CD}
\end{equation}
Here the bottom map arises from the equivalence of homotopy rings $L_\et(\MGL)_\ell^\comp \wequi \tilde\pi^*(\MU_\ell)_\ell^\comp$ from Corollary \ref{cor:htpy-ring-MUl}.
It follows from Lemma \ref{lem:chow-recover}(2) that $E \wequi \MGL_\ell^\comp$ as $\scr E_0$-algebras (in fact homotopy rings).
Thus as before by \cite[Theorem 2.3]{bachmann-topmod} we have $\1_{\ell,\eta}^\comp \wequi \Tot \CB(E)_\ell^\comp$.
We have a commutative square of cosimplicial commutative rings
\begin{equation*}
\begin{CD}
\CB(E) @>>> \tau^c_{\ge 0} \tilde\pi^*(\CB(\MU_\ell^\comp)) \\
@VVV                 @VVV \\
\tau^c_{=0} \CB(\MGL)_\ell^\comp @>{\omega'}>> \tau^c_{=0} \tilde\pi^*(\CB(\MU_\ell^\comp)).
\end{CD}
\end{equation*}
It follows from Lemma \ref{lem:chow-recover}(2) that this square is cartesian.
The desired result follows by totalizing, applying $e_\Ss^*$, using Lemma \ref{lem:chow-recover}(3) (with $A=E$) to commute $e_\Ss^*$ and truncation, and observing that $\omega'$ corresponds to $\omega$ by Proposition \ref{prop:cx-oriented-equiv}(3).
\end{proof}

\begin{exm}
The assumptions of Theorem \ref{thm:recons-2} regarding $\MGL$ are satisfied if $S$ is essentially smooth over a Dedekind scheme and $\pvcd_\ell(S) \le 2$, as follows from Corollaries \ref{cor:MGL-chow-conn} and \ref{cor:CQL-MGL}.
This is the case for example if $S = \Spec(\Z[1/\ell])$.
\end{exm}

\begin{dfn} \label{def:Prf}
Let $\Prf$ denote the category of profinite groupoids.
For a scheme $S$, we denote by $\Pi_{\le 1}(S_\et) \in \Prf_{/G_\ell}$ the étale fundamental groupoid of $S$ with its cyclotomic character.
\end{dfn}

Given $\rho: \scr X \to BG_\ell \in \Prf_{/BG_\ell}$, the invertible object $\1_\ell(1) \in \ShvSp(B_\et G_\ell)_\ell^\comp$ from Lemma \ref{lem:Zl-twisting} yield an invertible object in $\ShvSp_\et(\scr X)_\ell^\comp$, and hence a functor \[ e_{\Ss,\scr X}: \ShvSp^{\Ss} \to \ShvSp_\et(\scr X)_\ell^\comp. \]
\begin{cnstr}
We define a functor \[ R_{(\ph)}: \Prf_{/BG_\ell}^\op \to \CAlg(\ShvSp^\Ss) \] via the pullback square
\begin{equation*}
\begin{CD}
R_{\scr X} @>>> \Tot \tau^\Ss_{\ge 0} e_{\Ss,\scr X}^* \rho^* \CB(\MU_\ell^\comp)_\ell^\comp \\
@VVV                         @VVV \\
\Tot \pi_{2*}\CB(\MU)_\ell^\comp @>{\omega_\scr{X}}>> \Tot \tau^\Ss_{=0} e_{\Ss,\scr X}^* \rho^* \CB(\MU_\ell^\comp)_\ell^\comp,
\end{CD}
\end{equation*}
where $\omega_\scr{X}$ is induced by $\omega$.
\end{cnstr}

We know that $\MGL_\ell^\comp \in \tau_{\ge 0}^\ell \SH(S)_\ell^\comp$ and $e_\Ss^* \tau_{=0}^c \MGL_\ell^\comp \wequi \pi_{2*}\MU_\ell^\comp$.
We thus have a commutative square
\begin{equation*}
\begin{CD}
\MGL_\ell^\comp @>>> \tau_{\ge 0}^c L_\et(\MGL)_\ell^\comp \wequi \tau_{\ge 0}^c \tilde\pi^* \MU_\ell^\comp \\
@VVV @VVV \\
\tau_{=0}^c\MGL_\ell^\comp @>>> \tau_{=0}^c L_\et(\MGL)_\ell^\comp \wequi \tau_{=0}^c \tilde\pi^* \MU_\ell^\comp
\end{CD}
\end{equation*}
Applying $e_\Ss^*$ yields a commutative square (assuming that $\pi_{**} \MGL_\ell^\comp$ is concentrated in Chow degrees $\ge 0$)
\begin{equation*}
\begin{CD}
e_\Ss^(\MGL_\ell^\comp) @>>> \tau_{\ge 0}^\Ss  \Gamma_\et^\Ss \tilde\pi^*(\MU_\ell))_\ell^\comp \\
@VVV @VVV \\
\pi_{2*}(\MU)_\ell^\comp @>>> \tau_{=0}^\Ss  \Gamma_\et^\Ss \tilde\pi^*(\MU_\ell))_\ell^\comp.
\end{CD}
\end{equation*}
The same can be done for tensor powers of $\MGL$, and hence for the cobar construction of $\MGL$.
Since the totalization of the cobar construction is $\1_{\ell,\eta}^\comp$, we finally obtain a commutative square
\begin{equation*}
\begin{CD}
e_\Ss^*(\1_{\eta,\ell}^\comp) @>>> \Tot \tau^\Ss_{\ge 0} \Gamma_\et^\Ss \tilde\pi^*\CB(\MU_\ell))_\ell^\comp \\
@VVV                         @VVV \\
\Tot \pi_{2*}\CB(\MU)_\ell^\comp @>>> \Tot \tau^\Ss_{=0} \Gamma_\et^\Ss \tilde\pi^*\CB(\MU_\ell))_\ell^\comp.
\end{CD}
\end{equation*}

We can now construct our Galois approximation of $\1_{\eta,\ell}^\comp$.
\begin{cnstr} \label{def:galois-approx}
Let $\scr C \subset \Sch_{\Z[1/\ell]}$ denote the subcategory of schemes such that $\pi_{**} \MGL_\ell^\comp$ is concentrated in Chow degrees $\ge 0$.
Let $S \in \scr C$ and view $S_\et \in \Prf_{/BG_\ell}$.
We have $\Gamma_\et^\Ss \tilde\pi^*(\MU_\ell))_\ell^\comp \wequi  e_{\Ss,S_\et}^* \rho^* \MU_\ell^\comp$ and similar for the cobar constructions.
Consequently the above commutative square induces (by considering the cartesian gap) a morphism \[ e_\Ss^*(\1_{\eta,\ell}^\comp) \to R_{S_\et}. \]
This defines a natural transformation of functors \[ e_\Ss^*(\1_{(\ph),\eta,\ell}^\comp) \to R_{(\ph)_\et} \in \Fun(\scr C^\op, \CAlg(\ShvSp^\Ss)). \]
\end{cnstr}

\begin{rmk} \label{rmk:recons-G}
Suppose that $S$ satisfies the assumptions of Theorem \ref{thm:recons-2}.
Then the Galois approximation map $e_\Ss^*(\1_{S,\eta,\ell}^\comp) \to R_{S_\et}$ is an equivalence.
In fact, by construction, it identifies with the equivalence from Theorem \ref{thm:recons-2}.
\end{rmk}


\begin{exm}
Suppose in the situation of Remark \ref{rmk:recons-G}, we are additionally given a map $S_\et \to \scr X \in \Prf_{/BG_\ell}$.
Then we obtain a map \[ R_{\scr X} \to R_{S_\et} \wequi e_\Ss^*(\1_{S,\eta,\ell}^\comp). \]
Here are two special cases:
\begin{enumerate}
\item We may take $S = \Spec(\Z[1/\ell])$ and $\scr X = BG_\ell$.
\item We may take $S = \Spec(\Z[1/\ell, \mu_{\ell^\infty}])$ and $\scr X = *$.
  Since $\scr X = (\Spec \CC)_\et$, we can identify the source and write the comparison map as \[ e_\Ss^*((\1_\CC)_{\ell, \eta}^\comp) \to e_\Ss^*((\1_{\Z[1/\ell,\mu_{\ell^\infty}]})_{\ell,\eta}^\comp). \]
\end{enumerate}
\end{exm}

\subsection{Adams summands}
\label{subsec:recons-adams}
Let $d > 0$.
Define a function $f_d: \Z \to \Z$ by $f_d(n) = 2n + \epsilon_d(n)$, where \[ \epsilon_d(n) = \begin{cases} 0 & n \equiv 0 \pmod d \\ 1 & n \equiv -1 \pmod d \\ \dots \\ d-1 & n \equiv 1-d \pmod d \end{cases}. \]
\begin{lem} \label{lem:sublin}
For $a, b \in \Z$ we have $f_d(a) + f_d(b) \ge f_d(a+b)$.
\end{lem}
\begin{proof}
Equivalently, we must show that $\epsilon_d(a) + \epsilon_d(b) \ge \epsilon_d(a+b)$.
Since $\epsilon_d(\ph)$ only depends on the congruence class modulo $d$, we may assume that $-d < a,b \le 0$.
In this case $\epsilon_d(a) = -a$ (and similarly for $b$) and hence either $\epsilon_d(a+b) = \epsilon_d(a) + \epsilon_d(b)$ for $\epsilon_d(a+b) = \epsilon_d(a)+\epsilon_d(b) - d$.
The result follows.
\end{proof}

Now fix an odd prime $\ell$ and set $d=\ell-1$.
Write $f: \Ss \to \Z$ for the composite $\Ss \to \pi_0(\Ss) \wequi \Z \xrightarrow{f_d} \Z$.
Applying the construction from Lemma \ref{lem:graded-t} we obtain a new $t$-structure $\tau_{\ge 0}^{\Ss,f}\ShvSp^{\Ss}$ on $\ShvSp^{\Ss}$, compatible with the symmetric monoidal structure.

Recall the Adams summand $e_0 \MGL_{(\ell)}$ from Definition \ref{def:adams-summands}.

\begin{lem} \label{lemm:fake-chow-conn}
Let $S=\Spec(k)$, where $k$ is a field of characteristic $\ne \ell$ and $\cd_\ell(k) \le \ell - 1$.\NB{Is this the expected bound?}
Let $E \in \SH(k)$ be one of the spectra $\Sigma^{2(\ell-1)n, (\ell-1)n} H\Z_{(\ell)}$ (for $n \in \Z$) or $(e_0 \MGL_{(\ell)})^{\otimes m}$ (for $m \ge 0$).
Then $e_\Ss^*E_\ell^\comp \wequi \tau_{\ge 0}^{\Ss,f}e_\Ss^* L_\et(E)_\ell^\comp$.
\end{lem}
\begin{proof}
The claim holds for $E$ if and only if $e_\Ss^*E_\ell^\comp \in \tau_{\ge 0}^{\Ss,f}\ShvSp^{\Ss}$ and \[ \fib(e_\Ss^*E_\ell^\comp \to e_\Ss^* L_\et(E)_\ell^\comp) \in \tau_{< 0}^{\Ss,f}\ShvSp^{\Ss}. \]
It follows that the class of such spectra is closed under extensions, $\Sigma^{2(\ell-1)n,(\ell-1)n}$, sums (use $\cd_\ell(k) < \infty$ and Lemma \ref{lemm:complete-sums}) and cofiltered limits with increasingly connected transition maps.
Using Proposition \ref{prop:e0MGL-basics}(4,5) we can assume $m=1$, and now Proposition \ref{prop:etale-slice} (convergence of the étale slice filtration) and Lemma \ref{lem:MGL-sc} (slice completeness of $\MGL$) reduce to $E=H\Z_{(\ell)}$.
In this case we know that $e_\Ss^*E_\ell^\comp \wequi \tau_{\ge 0}^{\Ss,g}e_\Ss^* L_\et(E)_\ell^\comp$, where $g(n)=n$.
Now observe the following facts:
\begin{enumerate}
\item $f(n) > 0$ for $n > 0$
\item $f(n) = g(n)$ for $-(\ell-2) \le n \le 0$
\item $f(n) \le -(\ell-2)$ for $n < -(\ell -2)$.
\end{enumerate}
Recall that $\pi_{i,j} e_\Ss^* L_\et(E)_\ell^\comp = H^{-i}_\et(k,\Z_\ell(-s))$.
(1) ensures that $\pi_{*,n} \tau_{\ge 0}^{\Ss,f}e_\Ss^* L_\et(E)_\ell^\comp = 0 (= \pi_{*,n} e_\Ss^* H\Z_\ell^\comp)$ for $n>0$.
(2) ensures that $\pi_{*,n} \tau_{\ge 0}^{\Ss,f}e_\Ss^* L_\et(E)_\ell^\comp \wequi \pi_{*,n} \tau_{\ge 0}^{\Ss,g}e_\Ss^* L_\et(E)_\ell^\comp$ for $-(\ell-2) \le n \le 0$.
Finally (3) ensures the same for $n<-(\ell-2)$ since $H^r_\et(k, \Z_\ell(s)) = 0$ for $r > \ell-2$, so both the $f$ and the $g$ truncation do not do anything in these weights.\NB{picture might be better proof...}
\end{proof}

\begin{thm}\label{thm:recons-3}
Let $\ell$ be odd, $S=\Spec(k)$, where $k$ is a field of characteristic $\ne \ell$ and $\cd_\ell(k) \le \ell - 1$.
Then \[ e_\Ss^*(\1_{\ell,\eta}^\comp) \wequi \Tot \tau^{\Ss,f}_{\ge 0} \Gamma_\et^\Ss \pi^*(\CB(e_0\MU_\ell))_\ell^\comp \in \CAlg(\ShvSp^\Ss). \]
\end{thm}
\begin{proof}
Consider the following pullback square in $\CAlg(\SH(\Z[1/\ell]))$, defining the upper left hand corner
\begin{equation*}
\begin{CD}
M @>>> \tau^c_{\ge 0} \pi^* e_0 \MU_\ell \\
@VVV                     @VVV \\
\tau^c_{= 0} e_0\MGL_\ell^\comp @>>> \tau^c_{= 0} \pi^* e_0 \MU_\ell.
\end{CD}
\end{equation*}
Here the bottom map arises from applying $e_0$ to the bottom map in \eqref{eq:E-fake-MGL}.
Since $L_\et e_0 \MGL_\ell^\comp \wequi \pi^* e_0 \MU_\ell$ by Proposition \ref{prop:e0-etale-reln}, it follows from Lemma \ref{lem:chow-recover}(2) that $M \wequi e_0 \MGL_\ell^\comp$, as $\scr E_0$-rings.
Moreover we have a morphism \[ M \to \pi^* e_0 \MU_\ell \in \CAlg(\SH(\Z[1/\ell])) \] which is in fact an $\ell$-adic étale equivalence.
Base changing to $k$, we obtain morphisms of $\scr E_\infty$-rings \[ \1_k \to M_k \to \pi^* e_0 \MU_\ell. \]
Since $M_k \wequi (e_0 \MGL_k)_\ell^\comp$ as $\scr E_0$-rings, we have $\1_{\ell,\eta}^\comp \wequi \Tot \CB(M_k)_\ell^\comp$.
Applying $e_\Ss^*$ we get \[ e_\Ss^*(\1_{\ell,\eta}^\comp) \wequi \Tot e_\Ss^* \CB(M_k)_\ell^\comp \in \CAlg(\ShvSp^\Ss). \]
On the other hand by Lemma \ref{lemm:fake-chow-conn} we have \[ e_\Ss^* \CB^m(M_k)_\ell^\comp \wequi \tau_{\ge 0}^{\Ss,f}e^* L_\et(\CB^m(M_k))_\ell^\comp. \]
The result follows since $L_\et \CB(M_k)_\ell^\comp \wequi \pi^* \CB(e_0 \MU_\ell)$.
\end{proof}

\begin{rmk}
We could use the construction from Theorem \ref{thm:recons-3} to build Galois approximations as in Definition \ref{def:galois-approx}.
This would have the advantage of being an equivalence on more fields.
In the sequel we do not need this additional generality, so we forgo spelling it out.
\end{rmk}

\subsection{Reconstructing cellular subcategories}
\label{subsec:cellular-reconstruction}
For $\scr X \in \Prf_{/BG_\ell}$, define \[ \GalSp(\scr X)_{\ell,\eta}^{\comp\cell} = \Mod_{R_\scr{X}}(\ShvSp^{\Ss})_{\ell,\eta}^\comp. \]
\begin{cor} \label{corr:cell-recons}
Suppose that $1/\ell \in S$, $\dim S < \infty$, $\MGL$ satisfies $\CQL_1(\ell, S)$ and $\pi_{**}(\MGL_\ell^\comp)$ is concentrated in Chow degrees $\ge 0$.
Then there is a canonical equivalence \[ \SH(S)_{\ell,\eta}^{\comp\cell} \wequi \GalSp(\Pi_{\le 1}S_\et)_{\ell,\eta}^{\comp\cell}. \]
\end{cor}
\begin{proof}
Immediate from Theorem \ref{thm:recons-2}.
\end{proof}

\begin{cnstr}
Given $\scr X \in \Prf_{/G_\ell}$, define an object $\MGL_{\scr X} \in \CAlg(\ShvSp^{\Ss})_{R_{\scr X}/}$ by the pullback square
\begin{equation*}
\begin{CD}
\MGL_{\scr X} @>>> \tau^\Ss_{\ge 0} e_{\Ss,\scr X}^* \rho^* \MU_\ell^\comp \\
@VVV                         @VVV \\
\pi_{2*} \MU @>{\omega_\scr{X}}>> \tau^\Ss_{=0} e_{\Ss,\scr X}^* \rho^* \MU_\ell^\comp.
\end{CD}
\end{equation*}
\end{cnstr}
\begin{cor} \label{corr:cell-recons-MGL}
Suppose that $1/\ell \in S$, $\dim S < \infty$, $\MGL$ satisfies $\CQL_1(\ell, S)$ and $\pi_{**}(\MGL_\ell^\comp)$ is concentrated in Chow degrees $\ge 0$.
Then under the equivalence \[ \SH(S)_{\ell,\eta}^{\comp\cell} \wequi \GalSp(\Pi_{\le 1} S_\et)_{\ell,\eta}^{\comp\cell}, \] the object $\MGL_{S_\et}$ corresponds to $\MGL$ (as $\scr E_0$-algebras and as homotopy rings).
\end{cor}
\begin{proof}
This is essentially Lemma \ref{lem:chow-recover}(2).
\end{proof}

\begin{exm} \label{exm:tate-orientable-universal}
\begin{enumerate}
\item Let $k_0$ be obtained from either $\F_p$ or $\QQ$ by adjoining all $\ell$-power roots of unity.
Let $k/k_0$ be any extension, that is, $k$ is an arbitrary Tate-orientable field.
Write $e \in \Prf_{/BG_\ell}$ for the groupoid $*$ with the trivial map to $BG_\ell$.
By construction there is a morphism $\pi: Spec(k_0)_\et \to e \in \Prf_{/BG_\ell}$.
We thus obtain a diagram \[ e: \GalSp(e)_{\ell,\eta}^{\comp\cell} \xrightarrow{\pi^*} \GalSp(\Pi_{\le 1} Spec(k_0)_\et)_{\ell,\eta}^{\comp\cell} \wequi \SH(k_0)_{\ell,\eta}^{\comp\cell} \to \SH(k)_{\ell,\eta}^{\comp\cell}. \]

\item The functoriality of $\MGL_{(\ph)}$ induces a map $e(\MGL_e) \to \MGL_k$.

\item Let $\bar k/k_0$ be a separably closed extension.
  Then we find that the composite \[ \GalSp(e)_{\ell,\eta}^{\comp\cell} \to \SH(\bar k)_{\ell,\eta}^{\comp\cell} \] is the equivalence from Corollary \ref{corr:cell-recons}.
  It sends $\MGL_e$ to $\MGL_k$, by Corollary \ref{corr:cell-recons-MGL}.
\end{enumerate}
\end{exm}

\begin{exm}
Let $\scr X \in \Prf_{/BG_\ell}$ denote the terminal object.
Using somewhat suggestive notation, we write \[ \GalSp(\F_1)_{\ell,\eta}^{\comp\cell} := \GalSp(\scr X)_{\ell,\eta}^{\comp\cell}. \]
We thus obtain a composite \[ \GalSp(\F_1)_{\ell,\eta}^{\comp\cell} \to \GalSp(\Pi_{\le 1} \Spec(\Z[1/\ell])_\et)_{\ell,\eta}^{\comp\cell} \wequi \SH(\Z[1/{\ell,\eta}])_{\ell,\eta}^{\comp\cell} \to \SH(S)_{\ell,\eta}^{\comp\cell}, \] for any scheme $S$ on which $\ell$ is invertible.
\end{exm}

\subsection{Reconstructing Artin--Tate subcategories}
\label{subsec:AT-reconstruction}
Let $\scr X \in \Prf_{/BG_\ell}$.
Write \[ \ShvSp_\Nis(\scr X) = \PSh_\Sigma(\Fin_{\scr X})[(S^1)^{-1}]. \]
We give $\ShvSp_\Nis(\scr X)^\Ss$ a $t$-structure $\tau_{\ge 0}^{f}$ by applying the $t$-structure from \S\ref{subsec:recons-adams} sectionwise.
(If $\ell=2$, then $\tau_{\ge 0}^{f} := \tau_{\ge 0}^c$.)
The invertible object $\pi^* \1_\ell(1) \in \ShvSp_\et(\scr X)_\ell^\comp$ from Lemma \ref{lem:Zl-twisting} together with the evident symmetric monoidal functor $\ShvSp_\Nis(\scr X) \to \ShvSp_\et(\scr X)_\ell^\comp$ induces an adjunction \[ e_{\Ss,AT,\scr X} \ShvSp_\Nis(\scr X)^\Ss \to \ShvSp_\et(\scr X)_\ell^\comp: e_{\Ss,AT,\scr X}^*. \]
\begin{cnstr}
We define an object \[ R_{\scr X}^{AT} \in  \CAlg(\ShvSp_\Nis(\scr X)^\Ss) \] via the pullback square
\begin{equation*}
\begin{CD}
R_{\scr X}^{AT} @>>> \Tot \tau^{\Ss,f}_{\ge 0} e_{\Ss,AT,\scr X}^* \rho^* \CB(e_0\MU_\ell^\comp)_\ell^\comp \\
@VVV                         @VVV \\
\Tot \pi_{2*}\CB(e_0\MU) @>{\omega_\scr{X}}>> \Tot \tau^{\Ss,f}_{=0} e_{\Ss,AT,\scr X}^* \rho^* \CB(e_0\MU_\ell^\comp)_\ell^\comp,
\end{CD}
\end{equation*}
where $\omega_\scr{X}$ is induced by $\omega$.

Removing the occurrences ``$\CB$'' from the above square, we can define another object $e_0 M_{\scr X}^{AT} \in \CAlg(\ShvSp_\Nis(\scr X)^\Ss)$, and a morphism $R_{\scr X}^{AT} \to e_0M_{\scr X}^{AT}$.
Similarly we can define $M_{\scr X}^{AT}$ by further leaving out the ``$e_0$'' part.
\end{cnstr}

\begin{dfn}
We set \[ \GalSp(\scr X)_{\eta,\ell}^{\comp AT} = \Mod_{R_{\scr X}^{AT}}(\ShvSp_\Nis(\scr X)^\Ss)_{\ell,\eta}^\comp. \]
Note that $M_{\scr X}^{AT}$ defines an object of $\CAlg(\GalSp(\scr X)_{\eta,\ell}^{\comp AT})$.
\end{dfn}

\begin{rmk} \label{rmk:decompleted}
It will be occasionally useful to consider decompleted versions of the above categories, i.e. \[ \GalSp(\scr X)^{AT} := \Mod_{R_{\scr X}^{AT}}(\ShvSp_\Nis(\scr X)^\Ss). \]
\end{rmk}

Now let $k$ be a field of characteristic $\ne \ell$.
\begin{prop} \label{prop:artin-tate-recons}
\begin{enumerate}
\item If $\cd_\ell(k) \le \ell -1$ then there is a fully faithful functor $e: \GalSp(\Pi_{\le 1} k_\et)_{\eta,\ell}^{\comp AT} \to \SH(k)_{\eta,\ell}^\comp$.
  Its essential image is the subcategory generated by Artin--Tate objects.
  It identifies $e_0 M_{\Pi_{\le 1} k_\et}$ and $e_0 \MGL_\ell^\comp$.
\item If $\cd_\ell(k) \le 1$ then $e$ identifies $M_{\Pi_{\le 1} k_\et}$ and $\MGL_\ell^\comp$.
\item More generally, the above fully faithful map factors through an equivalence \[ e: \GalSp(\Pi_{\le 1} k_\et)^{AT} \xrightarrow{\wequi} \Mod_{\1_{\ell,\eta}^\comp}(\SH(k)^{AT}). \]
\end{enumerate}
\end{prop}
\begin{proof}
(1, 2) Minor modifications of the previous arguments yield the desired comparison maps.
They recover the previous maps sectionwise, and thus are equivalences by previous results.

(3) Apply $\Ind((\ph)^{\omega})$ to the equivalence of (1).
\end{proof}

\begin{dfn} \label{def:AT-F1}
If $G \subset G_\ell$ is a subgroup, we obtain $BG \in \Prf_{/G_\ell}$.
Depending on $G$, we use special evocative names for the resulting category $\GalSp(BG)_{\eta,\ell}^{\comp AT}$:
\begin{itemize}
\item For $\ell$ odd, $G=\Z_\ell^\times$ we write $\GalSp(\F_1)_{\eta,\ell}^{\comp AT}$.
\item For $\ell$ odd, $G = \ell^{n-1} \Z_\ell \subset \Z_\ell^\times$ we write $\GalSp(\F_1[\mu_{\ell^n}])_{\eta,\ell}^{\comp AT}$.
\item For $\ell=2$, we use similar definitions in terms of the subgroup $\Z_2 \subset G_2$.
\end{itemize}
\end{dfn}

\begin{exm} \label{exm:AT-e}
Let $n \ge 1$ ($n \ge 2$ if $\ell=2$).
Let $S \in \Sch_{/\Z[1/\ell,\mu_{\ell^n}}$.
Since the cyclotomic character of $\Spec(\Z[1/\ell,\mu_{\ell^n}])$ factors through the appropriate subgroup from Definition \ref{def:AT-F1}, we can form the composite \[ \GalSp(\F_1[\mu_{\ell^n}])_{\ell,\eta}^{\comp AT} \to \GalSp(\Pi_{\le 1} \Spec(\Z[1/\ell,\mu_{\ell^n}]))_{\ell,\eta}^{\comp AT} \wequi \SH(\Z[1/\ell, \mu_{\ell^n}])_{\ell,\eta}^{\comp AT} \to \SH(k)_{\ell,\eta}^{\comp AT}. \]
Similarly for the decompleted version.
\end{exm}

\begin{rmk} \label{rmk:identify-SHF1}
Fix $\ell$.
Write $\hat\F_p$ for the maximal prime-to-$\ell$ extension of $\F_p$, so that $Gal(\hat\F_p) = \Z_\ell$.
Let $n \ge 1$ ($n \ge 2$ if $\ell = 2$) and assume that $\mu_{\ell^{n+1}} \not\subset \hat\F_p[\mu_{\ell^n}]$.
(For fixed $\ell, n$, such $p$ can always be found.\NB{specifics?})
It follows that the cyclotomic character of $Gal(\hat\F_p[\mu_{\ell^n}])$ acts non-trivially on $\mu_{\ell^{n+1}}$ and so must surject onto the subgroup $\ell^{n-1} \Z_\ell$ ($\ell^{n-2} \Z_\ell$ if $\ell=2$) of Definition \ref{def:AT-F1}.
Consequently the object $BG$ is isomorphic to $\Pi_{\le 1} \Spec(\F_p[\mu_{\ell^n}])_\et$ (any continuous surjection $\Z_\ell \to \Z_\ell$ being an isomorphism).
Thus by Proposition \ref{prop:artin-tate-recons} we learn that \[ \GalSp(\F_1[\mu_{\ell^n}])_{\eta,\ell}^{\comp AT} \wequi \SH(\hat\F_p[\mu_{\ell^n}])_{\eta,\ell}^{\comp AT}, \] and similarly for the decompleted version.
\end{rmk}

\section{Motivic stable stems} \label{sec:motivic-stable-stems}
\localtableofcontents

\bigskip

In this section we put it all together, and determine the $(p,\eta)$-completed motivic stable stems of many fields $k$.
Instead of directly studying the motivic Adams--Novikov spectral sequence as sketched in \S\ref{subsec:intro-sketch}, we proceed slightly indirectly.
Let us begin with the case where $\mu_{p^\infty} \subset k$.
Under this assumption, our Galois approximations from the previous section furnish a functor $e: \SH(\CC)_{p,\eta}^{\comp\cell} \to \SH(k)_{p,\eta}^\comp$ with a right adjoint $e^*$.
We shall first show that this functor sends $H\Z/p$ to $H\Z/p$.
From this we deduce, by a compact-rigid generation argument, that $e_*(E) \otimes H\Z/p \wequi e_*(E \otimes H\Z/p)$.
We want to understand $e_*(\1_{p,\eta}^\comp)$; in fact we shall show that this splits into a sum of $\Gm$-twists of the unit of $\SH(\CC)_{p,\eta}^\comp$.
The main idea is this: given a map \[ \alpha: \bigoplus_{s \in S} \Sigma^{n_s,n_s} \1_{p,\eta}^\comp \to e_*(\1_{p,\eta}^\comp), \] in order to prove that $\alpha$ is an equivalence, we need only check this after $\otimes H\Z/(p,\tau)$.
Now $\pi_{**} H\Z/(p,\tau)(\CC) \wequi \F_p$, concentrated in bidegree $(0,0)$, and more generally $\pi_{**}e_*(H\Z/(p,\tau)) \wequi K_*^M(k)/p$, concentrated in bidegrees of the form $(n,n)$ (and note that, by what we said previously, $e_*(\1_{p,\eta}^\comp) \otimes H\Z/(p,\tau) \wequi e_*(\1_{p,\eta}^\comp \otimes H\Z/(p,\tau)) \wequi e_*(H\Z/(p,\tau))$.
In particular, the map $\alpha$ will be an equivalence if and only if it corresponds to a lift of an $\F_p$-basis of $K_*^M(k)/p$, which exists (note that $\pi_0(\1)_*(k) = K_*^{MW}(k)$, which surjects onto $K_*^M(k)/p$).
Once this direct sum decomposition is established, the tensor product formula of Theorem \ref{thm:intro-intro} is an easy consequence (see Theorem \ref{thm:tensor-product-formula-easy}).

The situation gets much more delicate without the assumption that $\mu_{p^\infty} \subset k$.
We replace the functor $e$ from the above discussion by a new functor $e: \SH(\F_1[\mu_{p^n}])_{p,\eta}^{\comp AT} \to \SH(k)_{p,\eta}^\comp$.\footnote{In fact, for technical reasons we will not complete the categories and instead work with modules over the completed sphere, but we ignore this for now.}
It will still be the case that $e(H\Z/p) \wequi H\Z/p$.
An evident variant of the homotopy $t$-structure exists on $\SH(\F_1[\mu_{p^n}])_{p,\eta}^{\comp AT}$.
The heart consists, roughly speaking, or graded Mackey functors for $B_\et \Z_p$.
It is still the case that $e_*(H\Z/(p,\tau))$ lies in this heart, but it so no longer just a sum of twists of the unit.
In fact it is not even a sum of representables.
Instead we shall prove that it is \emph{flat}.
This implies a (slightly weaker) tensor product formula, by a more complicated argument (see Theorem \ref{thm:motivic-stems-hard}).
The proof of flatness itself involves a careful study of the $\Z_p$-action on $K_*^M(k(\mu_{p^\infty}))/p$.

\subsubsection*{Organization}
We begin in \S\ref{subsec:tate-orientable} by implementing the strategy outlined above for fields containing $\mu_{p^\infty}$, which we call \emph{Tate-orientable}.
The remaining subsections deal with the case where $k$ is not Tate-orientable.
In \S\ref{subsec:non-tate} we introduce our setup and reduce the tensor product formula to the flatness property for Milnor $K$-theory.
This flatness is established in \S\ref{subsec:flatness}, except at the prime $2$, where we need an additional lemma.
Its proof is the content of \S\ref{subsec:H90}.

\subsubsection*{Standing assumptions}
We fix a prime $p$.
All fields will be of characteristic $\ne p$.

\subsection{Tate-orientable case} \label{subsec:tate-orientable}
We call the field $k$ \emph{Tate-orientable} if it contains $\mu_{p^n}$ for all $n$.
In this section we assume $k$ Tate-orientable.
Example \ref{exm:tate-orientable-universal} supplies us (because of Tate-orientability) with a functor \[ e: \SH(\CC)_{p,\eta}^{\comp\cell} \to \SH(k)_{p,\eta}^\comp. \]
The following summarizes the facts about this functor and the source category we will use.
\begin{prop} \label{prop:tate-orient-summary}
\begin{enumerate}
\item We have equivalences of ring spectra $e(\MGL_p^\comp) \wequi \MGL_p^\comp$ and $e(H\Z_p^\comp) \wequi H\Z_p^\comp$.
\item The category $\SH(\CC)_{p,\eta}^{\comp\cell}$ has a $t$-structure with non-negative part generated under colimits and extensions by $\Sigma^{i,i} \1_{p,\eta}^\comp$.
  The heart is equivalent to the category of derived $(p,\eta)$-complete modules over the commutative graded ring $K^{MW}(\CC)_{p,\eta}^\comp$.
  The functor sends an object $E$ in the heart to the module $[\Sigma^{-*,-*} \1_{p,\eta}^\comp, E]$.
\item Formation of homotopy objects commutes with arbitrary sums.\NB{but not filtered colimits}
\end{enumerate}
\end{prop}
\begin{proof}
(1) We have a canonical map $e(\MGL_{p,\eta}^\comp) \to \MGL_{p,\eta}^\comp$, which becomes an equivalence after base change to a separably closed extension of $k$ (see again Example \ref{exm:tate-orientable-universal}).
Both versions of $\MGL$ are built out of $(2n,n)$-cells for $n \ge 0$ (see e.g. \cite[Theorem A.1]{BKWX}).
They are in particular bounded below, and since we are working in an $(\eta,p)$-complete setting, to check $e(\MGL_{p,\eta}^\comp) \to \MGL_{p,\eta}^\comp$ is an equivalence it suffices to check the same after $\otimes H\Z/p$.
We are now left with a map of pure Tate $H\Z/p$-modules.
Since the category of these modules is independent of the field (see Lemma \ref{lemm:conservative-realization}), we are reduced to the case where $k$ is separably closed, which was already dealt with in Example \ref{exm:tate-orientable-universal}.

Note that $H\Z_{p,\eta}^\comp$ can be obtained from $\MGL_{p,\eta}^\comp$ by modding out generators \cite{hoyois2015algebraic}.
Since the equivalence $e(\MGL_{p,\eta}^\comp) \wequi \MGL_{p,\eta}^\comp$ respects the homotopy ring structures, we deduce that it also sends $H\Z_{p,\eta}^\comp$ to $H\Z_{p,\eta}^\comp$.

(2) Let $\scr M$ denote the subcategory of modules over $\1_{p,\eta}^\comp \in \SH(\CC)$ generated under colimits by $\Sigma^{i,j} \1_{p,\eta}^\comp$, so that $\scr M_{p,\eta}^\comp \wequi \SH(\CC)_{p,\eta}^{\comp\cell}$.
We give $\scr M$ the $t$-structure with non-negative part generated under colimits and extensions by $\Sigma^{i,i} \1_{p,\eta}^\comp$ for $i \in \Z$.
For $E \in \scr M$, write $\Pi_i(E)$ for the sequence of abelian groups $\pi_i(E)_*$.
Note that $\Pi_i(\1_{p,\eta}^\comp) = 0$ for $i<0$ and $\Pi_0(\1_{p,\eta}^\comp) \wequi K^{MW}(\CC)_{p,\eta}^\comp$.
Thus each $\Pi_i(E)$ is a graded module over $K^{MW}(\CC)_{p,\eta}^\comp$, and from now on we view $\Pi_i$ as taking values in this category.
Given $E \in \scr M$, we can inductively cone off copies of $\Sigma^{a,b} \1_{p,\eta}^\comp$ to build a cofiber sequence $C \to E \to E'$ with $\Pi_i(E') = 0$ for $i > 0$, $\Pi_i(E') \wequi \Pi_i(E)$ for $i \le 0$ and $C \in \scr M_{> 0}$.
The vanishing condition on $\Pi_i(E')$ implies that $E' \in \scr M_{\le 0}$.
We have thus shown that the $t$-structure on $\scr M$ is determined by the homotopy ``groups'' $\Pi_*(\ph)$.
We have the functor $\Pi_0: \scr M^\heart \to \Mod_{K^{MW}(\CC)_{p,\eta}^\comp}$.
It preserves colimits and its essential image contains the free modules, so the functor essentially surjective.
The functor has a right adjoint (for formal reasons), which we denote by $R$.
We get \[ [\Sigma^{i,i} \1_{p,\eta}^\comp,  R \Pi_0 E]_{\scr M} \wequi [(K^{MW}(\CC)_{p,\eta}^\comp)_{*-i}, \Pi_0 E]_{\Mod_{K^{MW}(\CC)_{p,\eta}^\comp}} \wequi \Pi_0(E)_{-i}, \] from which we deduce that $R \Pi_0 \wequi \id$.
It follows that $\scr M^\heart \wequi \Mod_{K^{MW}(\CC)_{p,\eta}^\comp}$.
We shall show that $M \in \scr M$ is $a$-complete (for some $a \in K_{|a|}^{MW}(\CC)_{p,\eta}^\comp$) if and only if each $\Pi_i(M)$ is derived $a$-complete.
This implies that the $t$-structure restricts to $a$-complete objects, and hence (by intersecting) also to $(a,b,...,c)$-complete objects, e.g. $(p,\eta)$-complete objects.

To establish (2), it thus remains to prove the claim.
We know that $M$ is $a$-complete if and only if\footnote{Observe that $L$ is the fiber of $M \to M_a^\comp$.} \[ L := \lim(M \xleftarrow{a} \Sigma^{|a|,|a|} M \xleftarrow{a} \dots) = 0. \]
The Milnor exact sequence supplies us with a short exact sequence \[ 0 \to L_{1,i+1} := \limone(\Pi_{i+1}(M) \xleftarrow{a} \dots) \to \Pi_i(L) \to L_{0,i} := \lim(\Pi_i(M) \xleftarrow{a} \dots) \to 0. \]
Altogether we learn that $M$ is derived $a$-complete if and only if $L_{1,i} = 0 = L_{0,i}$ for every $i$.
These are exactly the conditions defining derived $a$-completeness of $\Pi_i(M)$.

(3) We must prove that $(\SH(\CC)_{p,\eta}^{\comp\cell})_{\ge 0}$ and $(\SH(\CC)_{p,\eta}^{\comp\cell})_{\le 0}$ are closed under sums.
This is clear for $\ge 0$.
Thus let $\{E_\alpha\}_\alpha \in (\SH(\CC)_{p,\eta}^{\comp\cell})_{\le 0}$ and consider the fiber sequence \[ F \to \bigoplus_\alpha E_\alpha \to \prod_\alpha E_\alpha \in \scr M. \]
Note that $\prod_\alpha E_\alpha \in \scr M_{\le 0}$ and $F \in \scr M_{<0}$ (consider the homotopy groups).
We deduce that $F_p^\comp \in \scr M_{\le 0}$ (the usual formula for $p$-completion shows that $(\scr M_{<0})_p^\comp \subset \scr M_{\le 0}$), whence $(\bigoplus_\alpha E_\alpha)_p^\comp \in \scr M_{\le 0}$ (the product being already complete).
A similar argument applies to completion at $\eta$ instead.
Unfortunately the case of completion at both elements does not follow formally.
If $p$ is odd then $\eta=0$ (see Remark \ref{rmk:KMW-C-comp} below) and we are done.
Now suppose $p=2$.
It will suffice to consider the case where $E_\alpha \in \scr M^\heart$.
For each $\alpha$ we have a short exact sequence \[ 0 \to 2E_\alpha \to E_\alpha \to E_\alpha/2 \to 0 \in \scr M^\heart. \]
Since $2\eta = 0$, $2E_\alpha$ is $\eta$-torsion; also $E_\alpha/2$ is $2$-torsion.
It follows that \[ \left(\bigoplus_\alpha 2E_\alpha\right)_{2,\eta}^\comp \wequi \left(\bigoplus_\alpha 2E_\alpha\right)_{2}^\comp, \] and similarly for the sum of the $E_\alpha/2$.
All in all we obtain a cofiber sequence \[ \left(\bigoplus_\alpha 2E_\alpha\right)_2^\comp \to \left(\bigoplus_\alpha E_\alpha\right)_{2,\eta}^\comp \to \left(\bigoplus_\alpha E_\alpha/2\right)_\eta^\comp. \]
By assumption each $E_\alpha$ is $(2,\eta)$-complete, and it follows from (2) that images and quotients of derived $(2,\eta)$-complete objects in $\scr M^\heart$ are derived complete.
Thus $2E_\alpha$ is $2$-complete and $E_\alpha/2$ is $\eta$-complete, and hence the outer terms in the cofiber sequence are $\le 0$ by the case of completion at one element which we have already dealt with.
\end{proof}

\begin{rmk} \label{rmk:KMW-C-comp} \NB{justification?}
Note that if $p=2$ then \[ K_*^{MW}(\CC)_{2,\eta}^\comp \wequi \Z_2[\eta]/2\eta, \] whereas if $p$ is odd then \[ K_*^{MW}(\CC)_{p,\eta}^\comp \wequi \Z_p. \]
\end{rmk}

With this preparation out of the way, we can formulate the main result.
\begin{thm} \label{thm:tensor-product-formula-easy}
Let $k$ be a Tate-orientable field of characteristic $\ne p$.
There is a canonical isomorphism (in the category of derived complete modules over $K_*^{MW}(\CC)_{p,\eta}^\comp$) \[ \pi_{**}(\1_{p,\eta}^\comp)(k) \wequi K_*^{MW}(k)_{p,\eta}^\comp \hat\otimes \pi_{**}(\1_{p,\eta}^\comp)(\CC). \]
Moreover:
\begin{enumerate}
\item $K_*^{MW}(k)_{p,\eta}^\comp$ is a (graded) free module.
\item The derived completion $K_*^{MW}(k)_{p,\eta}^\comp$ coincides with the naive completion of $K_*^{MW}(k)$.
  The same holds for $K_*^M(k)$.
\end{enumerate}
\end{thm}
\begin{proof}
For brevity we put $\1_k = (\1_k)_{p,\eta}^\comp$.
Lemma \ref{lemm:projection-formula} shows that that $e_*(\1_k) \otimes H\Z/(p,\tau) \wequi e_*(e^* H\Z/(p,\tau))$, which by Proposition \ref{prop:tate-orient-summary} is the same as $e_*(H\Z/(p,\tau))$.
Note that $e_*(\1_k), \1_{\CC}, H\Z/(p,\tau)$ are all connective.
Pick a set $I$ and for each $i \in I$ an integer $n_i$ as well as an element $a_i \in \pi_0(e_*(H\Z/(p,\tau)))_{-n_i}$.
We can arrange $I, n_\bullet, a_\bullet$ in such a way that these elements form a basis of the graded $\FF_p$-module $\pi_0(e_*(H\Z/(p,\tau)))$.
Since $e_* \1_k$ is connective, \[ \pi_0(e_* \1_k) \to \pi_0(e_* \1_k/(p,\eta)) \wequi \pi_0(e_* H\Z/(p,\tau)) \] is surjective and hence we can pick elements $\tilde a_i \in \pi_0(e_*(\1_k))_{-n_i}$ lifting the $a_i$.
Now consider the map \[ \bigoplus_I \Sigma^{n_i,n_i} \1_{\CC} \to e_*(\1_k) \] corresponding to the $\tilde a_i$.
By construction, this map becomes an equivalence after $\otimes H\Z/(p,\tau)$.\footnote{Recall that $\pi_{**} H\Z/p(\CC) \wequi \FF_p[\tau]$, so that $\pi_{**}(H\Z/(p,\tau))(\CC)$ is a single copy of $\FF_p$ in bidegree $(0,0)$.}
Since $\pi_0(H\Z/(p,\tau)) \wequi \pi_0(\1_{\CC})/(p,\eta)$, it follows that the map becomes an equivalence after modding out by $p$ and $\eta$, whence, is an equivalence.
This in particular proves (1).
The formula for $\pi_{**} e_* \1_k$ follows via Proposition \ref{prop:tate-orient-summary}(3).

It remains to establish (2).
We first treat the case of $K_*^M(k)$.
We can obtain the derived $p$-completion as $\pi_0(\lim_n H\Z/p^n)_*(k)$.
But $\pi_{**}H\Z/p^n(k) \wequi K_*^M(k)/p^n[\tau]$ and so we see that the inverse system on $\pi_1$ (and in fact every $\pi_i$) is Mittag--Leffler, so does not contribute to $\pi_0$.
Now we treat $K_*^{MW}(k)$.
If $p$ is odd then $K_*^{MW}(k)_{p,\eta}^\comp \wequi K_*^{M}(k)_p^\comp$, in both the derived and the naive sense, so we conclude.\NB{details?}
Now we treat the case $p=2$.
We first show that the naive and derived completion at $2$ agree.
In other words we want to show that $\lim^1_n K_*^{MW}(k)[2^n] = 0$.
From $K_*^{MW}(k) \wequi K_*^M(k) \times_{k_*^M(k)} I^*(k)$\NB{ref?} we obtain an exact sequence \[ 0 \to K_*^{MW}(k)[2^n] \to K_*^M(k)[2^n] \oplus I^*(k) \to k_*^M(k) \to 0 \] (note that $I^*(k)$ is $2$-torsion, $-1$ being a square in $k$\NB{ref?}) and hence $\lim^1_n K_*^{MW}(k)[2^n] \wequi \lim^1_n K_*^{M}(k)[2^n] = 0$ (by what we already did).
Write $L_*$ for the derived $2$-completion of $K_*^{MW}(k)$.
It remains to show that the derived and ordinary $\eta$-completion of $L_*$ agree.
But multiplication by $\eta$ is an isomorphism on $L_*$ for $*<0$ (since this is already the case on $K_*^{MW}$\NB{ref?}).
This implies that for every $i$, the inverse system $ker(\eta^n: L_i \to L_{i-n})_n$ is pro-null, and so has vanishing $\lim^1$, as needed.
\end{proof}

\subsection{Non-orientable case: overview} \label{subsec:non-tate}
From now on we no longer assume that $k$ is Tate orientable, in fact, we assume that $\mu_{p^\infty} \not\subset k$.
We however do assume that $\mu_p \subset k$, and $\mu_4 \subset k$ if $p=2$.
\begin{rmk}
Note that if $\mu_p \not\subset k$, then if $k' = k[\mu_p]$ then $k'/k$ has degree dividing $p-1$, and so coprime to $p$.
If $p$ is odd, from this one can deduce that \[ \pi_{**}(\1_{p,\eta}^\comp)(k) \to \pi_{**}(\1_{p,\eta}^\comp)(k') \] is a split injection, onto the summand fixed by the action of $Gal(k'/k)$.
Our assumption that $\mu_p \subset k$ is thus mostly harmless in this case.
\end{rmk}

\begin{rmk} \label{rmk:Zp-action-Fpt-module}
Let $k^t = k(\mu_{p^\infty})$.
Then $Gal(k^t/k) \wequi \Z_p$ and we shall need to work with $D(Gal(k^t/k), \F_p)$.
A continuous action of $\Z_p$ on an $\F_p$-vector space $V$ is the same as an endomorphism $s$ such that for every $x \in V$ we have $s^{p^r}(x) = x$ for $r$ sufficiently large.
Setting $t=s-1$ we thus get $t^{p^r}(x) = 0$; hence the action is the same as a continuous $\F_p\fpsr{t}$-module structure.
All in all this yields an equivalence \[ D(Gal(k^t/k), \F_p) \wequi D(\F_p\fpsr{t})_{t-tors}. \]
On the left hand side we have the permutation representations coming from the quotient $\Z/p^r$ of $\Z_p$ (corresponding to adjoining appropriate roots of unity to $k$).
The corresponding $\F_p\fpsr{t}$-module is $\F_p\fpsr{t}/t^{p^r}$.
\end{rmk}

Let $n \ge 1$.
We have the presentably symmetric monoidal $\infty$-category $\GalSp(\FF_1[\mu_{p^n}])^{\comp AT}_{\eta,p}$ from Definition \ref{def:AT-F1}.
We shall also need to use the decompleted version $\GalSp(\FF_1[\mu_{p^n}])^{AT}$ from Remark \ref{rmk:decompleted}, mainly because it affords a more well-behaved $t$-structure.
Let us summarize its salient properties.
\begin{prop}\label{prop:tate-nonorient-summary}
\begin{enumerate}
\item $\GalSp(\FF_1[\mu_{p^n}])^{AT}$ is compact-rigidly generated by objects $\Gmp{i} \otimes \1_{1^r}$, $r \ge 1$.
  The unit is $\1 = \1_{1^n}$.
\item We have $[\Sigma^i \1_{1^r}, \1_{1^s} \otimes \Gmp{j}] = 0$ for $i < 0$.
  The $t$-structure with non-negative part generated by $\1_{1^r} \otimes \Gmp{i}$ (for $r,i$ arbitrary) is detected on appropriate homotopy presheaves.
  The $t$-structure is compatible with the symmetric monoidal structure and filtered colimits.
\item There are objects $\MGL, H\Z \in \GalSp(\FF_1[\mu_{p^n}])^{AT}_{\ge 0}$ and elements $\eta \in \pi_0(\1)_{-1}$, $\tau \in \pi_1(H\Z/p)_{-1}$.
  We have $\pi_0(\MGL) \wequi \pi_0(H\Z) \wequi \pi_0(\1)/\eta$ and $\pi_0(H\Z/p) \wequi H\Z/(p,\tau)$.
\item Let $\scr C \subset hD(\FF_p\fpsr{t})_{t-tors}$ denote the full subcategory on finite sums of objects $\Sigma^i \FF_p\fpsr{t}/t^{p^r}$ for $r \ge 0$, $i \in \Z$.
  Then there is an equivalence \[ \Mod(H\Z/(p,\tau))^\heart \wequi \PSh_\Sigma(\scr C)_{\le 0} \] under which $\1_{1^r} \otimes \Gmp{i}$ corresponds to $\Sigma^i \FF_p\fpsr{t}/t^{p^{r-n}}$.
\end{enumerate}
\end{prop}
\begin{proof}
Recall from Remark \ref{rmk:identify-SHF1} that $\GalSp(\FF_1[\mu_{p^n}])^{AT} \wequi \Mod_{\1_{p,\eta}^\comp}(\SH(\F)^{AT})$, for a well-chosen field $\F$, algebraic over $\F_\ell$ for some appropriate $\ell$.
(1), (2) and (3) are clear from this (taking $\MGL_\ell^\comp \in \Mod_{\1_{p,\eta}^\comp}(\SH(\F)^{AT})$ for $\MGL$, and similarly for $H\Z$).

(4) We know that $\Mod(H\Z/(p,\tau))$ is compactly generated by the $\1_{1^r} \otimes \Gmp{i} \otimes H\Z/(p,\tau)$.
Write $\scr D \subset \Mod(H\Z/(p,\tau))$ for the closure under finite sums of this category, and $\PSh_\Sigma(\scr D, \ShvSp)$ for the associated category of spectral presheaves.
Noting that $\scr D$, viewed as a spectrally enriched category, is in fact discrete (i.e. for $X, Y \in \scr D$ and $i \ne 0$ we have $[X,\Sigma^i Y] = 0$) we see that the canonical functor $c: \PSh_\Sigma(\scr D, \ShvSp) \to \Mod(H\Z/(p,\tau)$ is an equivalence.
Since $c$ preserves the generators of the $t$-structure, it is right-$t$-exact.
Since the $t$-structures can be checked on homotopy presheaves, the right adjoint of $t$ is also right-$t$-exact.
It follows that $c$ is a $t$-exact equivalence and so in particular $\Mod(H\Z/(p,\tau))^\heart \wequi \PSh_\Sigma(\scr D)_{\le 0}$.
It remains to show that $\scr D \wequi \scr C$.
We know that \[ [\1_{1^r} \otimes \Gmp{i}, \Gmp{j} \otimes H\Z/(p,\tau)] \wequi H^{j-i}_\et(k, \F_p), \] where $k$ is the field extension of $\F$ corresponding to $\1_{1^r}$.
Since $D(Gal(\F), \F_p) \wequi D(\F_p\fpsr{t})_{t-tors}$, the result follows.
\end{proof}

Let $n$ be maximal such that $\mu_{p^n} \subset k$.
We have $1 \le n < \infty$ ($2 \le n$ if $p=2$).
In Example \ref{exm:AT-e} we have constructed a functor \[ e: \GalSp(\FF_1[\mu_{p^n}])^{\comp AT}_{p,\eta} \to \SH(k)_{p,\eta}^\comp. \]
\begin{prop}\label{prop:tate-nonorient-e-summary}
All objects are implicitly $p$-completed in the following:
\begin{enumerate}
\item We have $e(\1_{1^r}) \wequi \Sigma^\infty_+ Spec(k[\mu_{p^r}])$ as well as $e(\Gm) \wequi \Gm$.
\item We have $e(\MGL) \wequi \MGL$ and $e(H\Z) \wequi H\Z$.
\item Choose a Tate-orientation $k^t/k$, i.e. $k^t$ is obtained by adding all $p^r$-th roots of unity for all $r$.
  We get $Gal(k^t/k) \wequi \Z_p$ and $D(Gal(k^t/k), \F_p) \wequi D(\F_p\fpsr{t})_{t-tors}$.
  Consider the functor \[ F: D(Gal(k^t/k)) \wequi D(\F_p\fpsr{t})_{t-tors} \xrightarrow{y} \PSh_\Sigma(\scr C)_{\le 0} \wequi \Mod(H\Z/(p,\tau))^\heart, \] where $y$ denotes the restricted Yoneda functor.
  We have a canonical equivalence \[ F\left(\bigoplus_i \Sigma^i K_*^M(k^t)/p\right) \wequi e_*(H\Z/(p,\tau)). \]
\end{enumerate}
\end{prop}
\begin{proof}
(1) This is true by construction.

(2) Recall the field $\F$ from the proof of Proposition \ref{prop:tate-nonorient-summary}.
Consider the commutative diagram
\begin{equation*}
\begin{CD}
\SH(\F)_{p,\eta}^\comp @<e'<< \GalSp(\F_1[\mu_{p^n}])_{p,\eta}^{\comp AT} @>e>> \SH(k)_{p,\eta}^\comp \\
@Vp^*VV                                @Vr^*VV                       @Vq^*VV \\
\SH(\bar\F)_{p,\eta}^\comp @<e'_\infty<< \GalSp(\F_1[\mu_{p^\infty}])_{p,\eta}^\comp \wequi \SH(\CC)_{p,\eta}^{\comp \cell} @>e_\infty>> \SH(k^t)_{p,\eta}^\comp. \\
\end{CD}
\end{equation*}
By Proposition \ref{prop:artin-tate-recons}, the functors $e',e'_\infty$ are fully faithful and identify $\MGL$ on both sides.
The functor $p^*$ identifies $\MGL$ on both sides; hence the same must be true for $r^*$.
Consequently $q^* e(\MGL) \wequi \MGL$ by Proposition \ref{prop:tate-orient-summary}(1).
One deduces that $e(\MGL) \wequi \MGL$ as in the proof of Proposition \ref{prop:tate-orient-summary}(1).

(3) As an object of $\PSh_\Sigma(\scr C)_{\le 0}$, $e_*(H\Z/(p,\tau))$ is characterized by having value at $\Sigma^i \F_p\fpsr{t}/t^{p^r}$ given by $K_{-i}^M(k[\mu_{p^{n+r}}])/p \wequi H^{-i}_\et(k[\mu_{p^{n+r}}], \F_p)$.
In other words using the adjunction \[ f^*: D(Gal(k^t/k), \F_p) \adj D(Gal(k), \F_p): f_*, \] $e_*(H\Z/(p,\tau))$ corresponds to the image of $f_*(\F_p)$.
Since $\F_p\fpsr{Gal(k^t/k)}\wequi \F_p\fpsr{t}$ has global dimension $1$, every object in $D(Gal(k^t/k), \F_p)$ splits into the sum of its homologies (see e.g. \cite[\S1.6]{MR2355772}).
For $f_*(\F_p)$ this is precisely $H^*_\et(k^t, \F_p) \wequi K_*^M(k^t)/p$ (with the action of $Gal(k^t/k)$).
This is the desired result.
\end{proof}

We call an object in the heart of a symmetric monoidal $t$-category \emph{flat} if tensoring with it is $t$-exact.
Perhaps the key result of this section is the following.
\begin{lem} \label{lemm:flatness}
The object $e_*(H\Z/(p,\tau)) \in \Mod(H\Z/(p,\tau))^\heart$ is flat.
\end{lem}
The proof is postponed to the next subsection.

We exploit flatness as follows.
\begin{lem} \label{lemm:flatness-pi*}
Let $E \in \GalSp(\FF_1[\mu_{p^n}])^{AT}_{\ge 0}$ such that $E \otimes H\Z/(p,\tau) \in \Mod(H\Z/(p,\tau))^\heart$ is flat.
Let $X \in \GalSp(\FF_1[\mu_{p^n}])^{AT}$ such that every map $\Sigma^i \Gmp{j} \otimes \1_{1^r} \to X$ is annihilated by some power of both $p$ and $\eta$.
Then \[ \pi_i(X \otimes E)_* \wequi \pi_i(X)_* \otimes^\heart \pi_0(E)_* \in \GalSp(\FF_1[\mu_{p^n}])^{AT\heart}. \]
\end{lem}
\begin{proof}
Let us call objects such as $X$ ``torsion''.
By the description of the $t$-structure from Proposition \ref{prop:tate-nonorient-summary}(2), an object is torsion if and only if its homotopy presheaves are.
Since $E$ is connective, if $X$ is connective then also $E \otimes X$ is connective.
It will suffice to show that in addition, if $X$ is coconnective, then $E \otimes X$ is also coconnective.
(Indeed then tensoring the cofiber sequences $X_{\ge i+1} \to X \to X_{\le i}$ and $X_{=i} \to X_{\le i} \to X_{\le i-1}$ with $E$ and considering the long exact sequences shows that $\pi_i(X \otimes E) \wequi \pi_i(X_{\le i} \otimes E) \wequi \pi_i(X_{=i} \otimes E)$, and the latter is equivalent to $\pi_i(X) \otimes^\heart \pi_0(E)$ since $\Sigma^{-i} X_{=i}$ and $E$ are both connective.)
Since $\GalSp(\FF_1[\mu_{p^n}])^{AT}_{\le 0}$ is closed under filtered colimits (Proposition \ref{prop:tate-nonorient-summary}(2)) and extensions, it is enough to treat the case $X \in \GalSp(\FF_1[\mu_{p^n}])^{AT\heart}$.
We have $X = \colim_{n,m} X[\eta^n, p^m]$ (where the right hand term denotes the intersection of the kernels of multiplication by $\eta^n$ and $p^m$), so using closure under filtered colimits again it suffices to treat $X[\eta^n,p^m]$.
In the short exact sequence $0 \to X[\eta^{n-1},p^m] \to X[\eta^n, p^m] \to Q \to 0,$ $\eta$ acts by zero on $Q$.
Using closure under extensions, in this way we can reduce to $X$ on which both $p$ and $\eta$ act by zero.
In other words, $X$ is a module over $H\Z/(p,\tau)$ (Proposition \ref{prop:tate-nonorient-summary}(3)).
Now $X \otimes E \wequi X \otimes_{H\Z/(p,\tau)} (H\Z/(p, \tau) \otimes E)$.
This lies in $\Mod(H\Z/(p,\tau))^\heart$ by our flatness assumption.
\end{proof}

\begin{thm} \label{thm:motivic-stems-hard}
Let $n, k, p$ as above.
For $i, j > 0$ we have \[ \pi_i(e_*(\1_k/(p^i,\eta^j)))_* \wequi \pi_i(\1_{1^n}/(p^i,\eta^j))_* \otimes \pi_0(e_*((\1_k)_{p,\eta}^\comp)). \]
\end{thm}
\begin{proof}
Combining Proposition \ref{prop:tate-nonorient-summary}(1) (compact-rigid generation), Proposition \ref{prop:tate-nonorient-e-summary}(2) ($e(H\Z)_p^\comp \wequi H\Z_p^\comp$) and Lemma \ref{lemm:projection-formula} (projection formula), we see that \[ e_*((\1_k)_{p,\eta}^\comp) \otimes H\Z/(p,\tau) \wequi e_*(H\Z/p). \]
This is flat by Lemma \ref{lemm:flatness}, and so by Lemma \ref{lemm:flatness-pi*} applied with $X = \1/(\eta^i,p^j)$ we obtain the desired result.
\end{proof}

\subsection{Flatness of Milnor $K$-theory} \label{subsec:flatness}
Let $k$ be a field of characteristic $\ne 2,p$.
Let $1\le n < \infty$ maximal with $\mu_{p^n} \subset k$.
Write $k^t/k$ for the field obtained by adding $\mu_{p^r}$ for all $r$.
Note $Gal(k^t/k) \wequi \Z_p$.
Using Remark \ref{rmk:Zp-action-Fpt-module}, we can view $K_i^M(k^t)/p$ as an $\FF_p\fpsr{t}$-module.

\begin{thm} \label{thm:KM-flatness}
For $m \ge 0$ there exists a free $\FF_p\fpsr{t}/t^{p^m}$-module $P_m$ together with a map $P_m \to M := K_i^M(k^t)/p$, such that the following hold.
Set $Q_m = \bigoplus_{m' \le m} P_{m'}$.
Then $Q_m \to M$ is split injective for every $m$.
Set $R_m = M/Q_m$.
Then $\colim_M R_m$ is isomorphic to a sum of modules of the form $\FF_p((t))/\FF_p\fpsr{t}$.
\end{thm}

We first show how the flatness lemma from the previous subsection follows from this.
\begin{proof}[Proof of Lemma \ref{lemm:flatness}.]
The functor \[ D(\F_p\fpsr{t}) \to \Mod(H\Z/(p,\tau))^\heart \] from Proposition \ref{prop:tate-nonorient-e-summary}(3) is additive and preserves filtered colimits.
It hence sends filtered colimits of split exact sequences to exact sequences.
Theorem \ref{thm:KM-flatness} together with Proposition \ref{prop:tate-nonorient-e-summary}(3) thus translates into an exact sequence \[ 0 \to A \to e_*(H\Z/(p,\tau)) \to B \to 0 \in \Mod(H\Z/(p,\tau))^\heart. \]
Here $A$ is a sum of twists of representables (see Proposition \ref{prop:tate-nonorient-summary}(4)), and $B$ is a sum of twists of the object corresponding to $\FF_p((t))/\FF_p\fpsr{t}$.
This can be written as a filtered colimit of $\FF_p\fpsr{t}/t^{p^i}$, and hence corresponds to a filtered colimit of representables.
Flat objects are always stable under extensions (since the non-negative and non-positive parts of $t$-structures are), and in our case also under filtered colimits (same reason) and twists (twisting being a $t$-exact automorphism).
It thus remains to show that representable objects are flat.
Let $E \in \Mod(H\Z/(p,\tau))^\heart$ be representable.
Then $E$ is strongly dualizable and in fact self-dual.
Since the $t$-structure is compatible with the symmetric monoidal structure, it is clear that $E \otimes \Mod(H\Z/(p,\tau))_{\ge 0} \subset \Mod(H\Z/(p,\tau))_{\ge 0}$.
Now let $X \in \Mod(H\Z/(p,\tau))_{<0}$.
To show that $E \otimes X \in \Mod(H\Z/(p,\tau))_{<0}$ we must show that for every $Y \in \Mod(H\Z/(p,\tau))_{\ge 0}$ we have $[Y, E \otimes X] = 0$.
But this is the same as $[Y \otimes DE, X]$, which vanishes since $DE \in \Mod(H\Z/(p,\tau))_{\ge 0}$ and hence $Y \otimes DE \in \Mod(H\Z/(p,\tau))_{\ge 0}$.
\end{proof}

We now prove Theorem \ref{thm:KM-flatness}.
We begin by classifying torsion $\F_p\fpsr{t}$-modules.

\begin{lem} \label{lem:Pn-nonsense}
Say that an $\F_p\fpsr{t}$-module $M$ satisfies property $P_n$ if whenever $x \in M$ and $t^n x=0$, then $x=ty$ for some $y \in M$.\NB{I.e. $M[t^{n+1}] \xrightarrow{t} M[t^n]$ is surjective.}
\begin{enumerate}
\item Let $M$ be a $\F_p\fpsr{t}/t^{n+1}$-module.
  Then $M$ is free if and only if it satisfies $P_n$.
\item Suppose that $M$ satisfies $P_n$.
  Then so do $M[t^{n+1}]$ and $M/t^{n+1}$.
\item Suppose that $\alpha: M \to N$ is a morphism of free $\F_p\fpsr{t}/t^{n+1}$-modules.
  Then there are splittings $M = M_1 \oplus M_2$ and $N = N_1 \oplus N_2$ such that $\alpha$ maps $M_1$ isomorphically onto $N_1$, and $\alpha(M_2) \subset tN$.
\item A morphism of free $\F_p\fpsr{t}/t^{n+1}$-modules is an isomorphism if and only if it is so modulo $t$, if and only if it is on $t$-torsion.
\item Suppose that $M$ satisfies $P_n$.
  Then there is a decomposition $M \wequi F \oplus M'$, where $F$ is a free $\F_p\fpsr{t}/t^{n+1}$-module and $M'$ satisfies $P_{n+1}$.
\end{enumerate}
\end{lem}
\begin{proof}
(1) Necessity is clear.
To prove sufficiency, suppose that $M$ satisfies $P_n$.
Note $(*)$ that it follows that $M[t^i] = t^{n+1-i}M$, for $1 \le i \le n$.
(If $t^i x = 0$ with $i\le n$ then $t^nx=0$ so $x=ty$, and $t^{i+1}y=0$---now repeat.)
Write $M = tM \oplus G$, a decomposition of $\F_p$-vector spaces.
Then $G \to M \xrightarrow{t^n} M$ is injective by condition $P_n$.
Since $tM = M[t^n]$ by $(*)$, we deduce $(a)$ that $G \to M \xrightarrow{t^n} t^nM$ is an isomorphism onto $t^nM = M[t]$ (the latter again by $(*)$).
Also $(b)$ by construction $G \to M/t$ is an isomorphism.
Now let $\alpha: F \to M$ be the morphism from a free module with generators in bijection with a basis of $G$.
Then $\alpha/t$ is an isomorphism by $(b)$, and $\alpha[t]$ is an isomorphism by $(a)$.
This implies that $\alpha$ is an isomorphism.\footnote{View $\alpha$ as a map in the derived category of $\F_p\fpsr{t}/t^{n+1}$-modules.
Then $\alpha$ is an equivalence if and only if $\alpha[t^{-1}]$ and $\alpha\sslash t$ are equivalences, where $\alpha \sslash t$ means the cofiber of multiplication by $t$.
The first map is between zero objects, and the second map is between objects with only two homotopy groups, given by $\alpha[t]$ and $\alpha/t$.}

(2) If $x \in M[t^{n+1}]$ satisfies $t^nx=0$, then $x=ty$ for some $y \in M$, and necessarily $y \in M[t^{n+1}]$.
Hence the first claim is clear.
For the second, let $\bar x \in M/t^{n+1}$ with $t^n \bar x = 0$.
Let $\bar x$ be the image of $x$, so that $t^n x = t^{n+1} y$ for some $y \in M$.
Then $t^n(x-ty) = 0$ and so $x-ty = tz$, whence $x = t(y+z)$.
This proves the second claim.

(3) Consider a splitting of $\F_p$-vector spaces $M = tM \oplus G_1 \oplus G_2$, where $G_1 \to N/t$ is injective and $\alpha(G_2) \subset tN$ (just write $M = tM \oplus G$, set $G_2 = ker(G \to N/t)$, and let $G_1$ be a complement of $G_2$ in $G$).
Next construct a splitting $N = tN \oplus \alpha(G_1) \oplus G_3$ (note that $\alpha(G_1) \cap tN = 0$ by construction).
Let $M_1$ be the free module with basis in bijection with a basis of $G_1$, $M_2$ the same with $G_2$, $N_1$ the same with $\alpha(G_1)$, and $N_2$ the same with $G_3$.
It follows from (4) that $M_1 \oplus M_2 \to M$ is an isomorphism and similarly for $N$.
This proves (3).

(4) Let $M$ be an $\F_p\fpsr{t}/t^{n+1}$-module.
Define a filtration by $F_iM = M[t^i]$ and write $G_i M = F_{i+1}M/F_{i}M$ for the associated graded.
Note that $F_0M = 0$ and $F_{n+1}M=M$.
This filtration is natural in $M$.
It follows that a map $\alpha$ is an isomorphism if and only if $G_i(\alpha)$ is, for $i=0, \dots, n$.
Multiplication by $t$ defines natural maps $F_{i+1}M \to F_iM$ and hence $[\cdot t]: G_{i+1}M \to G_i M$.
If $M$ is free then these maps $[\cdot t]$ are isomorphisms for $i = 0, \dots, n-1$ (just because the constructions are compatible with direct sums and the claim is evident for $M=\F_p\fpsr{t}/t^{n+1}$, where $G_i M \wequi \F_p\{[t^{n-i}]\}$).
It follows that for maps between free modules, $G_i(\alpha)$ is an isomorphism for all $i$ if and only if it is an isomorphism for any specific $i$.
Since $G_0 M$ is the $t$-torsion and $G_nM \wequi M/t^n$ ($M$ being free), the result follows.

(5) Both $M[t^{n+1}]$ and $M/t^{n+1}$ are $\F_p\fpsr{t}/t^{n+1}$-modules.
They satisfy $P_n$ by (2) and hence are free $\F_p\fpsr{t}/t^{n+1}$-modules by (1).
Applying (3) to the canonical map $M[t^{n+1}] \to M/t^{n+1}$ we may write $M[t^{n+1}] = M_1 \oplus M_2$, $M/t^{n+1} = N_1 \oplus N_2$ with $M_1 \xrightarrow{\wequi} N_1$ and the image of $M_2$ divisible by $t$.
Put $F = M_1$.
Then the composite $M \to M/t^{n+1} \to N_1 \wequi F$ splits the inclusion $F=M_1 \hookrightarrow M[t^{n+1}] \hookrightarrow M$, and so we get a decomposition $M \wequi F \oplus M'$.
If $x \in M'[t^{n+1}]$ then $x \in M[t^{n+1}]$.
Write $x = x_1 + x_2$, with $x_i \in M_i$.
Since $x \in M'$ we must have $x_1 = 0$, i.e. $x \in M_2$.
Thus the image of $x$ in $M/t^{n+1}$ is divisible by $t$, whence $x$ itself is divisible by $t$ as needed.
\end{proof}

\begin{prop} \label{prop:Fpt-structure}
Let $M$ be an $\F_p\fpsr{t}$-module.
For $i \ge 1$ there exit free $\F_p\fpsr{t}/t^i$-modules $F_i$ together with maps $\iota_i: F_i \to M$ and $\rho_i: M \to F_i$ such that $\rho_i \circ \iota_j = \delta_{ij}$, and moreover in the exact sequence \[ 0 \to \bigoplus_i F_i \to M \to M' \to 0 \] the quotient $M'$ satisfies $P_i$ for every $i$.
In particular if $M$ is torsion then $M'$ is divisible and so isomorphic to a sum of $\F_p((t))/\F_p\fpsr{t}$.
\end{prop}
\begin{proof}
Since any module satisfies $P_0$, we may inductively apply Lemma \ref{lem:Pn-nonsense}(5) to find decompositions \[ M \wequi F_1 \oplus F_2 \oplus \dots \oplus F_n \oplus M'_n, \] where $M'_n$ satisfies $P_n$.
This yields the desired retractions, and it is clear that $\bigoplus_i F_i \to M$ is injective.
We get $M' = \colim_i M'_i$.
Since $M'_i$ satisfies $P_n$ for $i \ge n$, and modules satisfying $P_n$ are stable under filtered colimits (filtered colimits being exact), it follows that $M'$ satisfies $P_n$ for any $n$.

If $M$ is torsion then so is $M'$, and $M' = \bigcup_i M'[t^i] = \bigcup_i t M'[t^{i+1}] = tM'$ as needed (the middle equality holds since $M'$ satisfies $P_i$ for all $i$).
The last claim follows from Lemma \ref{lem:classify-divisible} below.
\end{proof}

\begin{lem} \label{lem:classify-divisible} \NB{reference?}
Let $M$ be a divisible, torsion $\F_p\fpsr{t}$-module.
Then $M$ is isomorphic to a sum of modules of the form $\F_p((t))/\F_p\fpsr{t}$.
\end{lem}
\begin{proof}
Let $S$ be the set consisting of independent families of submodules of $M$ isomorphic to $\F_p((t))/\F_p\fpsr{t}$.
In other words an element of $S$ is a pair $(I, \alpha)$ of a set $I$ and an injection $\alpha: F_I := \bigoplus_I \F_p((t))/\F_p\fpsr{t} \to M$.
Order $S$ in such a way that $(I, \alpha) \le (I', \alpha')$ if there exists an injection $I \hookrightarrow I'$ making the evident triangle commute.
(Note that $I \to I'$ is unique if it exists.)
It is then clear that increasing chains in $S$ have upper bounds.
Suppose given an injection $F_I \to M$.
Since $\F_p\fpsr{t}$ is a PID, divisible modules are injective\NB{ref?}, and hence $M \wequi F_i \oplus M'$.
Note that $M'$ is still a divisible torsion module.
We shall now prove that there exists an injection $\F_p((t))/\F_p\fpsr{t} \to M$, provided that $M \ne 0$.
By the above remarks, this will at the same time prove that $S$ is non-empty, and that for a maximal element $(I, \alpha) \in S$, $\alpha$ is an isomorphism.
We will thus be done by Zorn's lemma.

Thus let $M$ be a non-zero, divisible, torsion module.
If $0 \ne x \in M$, then there exists a minimal $n$ with $t^n x = 0$.
We must have $n>0$ since $x \ne 0$.
Thus $0 \ne x_0 := t^{n-1} x \in M[t]$.
Choose inductively elements $x_i$ with $t x_{i+1} =x_i$.
The subset $M_0 = \F_p\{x_i|i=0,1,\dots\} \subset M$ is in fact an $\F_p\fpsr{t}$-submodule, and by inspection it is isomorphic to $\F_p((t))/\F_p\fpsr{t}$.
\end{proof}

The assertion in Theorem \ref{thm:KM-flatness} is thus that, when decomposing $K_*^M(k^t)/p$ as in Proposition \ref{prop:Fpt-structure}, summands of the form $\F_p\fpsr{t}/t^i$ appear only when $i$ is a power of $p$.
This is the statement of the next result.
\begin{prop} \label{prop:torsion-orders}
Viewing $K_*^M(k^t)/p$ as an $\F_p\fpsr{t}$-module, any $t$-torsion element divisible by $t^{p^n}$ is divisible by $t^{p^{n+1}-1}$.
\end{prop}

\begin{proof}[Proof of Theorem \ref{thm:KM-flatness}.]
Immediate from Propositions \ref{prop:Fpt-structure} and \ref{prop:torsion-orders}.
\end{proof}

The proof of Proposition \ref{prop:torsion-orders} is substantially simpler if $k$ is odd.
We give it below.
The case $p=2$ is relegated to the next section.
\begin{proof}[Proof of Proposition \ref{prop:torsion-orders} for $p$ odd.]
Let $M_n$ denote the $\F_p\fpsr{u}$-module obtained from $K_*^M(k^t)/p$ by setting $u=t^{p^n}$.
We first claim that if $x \in M_n[u]$ is divisible by $u$, then $x$ is divisible by $u^{p-1}$.\NB{If $p=2$, this statement is vacuous...}

Assuming the claim, here is how to deduce the result.
Let $x \in M_0[t]$ be divisible by $t^{p^n+a}$ for some $0 \le a \le p^n-1$, i.e. $x = t^{p^n+a}y$.
Set $x' = t^{p^n}y = uy$, where $u=t^{p^n}$.
Then $ux' = t^{p^n+p^n}y = t^{p^n-a}x = 0$, since $a < p^n$.
The claim implies that $x' = u^{p-1}y'$, whence \[ x = t^a x' = t^{p^{n+1} - p^n + a}y'. \]
We have $p^{n+1}-p^n = p^n + (p^{n+1}-2p^n)$ with $p^{n+1}-2p^n > 0$ (since $p$ is odd), and so $x = t^{p^n + a'}y'$ for some $a'>a$.
Repeating this construction sufficiently many times, we may assume that $a=p^n-1$ and hence $p^{n+1} - p^n + a = p^{n+1}-1$ as desired.

It thus remains to establish the claim.
Recall that $\F_p\fpsr{t} = \F_p\fpsr{G}$, where $G \wequi \Z_p$ is topologically generated by $\sigma'$ and $t=\sigma'-1$.
Let $\sigma = \sigma'^{p^n}$ so that $u = \sigma - 1$.
Let $k_1$ be the fixed field of $\sigma$ and $k_2$ the fixed field of $\sigma^p$.
Observe that $k_2/k_1$ is a cyclic extension of degree $p$ with generator $\sigma$; in fact $k_2 = k_1(\zeta_2)$ where $\zeta_2^p = \zeta_1$, and $\zeta_1 \in k_1$ is a primitive $p^d$-th root of unity for $d$ maximal.
Let $M \in D(\F_p\fpsr{u})$ denote the object with $H^*M = K_*^M(k^t)/p$, i.e., $M$ is obtained by forming an injective resolution of $\Z/p$ in the category of $Gal(k_1)$-modules, taking $Gal(k^t)$-fixed points (thus comupting $H^*_\et(k^t, \Z/p) \wequi K_*^M(k^t)/p$), and viewing the result as having a $Gal(k^t/k_1)$-action.
Put $M(k_1) = \fib(u: M \to M)$ and $M(k_2) = \fib(u^p: M \to M)$, so that $H^*M(k_i) = K_*^M(k_i)/p$ (indeed taking the fiber of $u$ is taking fixed points for the topologically cyclic group $Gal(k^t/k_1)$, and so this computes $H^*_\et(k_1, \Z/p)$; similarly for $k_2$).
We thus have an exact sequence \cite[Theorem 3.6]{haesemeyer2019norm} \[ H^nM(k_2) \xrightarrow{N} H^n M(k_1) \xrightarrow{\partial} H^{n+1} M(k_1) \xrightarrow{R} H^{n+1} M(k_2). \]
Since $\F_p\fpsr{u}$ has global dimension $1$ (being regular local of dimension $1$), we get a splitting $M \wequi \bigoplus_i \Sigma^{-i} N_i$ (see e.g. \cite[\S1.6]{MR2355772}) where $N_i = K_i^M(k^t)/p$.
Note that now $H^0 N_i(k_1) = \ker(u: M_i \to M_i)$ and $H^1 N_i(k_1) = \coker(u: M_i \to M_i)$ and $H^* N_i(k_1) = 0$ else, and similarly for $N_i(k_2)$.
Passing to the summand of the exact sequence corresponding to $N_n$ we thus obtain \[ K_n^M(k^t)/p[u^p] \xrightarrow{N} K_n^M(k^t)/p[u] \xrightarrow{\partial} (K_n^M(k^t)/p)/u \xrightarrow{R} (K_n^M(k^t)/p)/u^p. \]
Chasing through the definitions\NB{...}, the map $N$ is multiplication by $u^{p-1}$ and $\partial$ is the canonical map from the kernel to the cokernel.
Exactness at the second spot is thus precisely the desired statement.
\end{proof}

\subsection{A version of Hilbert 90 for $p=2$} \label{subsec:H90}
\begin{prop}[Merkurjev] \label{prop:H90-KM} \NB{HW label this argument as due to Merkurjev, but don't give a specific reference.}
  Let $k$ be a field of characteristic not $\ell$ and $E/k$ a cyclic galois extension of degree $\ell^m$ with $\sigma$ a generator of the galois group.
  Then higher Hilbert 90 implies that the sequence
  \[ K_n^M(E) \xrightarrow{1-\sigma} K_n^M(E) \to K_n^M(k) \]
  is exact.
\end{prop}
In \cite[Theorem 3.2]{haesemeyer2019norm}, the above result is stated and proved for $m=1$.
The same proof works for larger $m$, as we record now.
\begin{proof}
  After inverting $\ell$ the sequence is exact, so it suffices to handle the $\ell$-local case.
  Let $G = \langle \sigma \rangle$, then the Norm residue isomorphism lets us rewrite things in terms of $G$-modules. Our goal is now to prove that
  \[ H_{\et}^n(k, \Z_{(\ell)}[G](n)) \xrightarrow{1-\sigma} H_{\et}^n(k, \Z_{(\ell)}[G](n)) \to H_{\et}^n(k, \Z_{(\ell)}(n)) \]
  is exact.

Consider the following two short exact sequences of $G$-modules:
  \[ 0 \to I \to \Z[G] \to \Z \to 0, \text{ and} \]
  \[ 0 \to \Z \to \Z[G] \xrightarrow{\cdot(1-\sigma)} I \to 0. \]
  On cohomology this yields exact sequences
  \[ H_{\et}^n(k, I_{(\ell)}(n)) \to H_{\et}^n(k, \Z_{(\ell)}[G](n)) \to  H_{\et}^n(k, \Z_{(\ell)}(n)) \]
  \[ H_{\et}^n(k, \Z_{(\ell)}[G](n)) \to H_{\et}^n(k, I_{(\ell)}(n)) \to H_{\et}^{n+1}(k, \Z_{(\ell)}(n)) \stackrel{H90}{=} 0 \]
  Combining these sequences and noting that the map $\Z[G] \to I \to \Z[G]$ is $1 - \sigma$ finishes the proof.
\end{proof}

Using this proposition we can build a $C_4$ variant of the mod 2 version of Hilbert 90 for $K^M$.
\begin{lem} \label{lem:H90-cons}
  Let $k_2/k_1/k$ be a $C_4$ Galois extension of fields of characteristic not 2.
  Let $\sigma$ be a generator of $C_4$.
  Then, the sequence
  \[ K_n^M(k_1)/2 \oplus K^M_n(k_2)/2 \xrightarrow{(i, 1-\sigma)} K^M_n(k_2)/2 \to K^M_n(k)/2 \]
  is exact.  
\end{lem}

\begin{proof}
  As usual we start by using the Norm-residue isomorphism theorem to translate this into a statement about étale cohomology.
  Over $\F_2$ the group $C_4$ is unipotent and there are exactly four non-trivial indecomposable $C_4$ representations which we will denote $F_1, F_2, F_3$ and $F_4$ according to their rank. Between these representations we have inclusion and restriction maps which we will denote by $i$ and $r$.
In fact $\F_2[C_4] = \F_2[\sigma]/(\sigma^4-1) = \F_2[\epsilon]/\epsilon^4$ where $\epsilon := 1-\sigma$; then $F_i = \F_2[\epsilon]/\epsilon^i$ and we can denote maps as being induced by appropriate powers of $\epsilon$.
  We will argue by considering the following diagram of $C_4$-modules (where the unlabelled maps are the canonical reductions)
  \begin{center}    
    \begin{tikzcd}
      F_2 \ar[dr, "\cdot \epsilon"] & & F_2 \ar[r] & F_1 \ar[d, "\cdot \epsilon"] \\
      F_4 \ar[r] & F_3 \ar[ur] \ar[r, "\cdot \epsilon"] & F_4 \ar[r] & F_2 \ar[r] & F_1.
    \end{tikzcd}
  \end{center}
  
  Suppose $x$ is a class in $H_{\et}^n(k, F_4(n))$ (viewed as being in the copy of $F_4$ in the middle)
  which maps to zero in $H_{\et}^n(k, F_1(n))$.
  What we wish to show is that there exists a pair of classes $(a,b)$ in $H_{\et}^n(k, F_2(n))$ and $H_{\et}^n(k, F_4(n))$ which hit $x$ under $(i_2, 1-\sigma)=(\epsilon^2, \epsilon)$.
  Using the condition that $x$ maps to zero we can pick a lifting of $x$ to $w \in H_{\et}^n(k, F_3(n))$.
  Let $v$ be the image of $w$ in the top $H_{\et}^n(k, F_2(n))$ (that is, the reduction of $w$).  
  In order to proceed further we will need an integral lift of part of the diagram above:

  \begin{center}    
    \begin{tikzcd}
      \Z_{(2)}[C_2] \ar[drr, "i_2"] & & \Z_{(2)}[C_2] \ar[r, "t_1"] & \Z_{(2)} \ar[d, "i_1"] \\
      \Z_{(2)}[C_4] \ar[rr, "1 - \sigma"] & & \Z_{(2)}[C_4] \ar[r, "t_2"] & \Z_{(2)}[C_2] \ar[r, "t_1"] & \Z_{(2)}
    \end{tikzcd}
  \end{center}
  
  This diagram lies over the previous one and as a consequence of higher Hilbert 90 (see \cite[Definition 1.5]{haesemeyer2019norm}), the vertical quotient maps between these two diagrams are surjective on $H_{\et}^n(k, \ph(n))$.    
  Pick lifts $\wt{x}$ and $\wt{v}$ of $x$ and $v$ to the integral diagram above.
  Then, since $t_2 \wt{x}  - i_1 t_1 \wt{v}$ is zero mod 2 we may pick a $\wt{y}$ such that $2 \wt{y} = t_2 \wt{x}  - i_1 t_1 \wt{v}.$
  Consider the class $\wt{x} - i_2 \wt{v} - i_2 \wt{y}$.
  We have
  \begin{align*}
    t_1t_2(\wt{x} - i_2 \wt{v} - i_2\wt{y})
    &= t_1t_2\wt{x}  - t_1t_2i_2\wt{v} - t_1t_2i_2\wt{y}
    = t_1t_2\wt{x}  - t_1 2\wt{v} - t_1 2\wt{y} \\
    &= t_1t_2\wt{x} - t_1 2\wt{v} - t_1(t_2\wt{x} - i_1t_1\wt{v}) 
    = - t_1 2\wt{v} + t_1i_1t_1\wt{v} \\
    &= - t_1 2\wt{v} + 2t_1  \wt{v} = 0
  \end{align*}
  Using the exactness statement from \Cref{prop:H90-KM} there exists a class $\wt{z} \in H^n_\et(\Z_{(2)}[C_4])$ such that
  \[ (1-\sigma)\wt{z} = \wt{x} - i_2\wt{v}  - i_2\wt{y}. \]
  Reducing back down we learn that the pair
  $(v + y, z)$ maps to $x$  providing the desired lift.
\end{proof}

\begin{proof}[Proof of Proposition \ref{prop:torsion-orders} for $p=2$.]
Let $M_n$ denote the $\F_p\fpsr{u}$-module obtained from $K_*^M(k^t)/p$ by setting $u=t^{p^n}$.
Note that in contrast to the situation when $p$ is odd, the case $n=0$ of the assertion is trivial.
Our initial claim is thus the case $n=1$: if $x \in M_n[u]$ is divisible by $u^2$, then $x$ is divisible by $u^{p^2-1}=u^3$.
To prove this we re-interpret the sequence of Lemma \ref{lem:H90-cons} as in the odd case, to deduce that we have an exact sequence \[ M_n[u^2] \oplus M_n[u^4] \xrightarrow{i,\cdot u} M_n[u^4] \xrightarrow{\cdot u^3} M_n[u]. \]
Thus any $u^3$-torsion element in $M_n[u^4]$ is of the form $a + u b$, where $u^2 a = 0$.
Hence if $x = u^2 x' \in M_n[u]$ then $u^3x' = 0$, so $x' = a + u b$, and thus $x = u^2 x' = u^2(a + u b) = u^3b$, as desired.

To deduce the general case from this we argue as follows.
Let $x \in M_0[t]$ be divisible by $t^{2^n}$, say $x = t^{2^n}y$.
Set $u = t^{2^{n-1}}$ so that $ux = 0$ and $x = u^2y$.
The claim thus yields $x = u^3y'$.
Set $z = ty'$, $x' = u^2z$.
Then $ux' = u^3z = tu^3y' = tx = 0$, so $x' = u^3 z'$ (by the claim again).
We obtain \[ x = u^3 y' = u^2 t^{2^{n-1}-1} z = t^{2^{n-1}-1}x' = t^{2^{n-1}-1} u^3 z' = t^{2^{n-1}-1 + 2^n + 2^{n-1}} z' = t^{2^{n+1}-1} z', \] as needed.
\end{proof}

\appendix

\bibliographystyle{alpha}
\bibliography{bibliography}

\end{document}

%% file: Preamble.tex
\usepackage{verbatim}
\usepackage[textsize=scriptsize]{todonotes}
\usepackage{tikz-cd}
\usepackage{etoolbox}
\usepackage{etex}
\usepackage[pdfusetitle,unicode,hidelinks]{hyperref}
\usepackage{cleveref}
\usepackage[T1]{fontenc}

\hypersetup{
   colorlinks,
   linkcolor={black},
   citecolor={black},
   urlcolor={black}
}

\usepackage{chemarr}
\usepackage{amssymb}
\usepackage{amsmath}
\usepackage{comment}
\usepackage{cleveref}
\usepackage{mathtools}
\usepackage{rotating}
\usepackage{wrapfig}
\usepackage{outlines}
\usepackage{graphicx}
\usepackage{scalerel}

\numberwithin{equation}{section}

\theoremstyle{definition}
\newtheorem{dfn}[equation]{Definition}

\newtheorem{rmk}[equation]{Remark}

\newtheorem{cnstr}[equation]{Construction}

\newtheorem{exm}[equation]{Example}

\newtheorem{wrn}[equation]{Warning}

\newtheorem*{dfn*}{Definition}
\newtheorem*{axm*}{Axiom}
\newtheorem*{ntn*}{Notation}
\newtheorem*{exm*}{Example}
\newtheorem*{exr*}{Exercise}
\newtheorem*{int*}{Intuition}
\newtheorem*{qst*}{Question}
\newtheorem*{rmk*}{Remark}

\theoremstyle{plain}

\newtheorem{thm}[equation]{Theorem}
\newtheorem{prop}[equation]{Proposition}
\newtheorem{lem}[equation]{Lemma}

\newtheorem{cor}[equation]{Corollary}

\newtheorem*{thm*}{Theorem}
\newtheorem*{prop*}{Proposition}
\newtheorem*{cor*}{Corollary}
\newtheorem*{lem*}{Lemma}
\newtheorem*{cnj*}{Conjecture}

\let\oldwidetilde\widetilde
\protected\def\widetilde{\oldwidetilde}

\DeclareMathOperator{\Aut}{\mathrm{Aut}}

\DeclareMathOperator*{\colim}{\mathrm{colim}}
\DeclareMathOperator{\Map}{\mathrm{Map}}
\DeclareMathOperator{\map}{\mathrm{map}}
\DeclareMathOperator{\imap}{\underline{\mathrm{map}}}

\DeclareMathOperator{\fib}{\mathrm{fib}}

\DeclareMathOperator{\Hom}{\mathrm{Hom}} 
\DeclareMathOperator{\End}{\mathrm{End}}

\DeclareMathOperator{\coker}{\mathrm{coker}}

\newcommand{\Ss}{\mathbb{S}}
\newcommand{\FF}{\mathbb{F}}
\newcommand{\F}{\mathbb{F}}
\newcommand{\QQ}{\mathbb{Q}}
\newcommand{\NN}{\mathbb{N}}
\newcommand{\A}{\mathbb{A}}

\DeclareMathOperator{\MU}{\mathrm{MU}}

\DeclareMathOperator{\Spec}{\text{Spec}}

\DeclareMathOperator{\Fin}{\mathrm{Fin}}

\DeclareMathOperator{\Fun}{\mathrm{Fun}}
\DeclareMathOperator{\Alg}{\mathrm{Alg}}

\DeclareMathOperator{\GL}{\mathrm{GL}}

\DeclareMathOperator{\cof}{\mathrm{cof}}

\newcommand{\wt}{\widetilde}

\newcommand{\Tot}{\mathrm{Tot}}

\newcommand{\CC}{\mathbb{C}}

\newcommand{\Z}{\mathbb{Z}}

\usepackage{tikz}
\usetikzlibrary{matrix,arrows,decorations}
\usepackage{tikz-cd}

\usepackage{adjustbox}

